\documentclass[a4paper,11pt,reqno]{amsart}
\usepackage[latin1]{inputenc}
\usepackage[english]{babel}
\usepackage{amsmath, amsthm, amssymb, amsopn, amsfonts, amstext, enumerate, color, xcolor, mathtools, mathrsfs, graphicx, chngcntr, setspace, array, enumitem}
\usepackage[scr]{rsfso}
\usepackage[final]{showkeys}
\numberwithin{equation}{section}
\usepackage[margin=0.8in]{geometry}
\usepackage{hyperref}

\hypersetup{
	colorlinks=true,
	linkcolor=blue,
	citecolor=red,
	urlcolor=black,
}

\newcommand{\A}{\mathcal{A}}
\newcommand{\B}{\mathbf{B}}
\newcommand{\BB}{\mathfrak{B}}
\newcommand{\C}{\mathbf{C}}
\newcommand{\CC}{\mathscr{C}}
\newcommand{\D}{\mathcal{D}}
\newcommand{\E}{\mathscr{E}}
\newcommand{\FF}{\mathcal{F}}

\newcommand{\HH}{\mathcal{H}}

\newcommand{\K}{\mathbf{K}}
\renewcommand{\L}{\mathscr{L}}
\newcommand{\N}{\mathbb{N}}
\newcommand{\NN}{\mathscr{N}}
\newcommand{\NNN}{\mathcal{N}}

\newcommand{\R}{\mathbb{R}}

\newcommand{\T}{\mathcal{T}}
\newcommand{\V}{\mathscr{V}}
\newcommand{\X}{\mathcal{X}}
\newcommand{\Z}{\mathbb{Z}}

\newcommand{\loc}{{\rm loc}}

\newcommand{\etaini}{\eta^{(1)}}
\newcommand{\etadis}{\eta^{(2)}}
\newcommand{\PV}{{\mbox{\normalfont P.V.}}}
\newcommand{\Adot}{{\dot{\A}}}

\theoremstyle{plain}
\newtheorem{theorem}{Theorem}[section]
\newtheorem{proposition}[theorem]{Proposition}
\newtheorem{lemma}[theorem]{Lemma}

\newtheorem{definition}[theorem]{Definition}

\theoremstyle{definition}
\newtheorem{remark}[theorem]{Remark}

\renewcommand{\le}{\leqslant}

\renewcommand{\ge}{\geqslant}

\makeatletter
\newdimen\rh@wd
\newdimen\rh@hta
\newdimen\rh@htb
\newbox\rh@box
\def\rh@measure#1{\setbox\rh@box=\hbox{$#1$}\rh@wd=\wd\rh@box \rh@hta=\ht\rh@box}

\def\widecheck#1{\rh@measure{#1}%
  \setbox\rh@box=\hbox{$\widehat{\vrule height \rh@hta width\z@ \kern\rh@wd}$}%
  \rh@htb=\ht\rh@box \advance\rh@htb\rh@hta \advance\rh@htb\p@
  \ooalign{$\vrule height \ht\rh@box width\z@ #1$\cr
           \raise\rh@htb\hbox{\scalebox{1}[-1]{\box\rh@box}}\cr}}
\makeatother

\long\def\/*#1*/{}

\title[Long-time asymptotics for evolutionary crystal dislocation models]{Long-time asymptotics for evolutionary\\ crystal dislocation models}

\author{Matteo Cozzi}
\author{Juan D\'avila}
\author{Manuel del Pino}

\address{
\vspace{-\baselineskip}
\newline
\textit{Matteo Cozzi}
\newline
University of Bath, Department of Mathematical Sciences, Claverton Down, Bath BA2 7AY, UK
\newline
\textit{E-mail address}: \textit{\tt m.cozzi@bath.ac.uk}
}

\address{
\vspace{-\baselineskip}
\newline
\textit{Juan D\'avila}
\newline
Universidad de Antioquia, Instituto de Matem\'aticas, Calle 67, No. 53-108, Medell\'in, Colombia
\newline
Universidad de Chile, Departamento de Ingenier\'ia Matematica-CMM, Santiago 837-0456, Chile
\newline
\textit{E-mail address}: \textit{\tt jdavila@dim.uchile.cl}
}

\address{
\vspace{-\baselineskip}
\newline
\textit{Manuel del Pino}
\newline
University of Bath, Department of Mathematical Sciences, Claverton Down, Bath BA2 7AY, UK
\newline
Universidad de Chile, Departamento de Ingenier\'ia Matematica-CMM, Santiago 837-0456, Chile
\newline
\textit{E-mail address}: \textit{\tt m.delpino@bath.ac.uk}
}

\thanks{The first author has been supported by a Royal Society Newton International Fellowship~(UK), the second author by grants Fondecyt~1130360 and PAI~AFB-170001~(Chile), and the third author by a Royal Society Research Professorship~(UK) and grant PAI~AFB-170001~(Chile). Part of this work has been carried out while the first author was visiting the Universidad de Chile, which he thanks for the warm hospitality}



\begin{document}

\begin{abstract}
We consider a family of evolution equations that generalize the Peierls-Nabarro model for crystal dislocations. They can be seen as semilinear parabolic reaction-diffusion equations in which the diffusion is regulated by a fractional Laplace operator of order~$2 s \in (0, 2)$ acting in one space dimension and the reaction is determined by a~$1$-periodic multi-well potential. We construct solutions of these equations that represent the typical propagation of~$N \ge 2$ equally oriented dislocations of size~$1$. For large times, the dislocations occur around points that evolve according to a repulsive dynamical system. When~$s \in (1/2, 1)$, these solutions are shown to be asymptotically stable with respect to odd perturbations.
\end{abstract}

\maketitle

\section{Introduction and main results}

\noindent
In the '40s, Peierls and Nabarro developed a microscopic theory to describe plastic deformations in crystalline materials. Starting from discrete considerations, they devised a continuum theory for the presence of edge dislocations in atomic structures. The resulting model characterizes equilibrium configurations as solutions of a stationary nonlinear integro-differential equation. We refer the reader to~\cite{P40,N47} for the original papers and to~\cite{L05} for a more recent survey on the topic.

In recent years, a few time-dependent models have been proposed to study the dynamics of crystal dislocations and of related phenomena. Following~\cite{MBW98,EHM09,GM12,DPV15,DFV14}, we consider here the~$1$-dimensional evolution equation
\begin{equation} \label{maineq}
\partial_t u + (-\Delta)^s u + W'(u) = 0 \quad \mbox{in } \R \times (0, +\infty).
\end{equation}
Here,~$u = u(x, t) \in \R$ can be thought of as the dislocation (i.e., the displacement from its rest position) of an atom located at the point~$x \in \R$ at time~$t > 0$. Its time derivative is assumed to be equal to the sum of two forces. The term~$W'(u)$ represents a restorative force, responsible for the tendency of the atoms to occupy their original rest position or one located at an integer distance from it---we are assuming the interatomic distance to be equal to~$1$. It is given as the derivative of an even,~$1$-periodic, non-degenerate, multi-well potential~$W$. More precisely,~$W$ is assumed to be of class~$C^{4, 1}(\R)$ and to satisfy
\begin{equation} \label{Wprop}
\begin{alignedat}{3}
& W(r + k) = W(r) && \quad \mbox{for all } r \in \R, \, k \in \Z, \\
& W(r) = W(1 - r) && \quad \mbox{for all } r \in (0, 1), \\
& W(r) > 0 && \quad \mbox{for all } r \in (0, 1), \\
& W(0) = W'(0) = 0, && \quad \mbox{and} \quad W''(0) > 0.
\end{alignedat}
\end{equation}
The quantity~$(-\Delta)^s u$, for~$s \in (0, 1)$, encodes the presence of elastic interactions within the crystal. It is given by the~$1$-dimensional fractional Laplacian of order~$2 s$ in the variable~$x$, defined on a sufficiently smooth and bounded function~$v = v(x)$ as
\begin{equation} \label{fracLap}
(-\Delta)^s v(x) := \frac{1}{2} \int_{\R} \frac{2 v(x) - v(x + z) - v(x - z)}{|z|^{1 + 2 s}} \, dz \quad \mbox{for } x \in \R.
\end{equation}
Note that a positive factor should be included in front of this integral in order for~$(-\Delta)^s$ to really correspond to the~$s$-th power of the (negative) Laplacian---or, better, of (minus) the second derivative---see, e.g.,~\cite{S07,DPV12}. Such a factor does not play any role in our analysis and we therefore assume the above normalization for simplicity of exposition.

Notice that, when~$s = 1/2$, equation~\eqref{maineq} boils down to the evolutionary Peierls-Nabarro model proposed by Movchan, Bullough \& Willis~\cite{MBW98} (in the absence of external stress).

It is known that~\eqref{maineq} admits a bounded, non-constant, monotone, stationary solution~$w = w(x)$, which, at least when~$W$ does not have local minima in~$(0, 1)$, is unique up to translations in the independent variable~$x$ and integer translations and reflections in the dependent variable~$u$. In the model case of~$s = 1/2$ and~$W(r) = 1 - \cos(2 \pi r)$, this solution is explicit and was already found in~\cite{P40}---see the work~\cite{T97} of Toland for considerations on its uniqueness. In the more general setting considered here, similar results were obtained by Cabr\'e \& Sol\`a-Morales~\cite{CS05} (for~$s = 1/2$) and by Cabr\'e \& Sire~\cite{CS14, CS15} and Palatucci, Savin \& Valdinoci~\cite{PSV13} (for any~$s \in (0, 1)$). In these papers, the authors showed in particular that there exists a unique monotone increasing function~$w \in C^2(\R)$ that solves
\begin{equation} \label{ellPNeq}
(-\Delta)^s w + W'(w) = 0 \quad \mbox{in } \R
\end{equation}
and satisfies
\begin{equation} \label{wcond}
\lim_{x \rightarrow -\infty} w(x) = 0, \quad \lim_{x \rightarrow +\infty} w(x) = 1, \quad \mbox{and} \quad w(0) = \frac{1}{2}.
\end{equation}
This solution, often referred to as the (increasing) \emph{layer solution} of~\eqref{ellPNeq}, represents a dislocation of size~$1$ in the crystal. See also~\cite{CP16} and~\cite{CMY17} for related results.

In this paper, we are interested in studying dislocations of the size of any integer~$N \ge 2$ and oriented in the same direction, i.e., monotone solutions of~\eqref{maineq} which connect, say, the values~$0$ (as~$x \rightarrow -\infty$) and~$N$ (as~$x \rightarrow +\infty$). These kinds of solutions cannot be stationary (see Appendix~\ref{nomultiapp}) and represent therefore dislocations that evolve in time. In recent years, they have been investigated by several authors, for instance in~\cite{GM12,DPV15,DFV14,PV17}---see also~\cite{PV15,PV16} for studies on the case when the dislocations are not all equally oriented. In~\cite{PV17}, Patrizi \& Valdinoci showed that solutions of~\eqref{maineq} which at time~$0$ are equal to
\begin{equation} \label{initdatumPV}
\sum_{i = 1}^N w(x - x_i^0),
\end{equation}
with~$x_1^0 < x_2^0 < \ldots < x_N^0$ sufficiently spread apart and~$N$ even, converge as~$t \rightarrow +\infty$ to the constant~$N/2$---for~$N$ odd, their result suggests that the convergence should be instead to (a translation of) the layer solution~$w$. This statement can be deduced from a particular case of~\cite[Theorem~1.6]{PV17}.

We take this result as the starting point of our analysis, in which we address the fine asymptotic behavior of solutions of~\eqref{maineq} as~$t \rightarrow +\infty$. More precisely, we will construct solutions~$u$ of the evolutionary Peierls-Nabarro equation that, at large times, look like the superposition of~$N \ge 2$ dislocations of size~$1$ centered around~$N$ points~$\xi_1^0(t) < \xi_2^0(t) < \ldots < \xi_N^0(t)$ evolving according to the repulsive dynamical system
\begin{equation} \label{systemforxi0}
\dot{\xi}^0_i(t) = \frac{\gamma}{2 s} \sum_{j \ne i} \frac{\xi_i^0(t) - \xi_j^0(t)}{|\xi_i^0(t) - \xi_j^0(t)|^{1 + 2 s}} \quad \mbox{for all } t > 0 \mbox{ and } i = 1, \ldots, N,
\end{equation}
where we set
\begin{equation} \label{gammadef}
\gamma := \| w' \|_{L^2(\R)}^{-2} > 0.
\end{equation}
Around each~$\xi_i(t)$ the dislocation resembles a translation of the layer solution~$w$. That is, our solution~$u$ will be written as
\begin{equation} \label{udecomp}
u(x, t) = z(x, t) + \psi(x, t),
\end{equation}
with
\begin{equation} \label{zdef}
z(x, t) := \sum_{i = 1}^N w(x - \xi_i(t)) 
\end{equation}
and
\begin{equation} \label{xidecomp}
\xi_i(t) = \xi_i^0(t) + h_i(t) \quad \mbox{for } i = 1, \ldots, N,
\end{equation}
for some suitable perturbations~$\psi$ and~$h = (h_i)_{i = 1}^N$.

The connection between the dynamical system~\eqref{systemforxi0} and the Peierls-Nabarro equation has been already highlighted by Gonz\'alez \& Monneau~\cite{GM12} (for~$s = 1/2$), Dipierro, Palatucci \& Valdinoci~\cite{DPV15} (for~$s > 1/2$), and Dipierro, Figalli \& Valdinoci~\cite{DFV14} (for~$s < 1/2$)---see the forthcoming Remark~\ref{epsrmk} for more precise information. It turns out that system~\eqref{systemforxi0} has a solution~$\xi^0: (0, +\infty) \to \R^N$ of the form
\begin{equation} \label{xi0def}
\xi^0_i(t) := \beta_i \, t^{\frac{1}{1 + 2 s}} \quad \mbox{for } t > 0 \mbox{ and } i = 1, \ldots, N,
\end{equation}
for a vector~$\beta = (\beta_1, \ldots, \beta_N) \in \R^N$ whose components satisfy~$\beta_1 < \beta_2 < \ldots < \beta_N$. Solutions of this kind are unique and they characterize the long-time dynamics of \emph{all} solutions of~\eqref{systemforxi0}---see Propositions~\ref{explicitprop} and~\ref{xi0charprop} in Section~\ref{dynsystsec}. For this reason, when dealing with the long-time asymptotics of equation~\eqref{maineq}, it is not restrictive to assume the solution~$\xi^0$ of~\eqref{systemforxi0} to be precisely the one in~\eqref{xi0def}. In the remainder of the paper, we will always make this assumption.

Our main result is as follows.

\begin{theorem} \label{mainthm}
Let~$s \in (0, 1)$,~$W \in C^{4, 1}(\R)$ be a potential satisfying~\eqref{Wprop}, and~$N \ge 2$ be an integer. Then, for every~$T \ge 1$ sufficiently large, there exists a solution~$u$ of
\begin{equation} \label{maineqfromT}
\partial_t u + (-\Delta)^s u + W'(u) = 0 \quad \mbox{in } \R \times (T, +\infty),
\end{equation}
in the form~\eqref{udecomp}-\eqref{xidecomp}, with~$\psi: \R \times [T, +\infty) \to \R$ satisfying~$\psi(\cdot, T) = 0$ in~$\R$ and
\begin{equation} \label{psidecayintro}
|\psi(x, t)| \le C \min \left\{ t^{- \frac{2 s}{1 + 2 s}}, |x|^{- 2 s} \right\} \quad \mbox{for all } x \in \R, \, t > T,
\end{equation}
for some constant~$C > 0$, and~$h: [T, +\infty) \to \R^N$ such that~$h(T) = 0$ and~$\lim_{t \rightarrow +\infty} t^{- \frac{1}{1 + 2 s}} |h(t)| = 0$. When~$s \in (1/2, 1)$, it actually holds~$\lim_{t \rightarrow +\infty} |h(t)| = 0$.
\end{theorem}

Unless otherwise specified, by a solution of~\eqref{maineqfromT} or of other related evolution equations, we always mean a \emph{mild} solution, obtained, by Duhamel's principle, via convolution with the heat kernel associated to~$(-\Delta)^s$---see Definition~\ref{strongsoldef} in Section~\ref{solsec}. In the case of~\eqref{maineqfromT}, mild solutions are sufficiently regular for the equation to make sense pointwise. However, at several other points in the paper we will consider equations with weaker regularization properties and for which it is important to stipulate a notion of solution.

According to the final part of the statement of Theorem~\ref{mainthm}, when~$s \in (1/2, 1)$ the error term~$h(t)$ goes to zero as~$t \rightarrow +\infty$, ensuring that the points~$\{ \xi_i(t) \}$ at which the dislocations are centered converge to the exact solution~\eqref{xi0def} of~\eqref{systemforxi0}. For~$s \in (0, 1/2]$, our result is less precise, but still shows that the~$\xi_i$'s are asymptotical to the~$\xi_i^0$'s. See Theorem~\ref{mainthm2} in Section~\ref{outsec} for a more general and detailed reformulation of Theorem~\ref{mainthm}, containing, in particular, quantitative information on the growth/decay rate of~$h$.

To explain the difference between the two cases~$s \in (0, 1/2]$ and~$s \in (1/2, 1)$, we proceed to outline the general strategy of the proof of Theorem~\ref{mainthm}. We point out that our approach is inspired by similar ones of perturbative nature used for instance in~\cite{DdS18} and~\cite{dG18} to construct ancient solutions to the Yamabe flow and Allen-Cahn equation, respectively.

In order for~$u$, given by~\eqref{udecomp}-\eqref{xidecomp}, to be a solution of equation~\eqref{maineqfromT}, it is immediate to see that we need to find~$\psi$ and~$h$ satisfying
\begin{equation} \label{eqforpsiintro}
\partial_t \psi + (-\Delta)^s \psi +  W''(z) \psi + \NN[\psi] + \E = 0 \quad \mbox{in } \R \times (T, +\infty),
\end{equation}
for a nonlinear term~$\NN[\psi]$, quadratic in~$\psi$, and an error term~$\E$, independent of~$\psi$---see the forthcoming~\eqref{Ndef} and~\eqref{Edef} for their definitions. Both these functions (and~$z$) depend on~$\xi$, and thus on~$h$. To find a solution~$\psi$ of~\eqref{eqforpsiintro} with the decay prescribed in~\eqref{psidecayintro}, we implement the following Lyapunov-Schmidt type reduction.

First, we let~$h$ be fixed and infinitesimal at~$+\infty$ with respect to~$\xi^0$. To solve~\eqref{eqforpsiintro}, it is convenient to consider a projected version of it, in which we require~$\psi$ to be~$L^2$-orthogonal in space to the functions~$Z_i(x, t) := w'(x - \xi_i(t))$. As a result, we need to modify~\eqref{eqforpsiintro} in order for some compatibility conditions to be fulfilled. The result is that we look for a function~$\psi$ satisfying
\begin{equation} \label{psiortintro}
\int_\R \psi(x, t) Z_i(x, t) \, dx = 0 \quad \mbox{for all } t > T \mbox{ and } i = 1, \ldots, N
\end{equation}
and
\begin{equation} \label{projectedeqforpsi}
\partial_t \psi + (-\Delta)^s \psi +  W''(z) \psi + \NN[\psi] + \E = \sum_{i = 1}^N c_i Z_i \quad \mbox{in } \R \times (T, +\infty),
\end{equation}
for a uniquely determined vector of coefficients~$c: (T, +\infty) \to \R^N$. Through a fixed point argument, one shows that~\eqref{psiortintro}-\eqref{projectedeqforpsi} has a solution~$\psi$ satisfying the decay estimate~\eqref{psidecayintro}. The reason for considering~\eqref{psiortintro}-\eqref{projectedeqforpsi} instead of~\eqref{eqforpsiintro} is the following. Equation~\eqref{eqforpsiintro} can be understood, for~$x$ close to each~$\xi_i(t)$ and for a fixed large~$t$, as a perturbation of the linearization at~$w(\cdot - \xi_i(t))$ of the stationary Peierls-Nabarro equation~\eqref{ellPNeq}, which, after a translation of vector~$\xi_i(t)$ in the space variable, is determined by the operator~$L_w := (-\Delta)^s + W''(w)$. It turns out that~$L_w$ has a~$1$-dimensional kernel, spanned by~$w'$---this fact, called \emph{non-degeneracy} of the layer solution~$w$, follows, e.g., from~\cite[Lemma~5.3]{DPV15}. In view of this, to guarantee that a solution~$v$ of~$L_w v = f$ is small if the right-hand side~$f$ is small, we need~$v$ to be orthogonal to~$w'$. Going back to the old variables, it is natural to impose~\eqref{psiortintro} in order for~$\psi$ to fulfill~\eqref{psidecayintro}. The prescription of~\eqref{psiortintro} comes at the price of exchanging equation~\eqref{eqforpsiintro} for its projected version~\eqref{projectedeqforpsi}.

Up to now, we have shown that, for any~$h$, there exists a solution~$\psi$ of~\eqref{projectedeqforpsi} satisfying~\eqref{psidecayintro}. The problem of solving our initial equation~\eqref{eqforpsiintro} has then been reduced to that of finding~$h$ for which the coefficients~$\{ c_i \}$ in~\eqref{projectedeqforpsi} vanish identically. It can be seen that this is equivalent to having~$h$ solve a nonlinear system of the form
\begin{equation} \label{eqforh}
\dot{h}(t) + \frac{1}{t} \, \mathscr{M} h(t) = F[h](t) \quad \mbox{for all } t > T,
\end{equation}
for some constant positive semi-definite matrix~$\mathscr{M}$ and a right-hand side~$F = F[h]$ depending on~$h$ and~$t$---see~\eqref{Fdef}-\eqref{F2def} for the definition of~$F$, while~$t^{-1} \mathscr{M}$ is the matrix~$D_\xi \mathscr{R}(\xi^0(t))$ appearing in~\eqref{hdot=f}. Notice now that the reason we chose the quantity on the right-hand side of~\eqref{psidecayintro} to control~$\psi$ is that the error term~$\E$ appearing in~\eqref{projectedeqforpsi} can be bound in terms of it. In consequence of this and other facts,~$F$ satisfies the decay estimate
\begin{equation} \label{Fdecayintro}
|F(t)| \le C \, t^{- \frac{4 s}{1 + 2 s}} \quad \mbox{for all } t > T.
\end{equation}
From this, it can be seen that the solutions~$h$ of~\eqref{eqforh}---or, at least, those that are orthogonal, at each time, to the~$1$-dimensional kernel of~$\mathscr{M}$---are bounded by~$C_\varepsilon t^{1 - 4 s / (1 + 2 s) + \varepsilon} = C_\varepsilon t^{(1 - 2 s)/(1 + 2 s) + \varepsilon}$, for any arbitrarily small~$\varepsilon > 0$ and some positive constant~$C_\varepsilon$. Thus, for~$s \in (1/2, 1)$ they decay, while for~$s \in (0, 1/2]$ they might be unbounded---but still lower order with respect to~$\xi^0$.

A second ingredient needed to get~\eqref{Fdecayintro} is a good understanding of the asymptotic behavior of the layer solution. In~\cite{GM12,DPV15,DFV14}, it is shown that~$w$ has asymptotic expansion at infinity determined by the estimate
\begin{equation} \label{wasymptDPFV}
\left| w(x) - \chi_{(0, +\infty)}(x) + \frac{1}{2 s W''(0)} \frac{x}{|x|^{1 + 2 s}} \right| \le \frac{C}{|x|^{\vartheta}} \quad \mbox{for all } x \in \R,
\end{equation}
for some~$\vartheta > 2 s$. When~$s \in (1/2, 1)$, the exponent~$\vartheta$ is equal to~$1 + 2 s$---see~\cite[Proposition~7.2]{DPV15}. If we used this bound to estimate the decay of~$F$, we would have only gotten that~$|F(t)| \lesssim t^{-1}$, which is not enough to ensure that the solutions of~\eqref{eqforh} decay. In order to obtain the stronger bound~\eqref{Fdecayintro}, one therefore needs to improve~\eqref{wasymptDPFV}. We do this in Proposition~\ref{uimprovasymptprop}, where we show that~\eqref{wasymptDPFV} holds with~$\vartheta = 4 s$. This result, which actually holds for all~$s \in (0, 1)$, strongly uses the parity of the potential~$W$---an hypothesis that is not made in~\cite{GM12,DPV15,DFV14}.

A consequence of the fact that the error term~$h(t)$ is infinitesimal when~$s \in (1/2, 1)$ is the asymptotic stability of the solution~$u$ built in Theorem~\ref{mainthm}, with respect to small odd perturbations of its initial datum.

\begin{theorem} \label{asymptstabthm}
Let~$s \in (1/2, 1)$,~$W \in C^{4, 1}(\R)$ be a potential satisfying~\eqref{Wprop}, and~$N \ge 2$ be an integer. There exists~$\delta_0 > 0$ such that, for every~$T \ge 1$ sufficiently large, if~$u$ is the solution of~\eqref{maineqfromT} given by Theorem~\ref{mainthm} and~$\widetilde{u}$ is another solution of~\eqref{maineqfromT} satisfying~$\widetilde{u}(\cdot, T) = u(\cdot, T) + \eta$ in~$\R$, for some odd function~$\eta: \R \to \R$ such that
$$
|\eta(x)| \le \delta_0 \min \left\{ T^{- \frac{2 s}{1 + 2 s}}, |x|^{- 2 s} \right\} \quad \mbox{for all } x \in \R,
$$
then,
\begin{equation} \label{u-utildeto0}
\lim_{t \rightarrow +\infty} \| \widetilde{u}(\cdot, t) - u(\cdot, t) \|_{L^\infty(\R)} = 0.
\end{equation}
\end{theorem}

Theorem~\ref{asymptstabthm} gives the asymptotic stability of~$u$ with respect to odd perturbations. The requirement on the oddness of the perturbation~$\eta$ is related to a technical limitation of the construction leading to Theorem~\ref{mainthm}, where we require the~$N$ components of the error~$h$ to satisfy a symmetry condition---identities~\eqref{hodd} in Section~\ref{notsec}. It might be the case that such a restriction on~$\eta$ could be relaxed---this would probably entail carrying through the construction of Theorem~\ref{mainthm} under a weaker hypothesis than~\eqref{hodd}, such as the requirement that~$h$ has null barycenter~$N^{-1} \sum_{i = 1}^N h_i(t)$ for all~$t > T$. However, it cannot be completely removed, as one can easily see by considering as~$\widetilde{u}(x, t)$ the translation~$u(x + \delta, t)$, for a small~$\delta$.

Note that, although the case~$N = 1$ is not formally included in our framework, it is possible to adapt the arguments of the proof of Theorem~\ref{asymptstabthm} to this case and establish the asymptotic dynamical stability of the layer solution~$w$ with respect to odd perturbations.

\begin{remark} \label{epsrmk}
The works~\cite{GM12,DPV15,DFV14} were mostly focused on the large-scale limit of the solutions of~\eqref{maineq} and of more general Peierls-Nabarro type equations that include the presence of an external stress. In the simplest case of a vanishing stress, they proved that a solution~$u_\varepsilon = u_\varepsilon(x, t)$ of
\begin{equation} \label{epsPNeq}
\partial_t u_\varepsilon + \frac{1}{\varepsilon} \left( (-\Delta)^s u_\varepsilon + \frac{1}{\varepsilon^{2 s}} W'(u_\varepsilon) \right) = 0 \quad \mbox{in } \R \times (0, +\infty),
\end{equation}
with initial datum given by~\eqref{initdatumPV}, for some~$x_1^0 < x_2^0 < \ldots < x_N^0$, converges, as~$\varepsilon \rightarrow 0^+$, to the function
\begin{equation} \label{roughtrans}
\sum_{i = 1}^N \chi_{(0, +\infty)}(x - \xi^0_i(t)),
\end{equation}
a.e.~in~$\R \times (0, +\infty)$, for some solution~$\xi^0$ of~\eqref{systemforxi0} such that~$\xi^0(0) = x^0$.

The connection between equations~\eqref{epsPNeq} and~\eqref{maineq} is realized by the fact that~$u$ is a solution of~\eqref{maineq} if and only if
\begin{equation} \label{uepsdef}
u_\varepsilon(x, t) := u \! \left( \frac{x}{\varepsilon}, \frac{t}{\varepsilon^{1 + 2 s}} \right)
\end{equation}
solves~\eqref{epsPNeq}. Hence,~\eqref{epsPNeq} represents a blown-down version of~\eqref{maineq}, and the results of~\cite{GM12,DPV15,DFV14} show that solutions of~\eqref{maineq} having initial data like~\eqref{initdatumPV}, when viewed at large scales, look like~$N$ equally oriented rough dislocations centered around points that evolve according to~\eqref{systemforxi0}.

We point out that, via the rescaling~\eqref{uepsdef}, it is possible to recover this result using Theorem~\ref{mainthm}, at least for a restricted class of ``well-prepared'' initial data~$x^0$---this limitation could be partially overcome by shaping~$u$ around an appropriate solution~$\xi^0$ of~\eqref{systemforxi0} which may differ from~\eqref{xi0def}. In addition to this, through estimate~\eqref{psidecay} one can deduce quantitative information on the rate of convergence of~$u^\varepsilon$ to~\eqref{roughtrans}, away from the trajectories of~$\xi^0$.

In the series of papers~\cite{PV15,PV16,PV17}, Patrizi \& Valdinoci studied the evolution of dislocations which may not be all equally oriented---i.e., having initial data given by the superpositions of both increasing and decreasing layer solutions, centered at the points~$x^0_i$'s. In~\cite{PV15}, they proved that a large-scale limit result analogous to the one described before also holds true in this case. However, the system that regulates the evolution of the dislocation points~$\xi^0$ need not be repulsive anymore---unlike~\eqref{systemforxi0}---and, therefore, collisions in finite time may occur. Articles~\cite{PV16,PV17} were then mainly devoted to the analysis of the behavior of the solution~$u_\varepsilon$ past the first collision time~$T_c$.

We believe that a construction similar to the one performed in Theorem~\ref{mainthm} could lead to quantitative information on the profile of~$u_\varepsilon$ right before~$T_c$, at least under some assumptions on the initial orientations of the dislocations. For example, by modifying our techniques in agreement with the arguments of, say,~\cite{dG18}, it should be possible to construct ancient solutions of~$\partial_t u + (-\Delta)^s u + W'(u) = 0$ modeling, at large negative times, the propagation of~$N = 2$ dislocations oriented in opposite directions. Such solutions would have the form
$$
u(x, t) \approx w(x - \xi_1(t)) - w(x - \xi_2(t)) \quad \mbox{for } x \in \R, \, t < - T,
$$
with~$\xi_1 < \xi_2$ approximately satisfying the attractive dynamical system
$$
\dot{\xi}_i(t) \approx - \frac{\gamma}{2 s} \sum_{j \ne i} \frac{\xi_i(t) - \xi_j(t)}{|\xi_i(t) - \xi_j(t)|^{1 + 2 s}} \quad \mbox{for all } t < - T \mbox{ and } i = 1, 2,
$$
for some large~$T \ge 1$. When rescaled according to~\eqref{uepsdef}, these solutions would describe the blow-down profile right before the collision time~$T_c$.
\end{remark}

\smallskip

The remainder of the paper is organized as follows.

In Section~\ref{notsec}, we define some relevant quantities for our analysis, including the norms used to measure the error terms~$\psi$ and~$h$, as well as the corresponding function spaces.

Section~\ref{outsec} contains a detailed account of the strategy that we follow to establish Theorem~\ref{mainthm}---actually, in the more general and precise form of Theorem~\ref{mainthm2}, also stated there.

In Section~\ref{dynsystsec}, we carry out the analysis of the dynamical system~\eqref{systemforxi0}, showing in particular the existence and uniqueness of a solution in the form~\eqref{xi0def}.

The brief Section~\ref{layersec} is devoted to the asymptotic properties of~$w$. There, we state in particular our result on the improvement of expansion~\eqref{wasymptDPFV} for odd layer transitions. The proof of this and other related estimates is postponed to Appendix~\ref{layerapp}.

Section~\ref{solsec} is mostly a review of fractional parabolic equations. It contains the definition of the notion of mild solutions that we adopt throughout the paper, as well as some well-known properties which immediately follow from it---such as the existence and uniqueness of solutions, maximum principles, and basic regularity estimates (the proof of these estimates is deferred to Appendix~\ref{estforstrongapp}). At the end of the section, we also include the construction of a couple of barriers and a decay estimate for a particular class of solutions of the linearization of equation~\eqref{maineqfromT} at~$w$.

Sections~\ref{psisec} and~\ref{hsec} represent the core part of the paper. There, we finalize the proof of Theorem~\ref{mainthm} by establishing the main results stated in Section~\ref{outsec}.

Section~\ref{stabsec} is devoted to the proof of the stability properties of the solution~$u$, as stated in Theorem~\ref{asymptstabthm}.

The paper is closed by three appendices. As anticipated before, the first two contain proofs of results stated in previous sections. Conversely, Appendix~\ref{nomultiapp} is self-contained and concerned with the classification of bounded solutions to the stationary Peierls-Nabarro equation~\eqref{ellPNeq}.

\section{Notation} \label{notsec}

\noindent
In this section, we present some non-standard definitions that will be used in the sequel. We begin by introducing the functional setting for the perturbations~$h$ and~$\psi$. The last subsection contains the definitions of the functions~$\NN$ and~$\E$ which already appeared in~\eqref{eqforpsiintro}.

\subsection{Norms for~$h$.}

Let~$\mu$ be any number satisfying
\begin{equation} \label{mulimit0}
\mu > \frac{2 s}{1 + 2 s}
\end{equation}
and define, given~$T \ge 1$ and~$h \in C^1([T, +\infty); \R^N)$, the norm
\begin{equation} \label{normforh}
\| h \|_{\HH_{T, \mu}} := \sup_{t > T} \Big( t^{\mu - 1} |h(t)| \Big) + \sup_{t > T} \Big( t^\mu |\dot{h}(t)| \Big).
\end{equation}
We will always work with perturbations~$h$ that are evenly distributed, meaning that
\begin{equation} \label{hodd}
h_i(t) = - h_{N - i + 1}(t) \quad \mbox{for all } t > T \mbox{ and } i = 1, \ldots, N.
\end{equation}
As a result, we consider the spaces
\begin{equation} \label{HTmudef}
\HH_{T, \mu} := \Big\{ h \in C^1([T, +\infty); \R^N)  : h \mbox{ satisfies } \eqref{hodd} \mbox{ and } \| h \|_{\HH_{T, \mu}} < + \infty \Big\}
\end{equation}
along with their closed unit balls
$$
\bar{B}_1(\HH_{T, \mu}) := \Big\{ h \in \HH_{T, \mu} : \| h \|_{\HH_{T, \mu}} \le 1 \Big\}.
$$
Notice that assumption~\eqref{mulimit0} on~$\mu$ guarantees that the~$h_i$'s have strictly lower growth rates at~$+\infty$ than the~$\xi_i^0$'s given by~\eqref{xi0def}. In particular, up to taking~$T$ sufficiently large, any~$h \in \bar{B}_1(\HH_{T, \mu})$ gives rise to a family of trajectories~$\xi = \xi^0 + h$ satisfying
$$
\xi_1(t) < \xi_2(t) < \ldots < \xi_N(t) \quad \mbox{for all } t > T.
$$
In what follows, we will always assume this to hold.

As we shall see later on, the perturbation~$h$ leading to the desired solution~$u$ of~\eqref{maineqfromT} is obtained as the solution of a nonlinear system of equations. Its right-hand side will belong to the following spaces. Given,~$T \ge 1$,~$\mu > 0$, and~$\alpha \in (0, 1)$, we define the spaces
\begin{equation} \label{FTmudef}
\begin{aligned}
\FF_{T, \mu} & := \Big\{ f \in L^\infty((T, +\infty); \R^N) : f \mbox{ satisfies } \eqref{fodd} \mbox{ and } \| f \|_{\FF_{T, \mu}} < +\infty \Big\}, \\
\FF_{T, \mu}^{\alpha} & := \Big\{ f \in L^\infty((T, +\infty); \R^N) : f \mbox{ satisfies } \eqref{fodd} \mbox{ and } \| f \|_{\FF_{T, \mu}^{\alpha}} < +\infty \Big\},
\end{aligned}
\end{equation}
characterized by the property
\begin{equation} \label{fodd}
f_i(t) = - f_{N - i + 1}(t) \quad \mbox{for all } t > T \mbox{ and } i = 1, \ldots, N
\end{equation}
and by the norms
\begin{align*}
\| f \|_{\FF_{T, \mu}} & := \sup_{t > T} \Big( t^{\mu} \, |f(t)| \Big), \\
\| f \|_{\FF_{T, \mu}^{\alpha}} & := \| f \|_{\FF_{T, \mu}} + \sup_{t > T} \Big( t^{\mu} \, [f]_{C^{\alpha}((t, t + 1); \R^N)} \Big).
\end{align*}

\subsection{Norms for~$\psi$.}

We measure the decay of the corrector~$\psi$ through the weight
\begin{equation} \label{Phidef}
\Phi(x, t) := \min \left\{ |x|^{- 2 s}, t^{- \frac{2 s}{1 + 2 s}} \right\} \quad \mbox{for } x \in \R, \, t \ge 1.
\end{equation}
Given any open interval~$I \subseteq [1, +\infty)$, we define the Banach space
$$
\A_I := \Big\{ \psi \in L^\infty(\R \times I) : \| \psi \|_{\A_I} < +\infty \Big\},
$$
with norm
$$
\| \psi \|_{\A_I} := \left\| \frac{\psi}{\Phi} \right\|_{L^\infty(\R \times I)}.
$$
The reason for considering this particular norm has been already anticipated in the introduction and is related to the fact that the error term~$\E$ (defined below in~\eqref{Edef}) satisfies~$\| \E \|_{\A_I} < +\infty$---see Lemma~\ref{ElePhilem} in Section~\ref{psisec}.

A subset of~$\A_I$ that will be of key importance in our analysis is composed by the elements of~$\A_I$ which are~$L^2$-orthogonal in space to the~$N$ functions
\begin{equation} \label{Zidef}
Z_i(x, t) := w'(x - \xi_i(t)) \quad \mbox{for } i = 1, \ldots, N,
\end{equation}
for a.e.~time~$t \in I$. That is, those functions~$\psi \in \A_I$ for which
\begin{equation} \label{psiort}
\int_\R \psi(x, t) Z_i(x, t) \, dx = 0 \quad \mbox{for a.e.~} t \in I \mbox{ and every } i = 1, \ldots, N.
\end{equation}
We call this subset~$\Adot_I$, i.e., we define
$$
\Adot_I := \Big\{ \psi \in \A_I : \psi \mbox{ satisfies } \eqref{psiort} \Big\}.
$$
Observe that the integral in~\eqref{psiort} is well-defined and finite for every bounded~$\psi$, as~$Z_i(\cdot, t) \in L^1(\R)$ for every~$t$---see estimate~\eqref{w'asympt} in Section~\ref{layersec}. Also notice that~$\Adot_I$ is closed in~$\A_I$.

The case we are mostly interested in is that of a time interval~$I$ of the form~$(T, +\infty)$, with~$T \ge 1$. In this situation, we simply write
$$
\A_T := \A_{(T, +\infty)} \quad \mbox{and} \quad \Adot_T := \Adot_{(T, +\infty)},
$$
and similarly for their norm~$\| \cdot \|_{\A_T}$. Sometimes, we will need to measure the continuity of the corrector function~$\psi$. Given~$T \ge 1$ and~$\alpha \in (0, 1)$, we consider the weighted H\"older space
\begin{equation} \label{Atnualphadef}
\A_T^{\alpha} := \Big\{ \psi \in L^\infty(\R \times (T, +\infty)) : \| \psi \|_{\A_T^{\alpha}} < +\infty \Big\},
\end{equation}
determined by the norm
\begin{equation} \label{CTnualphadef}
\| \psi \|_{\A_T^\alpha} := \| \psi \|_{\A_T} + \sup_{t > T} \left( t^{\frac{2 s}{1 + 2 s}} [ \psi ]_{C^\alpha(\R \times (t, t + 1))} \right).
\end{equation}

We also define an appropriate space for the initial datum~$\psi_0$ for~$\psi$ that will be considered in Theorem~\ref{mainthm2}. Taking into account the orthogonality conditions
\begin{equation} \label{psi0ort}
\int_\R \psi_0(x) Z_i(x, T) \, dx = 0 \quad \mbox{for all } i = 1, \ldots, N
\end{equation}
and the norm
$$
\| \psi_0 \|_{\A_{\{ T \}}} := \left\| \frac{\psi_0}{\Phi(\cdot, T)} \right\|_{L^\infty(\R)},
$$
we introduce the spaces
$$
\A_{\{ T \}} := \Big\{ \psi_0 \in L^\infty(\R) : \| \psi_0 \|_{\A_{\{ T \}}} < +\infty \Big\} \quad \mbox{and} \quad \Adot_{\{ T \}} := \Big\{ \psi_0 \in \A_{\{ T \}} : \psi_0 \mbox{ satisfies } \eqref{psi0ort} \Big\},
$$
As we will sometimes need~$\psi_0$ to be regular, we also consider the weighted~$C^1$ norm
\begin{equation} \label{ATalphainit}
\| \psi_0 \|_{\A_{\{ T \}}^1} := \| \psi_0 \|_{\A_{\{ T \}}} + T^{\frac{2 s}{1 + 2 s}} \| \psi_0' \|_{L^\infty(\R)}
\end{equation}
as well as the corresponding spaces
$$
\A_{\{ T \}}^1 := \Big\{ \psi \in C^1(\R) : \| \psi_0 \|_{\A_{\{ T \}}^1} < +\infty \Big\} \quad \mbox{and} \quad \Adot_{\{ T \}}^1 := \A_{\{ T \}}^1 \cap \Adot_{\{ T \}}.
$$

\subsection{Additional terminology}

Recalling definition~\eqref{zdef} of~$z = z(x, t)$, we introduce the following two functions, that will play an important role in the continuation of the paper. Given~$\psi: \R \times (T, +\infty)$, we define
\begin{equation} \label{Ndef}
\NN[\psi](x, t) := W'(z(x, t) + \psi(x, t)) - W'(z(x, t)) - W''(z(x, t)) \psi(x, t) \quad \mbox{for } x \in \R, \, t > T
\end{equation}
and
\begin{equation} \label{Edef}
\E(x, t) := W'(z(x, t)) - \sum_{j = 1}^N W'(w(x - \xi_j(t))) - \sum_{j = 1}^N w'(x - \xi_j(t)) \dot{\xi}_j(t) \quad \mbox{for } x \in \R, \, t > T.
\end{equation}
Sometimes, it will be convenient to decompose~$\E$ as
\begin{equation} \label{Edecomp}
\E = \E_1 + \E_2,
\end{equation}
with
\begin{align}
\label{E1def}
\E_1(x, t) & := W'(z(x, t)) - \sum_{j = 1}^N W'(w(x - \xi_j(t))), \\
\label{E2def}
\E_2(x, t) & := - \sum_{j = 1}^N Z_j(x, t) \dot{\xi}_j(t).
\end{align}
Recall definition~\eqref{Zidef}. We also define the functions
\begin{equation} \label{E0idef}
\E_{0, j}(x, t) := W''(z(x, t)) - W''(w(x - \xi_j(t))) \quad \mbox{for } j = 1, \ldots, N.
\end{equation}

Finally, throughout the paper we denote with~$C$ any~\emph{generic} positive constant. The value of~$C$ is usually large (greater than~$1$) and may change from line to line. Generic here means that~$C$ only depends on the \emph{structural} parameters of the model under analysis, which are~$s$,~$W$,~$N$, and~$\mu$. When some constant depends on further, non-structural quantities, we will typically stress it by means of subscripts---e.g., the notation~$C_{\varepsilon, T}$ indicates dependence on~$\varepsilon$ and~$T$.

\section{Outline of the proof of Theorem~\ref{mainthm}} \label{outsec}

\noindent
We present here in detail the strategy of the proof of Theorem~\ref{mainthm}. Thanks to the definitions introduced in the previous section, we can give a more precise statement of it, providing in particular information on the decay/growth rate of the perturbation~$h$. This rate is encoded in the norm~$\| \cdot \|_{\HH_{T, \mu}}$ defined in~\eqref{normforh} and corresponding to a number~$\mu$ satisfying
\begin{equation} \label{mulimit1}
\mu < \min \left\{ \frac{4 s}{1 + 2 s}, 1 + \delta \right\} \quad \mbox{and} \quad \mu > \begin{dcases}
\frac{3 s}{1 + 2 s} & \quad \mbox{if } s \in (0, 1/2],\\
1 & \quad \mbox{if } s \in (1/2, 1),
\end{dcases}
\end{equation}
where~$\delta > 0$ is given by
\begin{equation} \label{deltadef}
\delta := \frac{N \gamma}{16 \, \beta_N^{1 + 2 s}},
\end{equation}
with~$\gamma$ as in~\eqref{gammadef}. Notice that such a~$\mu$ satisfies in particular assumption~\eqref{mulimit0}.

We have the following statement.

\begin{theorem} \label{mainthm2}
Let~$\mu$ be a real number fulfilling conditions~\eqref{mulimit1}-\eqref{deltadef}. Then, there exist three generic constants~$T_0, C \ge 1$ and~$\varepsilon_0 \in (0, 1)$ such that, given any~$T \ge T_0$, any odd function~$\psi_0 \in \Adot^1_{\{ T \}}$, and any vector~$h^{(0)} \in \R^N$ satisfying
\begin{equation} \label{h0odd}
h^{(0)}_i = - h^{(0)}_{N - i + 1} \quad \mbox{for all } i = 1, \ldots, N
\end{equation}
and
\begin{equation} \label{h0psi0small}
\| \psi_0 \|_{\A_{\{ T \}}} + T^{\mu - 1} |h^{(0)}| \le \varepsilon_0,
\end{equation}
there exist~$h \in \bar{B}_1(\HH_{T, \mu})$ with~$h(T) = h^{(0)}$ and~$\psi \in \A_T$ with~$\psi(\cdot, T) = \psi_0$ and~$\| \psi \|_{\A_T} \le C$ for which the function~$u$ given by~\eqref{udecomp}-\eqref{xidecomp} with~$\xi = \xi^0 + h$ is a solution of equation~\eqref{maineqfromT}.
\end{theorem}

Observe that Theorem~\ref{mainthm} is a particular case of Theorem~\ref{mainthm2}, obtained by taking~$\psi_0$ and~$h^{(0)}$ equal to zero. Nevertheless, Theorem~\ref{mainthm2} allows for more general initial data~$\psi_0$ and~$h^{(0)}$ satisfying the smallness assumption~\eqref{h0psi0small}. Note that the requirement on~$\psi_0$ to be of class~$C^1$ is merely technical and could in fact be relaxed by assuming only H\"older continuity. We also remark that the possibility of having non-zero initial data will be crucial to prove the stability result of Theorem~\ref{asymptstabthm}, later in Section~\ref{stabsec}.

The remainder of the section is occupied by the scheme of the proof of Theorem~\ref{mainthm2}. As we will shortly see, the argument rests on a few key results, the proofs of which are postponed to Sections~\ref{psisec} and~\ref{hsec}.

Let~$\xi$ be as in~\eqref{xidecomp}, with~$\xi^0$ being the explicit solution~\eqref{xi0def} of system~\eqref{systemforxi0}---whose existence and main properties will be discussed in Section~\ref{dynsystsec}---and for some~$h \in \bar{B}_1(\HH_{T, \mu})$ to be later determined. It is immediate to see that~$u$ given by~\eqref{udecomp}-\eqref{zdef} is a solution of~\eqref{maineqfromT} if and only if~$\psi$ solves
\begin{equation} \label{eqforpsi}
\partial_t \psi + (-\Delta)^s \psi +  W''(z) \psi + \NN[\psi] + \E = 0 \quad \mbox{in } \R \times (T, +\infty),
\end{equation}
with~$\NN[\psi]$ and~$\E$ respectively given by~\eqref{Ndef} and~\eqref{Edef}. To find a solution~$\psi$ of~\eqref{eqforpsi} belonging to~$\Adot_T$, we consider, at a first stage, the projected Dirichlet problem
\begin{equation} \label{psiDirprob}
\begin{dcases}
\partial_t \psi + (- \Delta)^s \psi + W''(z) \psi + \NN[\psi] + \E = \sum_{j = 1}^N c_j Z_j & \quad \mbox{in } \R \times (T, +\infty) \\
\psi = \psi_0 \vphantom{\sum_{j = 1}^N c_j Z_j} & \quad \mbox{on } \R \times \{ T \},
\end{dcases}
\end{equation}
where~$\psi_0$ lies in~$\Adot_{\{ T \}}$ and~$c = (c_j)_{j = 1}^N : (T, +\infty) \to \R^N$ is a suitable vector of (time-dependent) coefficients---the reasons for considering~\eqref{psiDirprob} instead of~\eqref{eqforpsi} have been explained in the introduction, after the statement of Theorem~\ref{mainthm}.

It turns out that, if~$T$ is sufficiently large, problem~\eqref{psiDirprob} is uniquely solvable in~$\Adot_T$, as shown by the following result.

\begin{theorem} \label{nonlinearthm}
Assume that~$\mu$ satisfies~\eqref{mulimit0}. Then, there exist three generic constants~$T_0, C \ge 1$, and~$\varepsilon_0 \in (0, 1)$ such that, given any~$T \ge T_0$,~$h \in \bar{B}_1(\HH_{T, \mu})$, and~$\psi_0 \in \Adot_{\{ T \}}$ with~$\| \psi_0 \|_{\A_{\{ T \}}} \le \varepsilon_0$, there exists a unique solution~$\psi \in \Adot_T$ of problem~\eqref{psiDirprob} satisfying
\begin{equation} \label{psidecay}
\| \psi \|_{\A_T} \le C.
\end{equation}
The vector of coefficients~$c$ is uniquely determined as the solution of the linear system
\begin{equation} \label{Ac=bt>Tige1}
A(t) \, c(t) = b(t) \quad \mbox{for all } t > T,
\end{equation}
where the matrix~$A(t) = (A_{i j}(t))_{i, j = 1}^N$ is given by
\begin{equation} \label{Aijdef}
A_{i j}(t) := \int_\R Z_i(x, t) Z_j(x, t) \, dx
\end{equation}
and the vector~$b(t) = (b_i(t))_{i = 1}^N$ by
\begin{equation} \label{bdef}
b_i(t) := \int_{\R} \Big\{ \E_{0, i}(x, t) \psi(x, t) + \E(x, t) + \NN[\psi](x, t) \Big\} Z_i(x, t) \, dx + \dot{\xi}_i(t) \int_{\R} \psi(x, t) w''(x - \xi_i(t)) \, dx.
\end{equation}
In addition, if~$\psi_0$ is odd, then~$\psi$ is odd as well in the variable~$x$, meaning that
\begin{equation} \label{psiodd}
\psi(- x, t) = - \psi(x, t) \quad \mbox{for all } x \in \R, \, t > T.
\end{equation}
\end{theorem}

Note that, for~$t$ sufficiently large, the matrix~$A(t)$ is invertible and thus~\eqref{Ac=bt>Tige1} admits a unique solution~$c(t) = A(t)^{-1} b(t)$---see Lemma~\ref{clem}.

We will obtain Theorem~\ref{nonlinearthm} via a fixed point argument based on the linearization of problem~\eqref{psiDirprob}, namely
\begin{equation} \label{psiprobthm}
\begin{dcases}
\partial_t \psi + (-\Delta)^s \psi + W''(z) \psi = f + \sum_{j = 1}^N c_j Z_j & \quad \mbox{in } \R \times (T, +\infty) \\
\psi = \psi_0 \vphantom{\sum_{j = 1}^N c_j Z_j} & \quad \mbox{on } \R \times \{ T \},
\end{dcases}
\end{equation}
with~$f \in \A_T$. The solvability of~\eqref{psiprobthm} is addressed by the following result.

\begin{proposition} \label{mainlinprop}
Assume that~$\mu$ satisfies~\eqref{mulimit0}. Then, there exists a generic constant~$T_0 \ge 1$ such that, given any~$T \ge T_0$,~$h \in \bar{B}_1(\HH_{T, \mu})$,~$f \in \A_T$, and~$\psi_0 \in \Adot_{\{ T \}}$, there exist a solution~$\psi \in \Adot_T$ of~\eqref{psiprobthm}. The vector of coefficients~$c$ is the solution of the system~\eqref{Ac=bt>Tige1}, with~$A = (A_{i j})_{i, j = 1}^N$ given by~\eqref{Aijdef} and~$b = (b_i)_{i = 1}^N$ by
\begin{equation} \label{bidef}
b_i(t) := \int_{\R} \Big\{ \E_{0, i}(x, t) \psi(x, t) - f(x, t) \Big\} Z_i(x, t) \, dx + \dot{\xi}_i(t) \int_{\R} \psi(x, t) w''(x - \xi_i(t)) \, dx.
\end{equation}
Finally, the estimate
\begin{equation} \label{psioverPhiestthm}
\| \psi \|_{\A_T} \le C \left( \| f \|_{\A_T} + \| \psi_0 \|_{\A_{\{ T \}}} \right),
\end{equation}
holds true for some generic constant~$C > 0$.
\end{proposition}

That~$c$ is characterized by system~\eqref{Ac=bt>Tige1} (with~$b$ as in~\eqref{bidef} or~\eqref{bdef}, in either the linear or nonlinear case) can be checked, at least formally, by multiplying the equations solved by~$\psi$ against each~$Z_i$, integrating in space, and taking into account the orthogonality conditions~\eqref{psiort} satisfied by~$\psi$ (recall that~$\psi \in \Adot_T$). See Lemma~\ref{ortequivsystem} for a rigorous derivation of this fact in the context of mild solutions.

Theorem~\ref{nonlinearthm} ensures that, given any~$h \in \bar{B}_1(\HH_{T, \mu})$, there exist a solution~$\psi \in \Adot_T$ of~\eqref{psiDirprob}, for some array of coefficients~$c$. Of course, both~$\psi$ and~$c$ depend on the choice of~$h$. In order to solve equation~\eqref{eqforpsi}, our goal is then to find a suitable~$h$ for which~$c = c[h]$ satisfies
\begin{equation} \label{c=0}
c(t) = 0 \quad \mbox{for all } t > T.
\end{equation}
In view of how~$c$ is defined by~\eqref{Ac=bt>Tige1}-\eqref{bdef} and recalling definitions~\eqref{Edecomp}-\eqref{E2def}, it is immediate to verify that~\eqref{c=0} is equivalent to selecting~$h \in \bar{B}_1(\HH_{T, \mu})$ in a way that~$\xi = \xi^0 + h$ solves
$$
M(t) \, \dot{\xi}(t) = d(t) \quad \mbox{for all } t > T,
$$
where~$M(t) = (M_{i j}(t))_{i, j = 1}^N$ is defined by
\begin{equation} \label{Mdef}
M_{i j}(t) := A_{i j}(t) - \delta_{i j} \int_{\R} \psi(x, t) w''(x - \xi_i(t)) \, dx,
\end{equation}
with~$A_{i j}$ as in~\eqref{Aijdef}, and~$d(t) = (d_i(t))_{i = 1}^N$ by
\begin{equation} \label{didef}
d_i(t) := \int_{\R} \Big\{ \E_{0, i}(x, t) \psi(x, t) + \E_1(x, t) + \NN[\psi](x, t) \Big\} Z_i(x, t) \, dx.
\end{equation}
Once again, we stress that both~$M$ and~$d$ depend (nonlinearly) in~$h$.

It is not hard to see that the matrix~$M(t)$ is invertible, provided~$t$ is sufficiently large---see Lemma~\ref{Mlem}. By this and the fact that~$\xi^0$ solves~\eqref{systemforxi0}, we may equivalently rewrite~\eqref{Mdef} as
\begin{equation} \label{hdot=f}
\dot{h} + D_\xi \mathscr{R}(\xi^0) h = F[h] \quad \mbox{in } (T, +\infty),
\end{equation}
where~$\mathscr{R}$ is the vector-valued function defined by
\begin{equation} \label{Rmapdef}
\mathscr{R}_i(\xi) := - \frac{\gamma}{2 s} \sum_{j \ne i} \frac{\xi_i - \xi_j}{|\xi_i - \xi_j|^{1 + 2 s}} \quad \mbox{for } \xi \in \R^N \mbox{ and } i = 1, \ldots, N,
\end{equation}
while~$F = F[h]: (T, +\infty) \to \R^N$ is given by
\begin{equation} \label{Fdef}
F(t) := F^{(1)}(t) + F^{(2)}(t),
\end{equation}
with
\begin{align}
\label{F1def}
F^{(1)}(t) & := - \Big\{ \, \mathscr{R}(\xi^0(t) + h(t)) - \mathscr{R}(\xi^0(t)) - D_\xi \mathscr{R}(\xi^0(t)) h(t) \Big\}, \\
\label{F2def}
F^{(2)}(t) & := M^{-1}(t) \, d(t) + \mathscr{R}(\xi(t)).
\end{align}

Note that system~\eqref{hdot=f} is nonlinear in~$h$. Its solvability within~$\bar{B}_1(\HH_{T, \mu})$---for any small initial datum~$h^{(0)}$ satisfying~\eqref{h0odd}---is established in the next result, which holds true under the assumptions that~$\psi_0$ is an odd~$C^1$ function and that~$\mu$ fulfills~\eqref{mulimit1}-\eqref{deltadef}.

\begin{theorem} \label{Msystsolvthm}
Assume that~$\mu$ satisfies~\eqref{mulimit1}-\eqref{deltadef}. Then, there exists a generic constant~$T_0 \ge 1$ such that, given any~$T \ge T_0$, any odd~$C^1$ function~$\psi_0 \in \Adot_{\{ T \}}$, and any vector~$h^{(0)} \in \R^N$ satisfying~\eqref{h0odd} and~\eqref{h0psi0small}, there exists a solution~$h \in \bar{B}_1(\HH_{T, \mu})$ of
\begin{equation} \label{hdot+Rh=Fh}
\begin{cases}
\dot{h} + D_\xi \mathscr{R}(\xi^0) h = F[h] & \quad \mbox{in } (T, +\infty) \\
h(T) = h^{(0)}. &
\end{cases}
\end{equation}
\end{theorem}

The proof of Theorem~\ref{Msystsolvthm} is based on the study of the decay properties of~$F$, on the resolution of the linear problem associated to~\eqref{hdot+Rh=Fh}, and on a suitable application of a fixed point theorem. Note that, in the literature, the existence of a solution to nonlinear reduced problems such as~\eqref{hdot+Rh=Fh} is often proved using the contraction lemma. Here, this strategy does not seem feasible (at least when~$s \le 1/2$), due to the lack of regularity of the map~$\Psi$ that associates to each~$h \in \bar{B}_1(\HH_{T, \mu})$ the corresponding solution~$\psi = \Psi[h]$ of problem~\eqref{psiDirprob} given by Theorem~\ref{nonlinearthm}---see Remark~\ref{LipvsHolrmk} at the end of Section~\ref{psisec}. To circumvent this issue, we use instead the Schauder fixed point theorem, whose application is justified after a careful inspection of the compactness properties of the map~$h \mapsto F[h]$ with respect to an appropriate target space---chosen within the scale~\eqref{FTmudef}.

The analysis needed for the proof of Theorem~\ref{Msystsolvthm} is conducted in Section~\ref{hsec}, while the arguments leading to Theorem~\ref{nonlinearthm} and Proposition~\ref{mainlinprop} are contained in Section~\ref{psisec}. Pending the verification of these results, the proof of Theorem~\ref{mainthm2}---and, thus, of Theorem~\ref{mainthm}---is concluded.

\section{Analysis of the dynamical system~\eqref{systemforxi0}.} \label{dynsystsec}

\noindent
In this section, we provide some results concerning system~\eqref{systemforxi0}. Specifically, we prove that there exists a unique solution~$\xi^0$ in the form~\eqref{xi0def} and we show that any other solution of~\eqref{systemforxi0} with vanishing barycenter behaves as~$\xi^0$ as~$t \rightarrow +\infty$. The results stated here do not require the constant~$\gamma$ to assume the specific value prescribed by~\eqref{gammadef} and hold in fact for any~$\gamma > 0$.

We begin with the following result. To state it, we introduce the notation
$$
\BB := \Big\{ b \in \R^N : b_1 < b_2 < \ldots < b_N \Big\}.
$$

\begin{proposition} \label{explicitprop}
There exists a unique~$\beta \in \BB$ for which~$\xi^0$ defined by~\eqref{xi0def} solves system~\eqref{systemforxi0}. The vector~$\beta$ is the unique solution of
\begin{equation} \label{betaeq}
\beta_i = \frac{(1 + 2 s) \gamma}{2 s} \sum_{j \ne i} \frac{\beta_i - \beta_j}{|\beta_i - \beta_j|^{1 + 2 s}} \quad \mbox{for every } i = 1, \ldots, N
\end{equation}
that belongs to~$\BB$. In particular, it satisfies
\begin{equation} \label{betaodd}
\beta_i = - \beta_{N - i + 1} \quad \mbox{for all } i = 1, \ldots, N.
\end{equation}
\end{proposition}
\begin{proof}
Let~$\xi^0$ be given by~\eqref{xi0def}, for some~$\beta$ to be chosen. A straightforward computation reveals that~$\xi^0$ solves~\eqref{systemforxi0} if and only if~$\beta$ is a solution of~\eqref{betaeq}, that is, if and only if~$\beta$ is a stationary point of the functional~$\V \in C^2(\BB)$ defined by
$$
\V(b) := \frac{|b|^2}{2} + \V_0(b) \quad \mbox{for } b \in \BB,
$$
with
$$
\V_0(b) := \begin{dcases}
\frac{(1 + 2 s) \gamma}{2 s (2 s - 1)} \sum_{1 \le i < j \le N} |b_i - b_j|^{1 - 2 s} & \quad \mbox{if } s \ne \frac{1}{2},\\
- 2 \gamma \sum_{1 \le i < j \le N} \log |b_i - b_j| & \quad \mbox{if } s = \frac{1}{2}.
\end{dcases}
$$
The proof will be complete if we show that~$\V$ has a unique stationary point in~$\BB$.

In order to check this, we first observe that~$\BB$ is a convex open set. Secondly, it holds
$$
\langle D^2 \V(b) \eta, \eta \rangle = |\eta|^2 + (1 + 2 s) \gamma \sum_{1 \le i < j \le N} \frac{|\eta_i - \eta_j|^2}{|b_i - b_j|^{1 + 2 s}}  \ge |\eta|^2 \quad \mbox{for all } b \in \BB \mbox{ and } \eta \in \R^N.
$$
Accordingly,~$\V$ is strictly convex in~$\BB$, and its only possible stationary point is its unique global minimum in~$\BB$, provided it exists. To see that such a minimum indeed exists, we begin by noticing that
\begin{equation} \label{Vcoercive}
\V(b) \ge \frac{|b|^2}{4} \quad \mbox{for all } b \in \BB \mbox{ such that } |b| \ge R,
\end{equation}
for some constant~$R > 0$.

Now, when~$s \in [1/2, 1)$, it holds
$$
\lim_{\BB \ni b \rightarrow \overline{b}} \V(b) = +\infty \quad \mbox{for all } \overline{b} \in \partial \BB.
$$
From this and~\eqref{Vcoercive} we then immediately infer that~$\V$ has a global minimum in~$\BB$.

The case~$s \in (0, 1/2)$ is a bit more delicate. Under this assumption, we have that~$\V \in C^0(\overline{\BB})$. By this and~\eqref{Vcoercive}, it follows that~$\V$ has a global minimum in~$\overline{\BB}$. We claim that, given any~$\overline{b} \in \partial \BB$, we can find~$b \in \BB$ such that
\begin{equation} \label{Vb<Vbbar}
\V(b) < \V(\overline{b}).
\end{equation}
Of course, this would lead us to conclude that the global minimum of~$\V$ lies in~$\BB$. To verify this claim, let
$$
b := \overline{b} + \varepsilon (1, 2, \ldots, N),
$$
for some small~$\varepsilon > 0$ to be decided later. Clearly,~$b \in \BB$. Using that~$\overline{b}_1 \le \overline{b}_2 \le \ldots \le \overline{b}_N$, it is easy to see that
\begin{equation} \label{bij>barbij}
|b_i - b_j| > |\overline{b}_i - \overline{b}_j| \quad \mbox{for every } i \ne j.
\end{equation}
Moreover, as~$\overline{b} \in \partial \BB$, there exists~$I \in \{ 1, \ldots, N - 1 \}$ such that~$\overline{b}_I = \overline{b}_{I + 1}$. By this and~\eqref{bij>barbij}, we get
\begin{align*}
\V(b) - \V(\overline{b}) & = \varepsilon \sum_{j = 1}^N j \overline{b}_j + \frac{\varepsilon^2}{2} \sum_{i = 1}^N j^2 - \frac{(1 + 2 s) \gamma}{2 s (1 - 2 s)} \sum_{1 \le i < j \le N} \left( |b_i - b_j|^{1 - 2 s} - |\overline{b}_i - \overline{b}_j|^{1 - 2 s} \right) \\
& \le \varepsilon N^2 \overline{b}_N + \varepsilon^2 N^3 - \frac{\gamma}{1 - 2 s} \left( |b_I - b_{I + 1}|^{1 - 2 s} - |\overline{b}_I - \overline{b}_{I + 1}|^{1 - 2 s} \right) \\
& = - \varepsilon^{1 - 2 s} \left( \frac{\gamma}{1 - 2 s} - \varepsilon^{2 s} N^2 \overline{b}_N - \varepsilon^{1 + 2 s} N^3 \right).
\end{align*}
From this, inequality~\eqref{Vb<Vbbar} immediately follows, provided we take~$\varepsilon$ small enough.

This concludes the proof of the existence and uniqueness of~$\beta$. The fact that~$\beta$ is the unique solution of~\eqref{betaeq} within~$\BB$ also ensures that~\eqref{betaodd} holds true, as, otherwise, the vector~$\tilde{\beta} \in \BB$ defined by~$\tilde{\beta}_i := - \beta_{N - i + 1}$ will provide a different solution of~\eqref{betaodd}. The proof of Proposition~\ref{explicitprop} is thus complete.
\end{proof}

For a general~$N$ the solution of~\eqref{betaeq} is not explicit. However, a straightforward computation gives that, for~$N = 2$, it holds
$$
\beta_2 = - \beta_1 = \frac{1}{2} \left( \frac{(1 + 2 s) \gamma}{s} \right)^{\frac{1}{1 + 2 s}},
$$
while, for~$N =3$,
$$
\beta_3 = - \beta_1 = \left( \frac{(1 + 2 s) (1 + 4^{-s}) \gamma}{2 s} \right)^{\frac{1}{1 + 2 s}} \quad \mbox{and} \quad \beta_2 = 0.
$$

The remainder of the section is devoted to show that~$\xi^0$ characterizes the long-time asymptotics of all solutions of~\eqref{systemforxi0}. To this aim, we first establish the next lemma, which provides a rough estimate on the diverging rate of two consecutive components of any solution of~\eqref{systemforxi0}.

\begin{lemma} \label{divratelem}
Let~$\eta \in C^1([1, +\infty); \R^N)$ be a solution of
\begin{equation} \label{systemforeta}
\dot{\eta}_i(t) = \frac{\gamma}{2 s} \sum_{j \ne i} \frac{\eta_i(t) - \eta_j(t)}{|\eta_i(t) - \eta_j(t)|^{1 + 2 s}} \quad \mbox{for all } t > 1 \mbox{ and } i = 1, \ldots, N,
\end{equation}
satisfying
\begin{equation} \label{eta0welldistanced}
\eta_1(1) < \eta_2(1) < \ldots < \eta_N(1).
\end{equation}
Then, there exists a constant~$C \ge 1$, depending only on~$s$,~$N$,~$\gamma$, and the initial datum~$\eta(1)$, such that
\begin{equation} \label{diest}
\frac{1}{C} \, t^{\frac{1}{1 + 2 s}} \le \eta_{i + 1}(t) - \eta_i(t) \le C \, t^{\frac{1}{1 + 2 s}} \quad \mbox{for all } t \ge 1 \mbox{ and } i = 1, \ldots, N - 1.
\end{equation}
\end{lemma}
\begin{proof}
First of all, we point out that condition~\eqref{eta0welldistanced} is preserved along the dynamics, meaning that
\begin{equation} \label{etatwelldistanced}
\eta_1(t) < \eta_2(t) < \ldots < \eta_N(t) \quad \mbox{for all } t \ge 1.
\end{equation}
Indeed, this is an immediate consequence of the fact that~$\eta$ is~$C^1$ and solves~\eqref{systemforxi0}.

We begin by establishing the right-hand inequality in~\eqref{diest}. Set~$d_i := \eta_{i + 1} - \eta_i$. By~\eqref{systemforeta}, we have
\begin{equation} \label{divratetech1}
\begin{aligned}
\dot{d_i} & = \frac{\gamma}{2 s} \left\{ \sum_{j \ne i + 1} \frac{\eta_{i + 1} - \eta_j}{|\eta_{i + 1} - \eta_j|^{1 + 2 s}} - \sum_{j \ne i} \frac{\eta_i - \eta_j}{|\eta_i - \eta_j|^{1 + 2 s}} \right\} \\
& = \frac{\gamma}{2 s} \left\{ \sum_{j = 1}^{i - 1} \left( \frac{1}{(\eta_{i + 1} - \eta_j)^{2 s}} - \frac{1}{(\eta_i - \eta_j)^{2 s}} \right) - \sum_{j = i + 2}^{N} \left( \frac{1}{(\eta_j - \eta_{i + 1})^{2 s}} - \frac{1}{(\eta_j - \eta_i)^{2 s}} \right) \right\} \\
& \quad + \frac{\gamma}{s} \frac{1}{(\eta_{i + 1} - \eta_i)^{2 s}}.
\end{aligned}
\end{equation}
Condition~\eqref{etatwelldistanced} yields that
\begin{align*}
\frac{1}{(\eta_{i + 1} - \eta_j)^{2 s}} - \frac{1}{(\eta_i - \eta_j)^{2 s}} & \le 0 \quad \mbox{for every } j \le i - 1, \\
\frac{1}{(\eta_j - \eta_{i + 1})^{2 s}} - \frac{1}{(\eta_j - \eta_i)^{2 s}} & \ge 0 \quad \mbox{for every } j \ge i + 2.
\end{align*}
Using this, we deduce from~\eqref{divratetech1} that
$$
\dot{d}_i \le \frac{\gamma}{s} \frac{1}{(\eta_{i + 1} - \eta_i)^{2 s}} = \frac{\gamma}{ s} \frac{1}{d_i^{2 s}},
$$
that is,~$d_i^{2 s} \dot{d}_i \le \gamma / s$. By integrating this inequality between~$1$ and~$t$, we easily get that
\begin{equation} \label{diestabove}
d_i(t) \le C \, t^{\frac{1}{1 + 2 s}} \quad \mbox{for all } t \ge 1 \mbox{ and } i = 1, \ldots, N,
\end{equation}
for some constant~$C \ge 1$ depending only on~$s$,~$\gamma$, and on the initial distance~$\eta_{i + 1}(1) - \eta_i(1)$.

The deduction of the lower bound in~\eqref{diest} is slightly more involved. The computation is inspired to the one performed in~\cite[Lemma~8.2]{FIM09}. Set~$d := \min \left\{ d_i : i = 1, \ldots, N - 1 \right\}$. For all but a countable number of~$t > 1$, the function~$d$ is differentiable at~$t$ and it holds~$\dot{d}(t) = \dot{d}_i(t)$ for some~$i = 1, \ldots, N$. Using~\eqref{systemforxi0}, we compute
$$
\dot{d} = \dot{d}_i = \frac{\gamma}{2 s} \left\{ \sum_{j = 1}^{i} \frac{1}{(\eta_{i + 1} - \eta_j)^{2 s}} - \sum_{j = i + 2}^N \frac{1}{(\eta_j - \eta_{i + 1})^{2 s}} - \sum_{j = 1}^{i - 1} \frac{1}{(\eta_i - \eta_j)^{2 s}} + \sum_{j = i + 1}^N \frac{1}{(\eta_j - \eta_i)^{2 s}} \right\}.
$$
Shifting indices in the second and third sum and rearranging, this becomes
\begin{equation} \label{dotd=}
\begin{aligned}
\dot{d} & = \frac{\gamma}{2 s} \left\{ \sum_{j = 2}^{i} \left( \frac{1}{(\eta_{i + 1} - \eta_j)^{2 s}} - \frac{1}{(\eta_i - \eta_{j - 1})^{2 s}} \right) + \sum_{j = i + 1}^{N - 1} \left( \frac{1}{(\eta_j - \eta_i)^{2 s}} - \frac{1}{(\eta_{j + 1} - \eta_{i + 1})^{2 s}} \right) \right\} \\
& \quad + \frac{\gamma}{2 s} \left\{ \frac{1}{(\eta_{i + 1} - \eta_1)^{2 s}} + \frac{1}{(\eta_N - \eta_i)^{2 s}} \right\}.
\end{aligned}
\end{equation}
By the minimality of~$d_i$, we have that~$d_i \le d_{\ell}$, i.e.,~$\eta_{i + 1} - \eta_{\ell + 1} \le \eta_i - \eta_{\ell}$ for all~$\ell = 1, \ldots, N - 1$. This gives
\begin{align*}
\frac{1}{(\eta_{i + 1} - \eta_j)^{2 s}} - \frac{1}{(\eta_i - \eta_{j - 1})^{2 s}} & \ge 0 \quad \mbox{for every } j \le i, \\
\frac{1}{(\eta_j - \eta_i)^{2 s}} - \frac{1}{(\eta_{j + 1} - \eta_{i + 1})^{2 s}} & \ge 0 \quad \mbox{for every } j \ge i + 1.
\end{align*}
Thus, both sums on the first line of~\eqref{dotd=} are non-negative and we deduce that
$$
\dot{d} \ge \frac{\gamma}{2 s} \left\{ \frac{1}{(\eta_{i + 1} - \eta_1)^{2 s}} + \frac{1}{(\eta_N - \eta_i)^{2 s}} \right\} \ge \frac{\gamma}{(\eta_N - \eta_1)^{2 s}}.
$$
By integrating this over the interval~$(1, t)$ and recalling~\eqref{diestabove}, we get
$$
d(t) \ge d(1) + \frac{1}{C} \int_1^t \tau^{- \frac{2s}{1 + 2 s}} \, d\tau \ge d(1) + \frac{1}{C} \, t^{- \frac{2 s}{1 + 2 s}} (t - 1),
$$
which easily leads us to the left-hand inequality in~\eqref{diest}. Thus, the lemma is proved.
\end{proof}

With the aid of Lemma~\ref{divratelem}, we can now prove the following result.

\begin{proposition} \label{xi0charprop}
Let~$\eta \in C^1([1, +\infty); \R^N)$ be a solution of~\eqref{systemforeta} satisfying~\eqref{eta0welldistanced} and
\begin{equation} \label{baryofeta=0}
\sum_{i = 1}^N \eta_i(1) = 0.
\end{equation}
Let~$\xi^0$ be the solution given by Proposition~\ref{explicitprop}. Then, it holds
\begin{equation} \label{xitilde0toxi0atinfty}
\lim_{t \rightarrow +\infty} \big| \eta(t) - \xi^0(t) \big| = 0.
\end{equation}
\end{proposition}
\begin{proof}
Let~$b_\eta(t)$ be the barycenter of~$\eta$ at time~$t$, namely
$$
b_\eta(t) := \frac{1}{N} \sum_{i = 1}^N \eta_i(t).
$$
Using that~$\eta$ solves~\eqref{systemforxi0} and anti-symmetrizing we deduce that
$$
\dot{b}_\eta = \frac{\gamma}{2N s} \sum_{i = 1}^N \sum_{j \ne i} \frac{\eta_i - \eta_j}{|\eta_i - \eta_j|^{1 + 2 s}} = \frac{\gamma}{2N s} \sum_{1 \le i < j \le N} \left( \frac{\eta_i - \eta_j}{|\eta_i - \eta_j|^{1 + 2 s}} + \frac{\eta_j - \eta_i}{|\eta_j - \eta_i|^{1 + 2 s}} \right) = 0.
$$
That is, the barycenter is preserved by the dynamics. By~\eqref{baryofeta=0}, this gives that
\begin{equation} \label{baryconst}
b_\eta(t) = 0 \quad \mbox{for all } t \ge 1.
\end{equation}

We now proceed to establish~\eqref{xitilde0toxi0atinfty}.
For~$i = 1, \ldots, N$, set~$d_i(t) := \eta_i(t) - \xi^0_i(t)$. Since both~$\eta$ and~$\xi^0$ are solutions of~\eqref{systemforeta}, using the mean value theorem we write, for all~$t > 1$ and~$i = 1, \ldots, N$,
\begin{equation} \label{dieq2}
\dot{d}_i(t) = \frac{\gamma}{2 s} \sum_{j \ne i} \left\{ \frac{\eta_i(t) - \eta_j(t)}{|\eta_i(t) - \eta_j(t)|^{1 + 2 s}} - \frac{\xi^0_i(t) - \xi^0_j(t)}{|\xi^0_i(t) - \xi^0_j(t)|^{1 + 2 s}} \right\} = - \gamma \sum_{j \ne i} \frac{d_i(t) - d_j(t)}{\left| P_{ij}(t) \right|^{1 + 2 s}},
\end{equation}
where
$$
P_{i j}(t) := \theta_{ij}(t) (\eta_i(t) - \eta_j(t)) + \left( 1 - \theta_{i j}(t) \right) (\xi^0_i(t) - \xi^0_j(t)),
$$
for some functions~$\theta_{i j}: (1, +\infty) \to [0, 1]$. Observe that, by Lemma~\ref{divratelem}, there exists a constant~$C > 0$ for which
\begin{equation} \label{Pijgrowth}
|P_{i j}(t)| \le C \, t^{\frac{1}{1 + 2 s}} \quad \mbox{for all } t > 1.
\end{equation}

Given~$t > 1$, let now~$m, M \in \{ 1, \ldots, N \}$ be such that~$d_m(t) \le d_j(t) \le d_M(t)$ for all~$j = 1, \ldots, N$.
By~\eqref{dieq2} and~\eqref{Pijgrowth}, we have
$$
\dot{d}_M(t) - \dot{d}_m(t) = - \gamma \sum_{j = 1}^N \left( \frac{d_M(t) - d_j(t)}{\left| P_{M j}(t) \right|^{1 + 2 s}} + \frac{d_j(t) - d_m(t)}{\left| P_{m j}(t) \right|^{1 + 2 s}} \right) \le - \frac{\varepsilon}{t} \left( d_M(t) - d_m(t) \right),
$$
for some constant~$\varepsilon > 0$. Set~$d(t) := d_M(t) - d_m(t)$ and observe that~$\dot{d}(t) = \dot{d}_M(t) - \dot{d}_m(t)$ for a.e.~$t > 1$. Therefore, we may rewrite the above inequality as
$$
\dot{d}(t) \le - \frac{\varepsilon}{t} \, d(t) \quad \mbox{for a.e.~} t > 1,
$$
which leads to
$$
d(t) \le d(1) \, t^{- \varepsilon} \quad \mbox{for all } t > 1.
$$
Furthermore, taking advantage of~\eqref{baryconst} and of the fact that the same holds for~$\xi^0$, thanks to~\eqref{betaodd}, we deduce that~$\sum_{i = 1}^N d_i(t) = 0$ and thus that~$d_m(t) \le 0 \le d_M(t)$ for a.e.~$t > 1$. Consequently, we get that
$$
|\eta_i(t) - \xi^0_i(t)| = |d_i(t)| \le d(t) \le d(1) \, t^{- \varepsilon} \quad \mbox{for all } t > 1 \mbox{ and } i = 1, \ldots, N,
$$
which yields~\eqref{xitilde0toxi0atinfty}.
\end{proof}

\section{On the asymptotic behavior of odd layer solutions} \label{layersec}

\noindent
The purpose of this section is to present some improvements on the known asymptotics of the layer solution~$w$, in the case of an even potential function~$W$.

Let~$W \in C^3(\R)$ be a non-degenerate double-well potential with zeroes at~$0$ and~$1$, i.e., a function satisfying
\begin{equation} \label{Wdoublewell}
\begin{aligned}
& W(r) > 0 \quad \mbox{for all } r \in (0, 1), \\
& W(0) = W(1) = W'(0) = W'(1) = 0, \\
& W''(0) = W''(1) > 0.
\end{aligned}
\end{equation}
Notice that a potential fulfilling the set of assumptions~\eqref{Wprop} satisfies in particular~\eqref{Wdoublewell}. However, the periodicity of~$W$ implied by~\eqref{Wprop} does not play any role in the results presented in this section and it is therefore not assumed here.

Under assumption~\eqref{Wdoublewell}, in~\cite{CS05,PSV13,CS15} it is showed that there exists a unique non-decreasing solution~$w$ of~\eqref{ellPNeq} satisfying~\eqref{wcond}. In the same papers, the authors proved that~$w \in C^2(\R; (0, 1))$ and that it satisfies
\begin{alignat}{3}
\label{wasympt}
\left| w(x) - \chi_{(0, +\infty)}(x) \right| & \le \frac{C}{(1 + |x|)^{2 s}} && \qquad \mbox{for all } x \in \R, \\
\label{w'asympt}
0 < w'(x) & \le \frac{C}{(1 + |x|)^{1 + 2 s}} && \qquad \mbox{for all } x \in \R.
\end{alignat}
The works~\cite{GM12,DPV15,DFV14} furthered the knowledge on the behavior of~$w$ at infinity by establishing the asymptotic expansion~\eqref{wasymptDPFV}, with~$\vartheta$ given by
\begin{equation} \label{thetadef}
\vartheta := \begin{cases}
4 s & \quad \mbox{if } s \in \left( 0, \frac{1}{6} \right), \\
\frac{1}{2} + s & \quad \mbox{if } s \in \left[ \frac{1}{6}, \frac{1}{2} \right), \\
1 + 2 s & \quad \mbox{if } s \in \left[ \frac{1}{2}, 1 \right).
\end{cases}
\end{equation}

The primary aim of this section is to provide a refinement of estimate~\eqref{wasymptDPFV} under the hypothesis that the potential~$W$ is even w.r.t.~$1/2$, that is,
\begin{equation} \label{Weven}
W(r) = W(1 - r) \quad \mbox{for all } r \in (0, 1).
\end{equation}
Note that, from~\eqref{Weven} and the fact that~$w$ is the only non-decreasing solution of~\eqref{ellPNeq} satisfying~\eqref{wcond}, we get that
\begin{equation} \label{wodd}
w(x) = 1 - w(- x) \quad \mbox{for all } x \in \R.
\end{equation}
Our main result is the following.

\begin{proposition} \label{uimprovasymptprop}
Assume that~$W \in C^3(\R)$ satisfies~\eqref{Wdoublewell} and~\eqref{Weven}. Then,
\begin{equation} \label{uimprovasympt}
\left| w(x) - \chi_{(0, +\infty)}(x) + \frac{1}{2 s W''(0)} \frac{x}{|x|^{1 + 2 s}} \right| \le \frac{C}{|x|^{4 s}} \quad \mbox{for all } x \in \R,
\end{equation}
for some constant~$C > 0$ depending only on~$s$ and~$W$.
\end{proposition}

Observe that estimate~\eqref{uimprovasympt} improves~\eqref{wasymptDPFV} for all~$s \in (0, 1)$, as can be seen by recalling definition~\eqref{thetadef} of~$\vartheta$---for~$s \in (0, 1/6]$ the two results are equivalent. Of course,~\eqref{uimprovasympt} holds under the parity assumption~\eqref{Weven} on~$W$, which is not needed for~\eqref{wasymptDPFV}.

The exponent~$4 s$ appearing on the right-hand side of~\eqref{uimprovasympt} is in general not optimal. For~$s = 1/2$, this can be seen by looking at the explicit solution~$w(x) = \frac{1}{2} + \frac{1}{\pi} \arctan(4 \pi x)$ corresponding to the potential~$W(r) = 1 - \cos(2 \pi r)$. When~$s = 1/2$, our techniques can actually be modified in order to obtain the sharp exponent~$3$ in place of~$4 s = 2$ in~\eqref{uimprovasympt}, for all~$C^4$ potentials~$W$ satisfying~\eqref{Wdoublewell} and~\eqref{Weven}. However, we do not know whether this is the case for a general~$s \ne 1/2$. See the forthcoming Remark~\ref{impros=12rmk} for more information on this.

In addition to~\eqref{uimprovasympt}, we have the following estimates for the decay of the second and third derivatives of~$w$.

\begin{proposition} \label{w''decayprop}
Assume that~$W \in C^{3, 1}(\R)$ satisfies~\eqref{Wdoublewell} and~\eqref{Weven}. Then,
\begin{equation} \label{w''decay}
|w''(x)| \le \frac{C}{(1 + |x|)^{2 + 2 s}} \quad \mbox{for all } x \in \R,
\end{equation}
for some constant~$C > 0$ depending only on~$s$ and~$W$. Moreover, if~$W \in C^{4, 1}(\R)$, then,
\begin{equation} \label{w'''decay}
|w'''(x)| \le \frac{C}{(1 + |x|)^{1 + 2 s}} \quad \mbox{for all } x \in \R.
\end{equation}
\end{proposition}

We believe~\eqref{w''decay} to be optimal, while~\eqref{w'''decay} is certainly not---again, consider the explicit solution for~$s = 1/2$. The ultimate reason for the different behaviors of~\eqref{w''decay} versus~\eqref{w'''decay} is that~$w''$ is odd, whereas~$w'''$ is even. The function~$w''$ can be thus estimated at infinity via the use of odd barriers, whose fractional Laplacians of order~$2 s$ decay faster than~$|x|^{- 1 - 2 s}$---which is the typical decay rate of~$(-\Delta)^s \eta$, for a general smooth and compactly supported function~$\eta$ in one dimension, that is, the generic decay rate of the~$s$-Laplacian. Conversely, for non-symmetric or even functions, it is unclear how to go beyond this rate and thus we only get estimate~\eqref{w'''decay} for~$w'''$.

The proofs of Propositions~\ref{uimprovasymptprop} and~\ref{w''decayprop} follow the general strategy developed in~\cite{GM12,DPV15}---apart from the modifications needed for our improved estimates, made possible by the parity of~$W$. In order not to break the flow of the paper, we postpone the arguments to Appendix~\ref{layerapp}.

\section{A few general facts about fractional parabolic equations} \label{solsec}

\noindent
Here, we mostly collect some known results about bounded solutions to semilinear parabolic equations driven by the fractional Laplacian. These results will be then frequently used throughout the remainder of the paper. We largely follow~\cite[Section~2]{CR13} although our treatment is essentially self-contained.

Denote with~$\widecheck{\textcolor{white}{a}}$ the inverse Fourier transform in~$\R$ and set
$$
\widetilde{p}(x, t) := \left( e^{- t \, |\,\cdot\,|^{2 s}} \right) \!\! \widecheck{\textcolor{white}{\big)}} (x) = \frac{1}{2 \pi} \int_{\R} e^{i x \xi - t |\xi|^{2 s}} d\xi.
$$
Up to a rescaling in the time variable~$t$, the function~$\widetilde{p}$ is the heat kernel corresponding to the operator~$(-\Delta)^s$ defined in~\eqref{fracLap}. Indeed, it is well-known that there exists a unique~$\lambda_s > 0$, depending only on~$s$, such that the function~$p$ defined by~$p(x, t) := \widetilde{p}(x, \lambda_s t)$ satisfies the following properties:
\begin{enumerate}[label=$(\mbox{P}\arabic*)$,leftmargin=*]
\item \label{preg} $p \in C^\infty(\R\times (0, +\infty))$.
\item \label{pscalor} $\partial_t p + (- \Delta)^s p = 0$ in~$\R \times (0, +\infty)$.
\item \label{pmass=1} $\displaystyle \int_\R p(x, t) \, dx = 1$ for all~$t > 0$.
\item \label{psemi} $p(\cdot, t) \ast p(\cdot, \tau) = p(\cdot, t + \tau)$ for all~$t, \tau > 0$.
\item \label{pscales} $p(x, t) = t^{-1/2 s} p(t^{- 1/2s} x, 1)$ for all~$x \in \R, \, t > 0$.
\item \label{pbounds} There exists a constant~$\Lambda_s \ge 1$, depending only on~$s$, for which
$$
\frac{\Lambda_s^{-1} \, t}{(x^2 + t^{1/s})^{\frac{1 + 2 s}{2}}} \le p(x, t) \le \frac{\Lambda_s \, t}{(x^2 + t^{1/s})^{\frac{1 + 2 s}{2}}} \quad \mbox{for all } x \in \R, \, t > 0.
$$
\end{enumerate}
Properties~\ref{preg}-\ref{pscales} can all be easily deduced from the definition of~$p$, while~\ref{pbounds} follows from~\cite{P23}---see also~\cite[Theorem~2.1]{BG60}.

Given~$t > 0$ and~$u \in L^\infty(\R)$, we define
$$
\T_t[u](x) := \left( p(\cdot, t) \ast u \right)(x) = \int_\R p(x - y, t) u(y) \, dy \quad \mbox{for all } x \in \R.
$$
By property~\ref{pmass=1},~$\T_t$ maps~$L^\infty(\R)$ to~$L^\infty(\R)$. Also,~\ref{psemi} gives that~$\T_t$ is a semigroup in~$L^\infty(\R)$.

Let~$t_0 \in \R$ and~$G = G(x, t, u): \R \times (t_0, +\infty) \times \R \to \R$ be a measurable function satisfying
\begin{equation} \label{gbounded}
\begin{gathered}
\mbox{for every } M > 0 \mbox{ and } T > t_0, \mbox{ there exists } C_{T, M} > 0 \mbox{ such that }\\
|G(x, t, u)| \le C_{T, M} \mbox{ for all } x \in \R, \, t \in (t_0, T], \, u \in [-M, M]
\end{gathered}
\end{equation}
and
\begin{equation} \label{gLip}
\mbox{there is } L > 0 \mbox{ such that } |G(x, t, u) - G(x, t, v)| \le L |u - v| \mbox{ for all } x \in \R, \, t > t_0, \, u, v \in \R.
\end{equation}
Sometimes, we will adopt the notation~$G[u](x, t) := G(x, t, u(x, t))$. In the following, we will be particularly interested in~$G$'s of the form
$$
G(x, t, u) = f(x, t) + b(x, t) u, 
$$
for some~$f, b \in L^\infty(\R \times (t_0, +\infty))$.

Given~$t_1 \in (t_0, +\infty]$ and~$u_0 \in L^\infty(\R)$, we consider the initial value problem
\begin{equation} \label{Gprob}
\begin{cases}
\partial_t u + (-\Delta)^s u = G[u] & \quad \mbox{in } \R \times (t_0, t_1),\\
u = u_0 & \quad \mbox{on } \R \times \{ t_0 \}.
\end{cases}
\end{equation}
The notion of solution of~\eqref{Gprob} that we will mostly consider is presented in the following definition.

\begin{definition} \label{strongsoldef}
Assume~$t_1 < +\infty$. We say that~$u \in L^\infty(\R \times (t_0, t_1))$ is a \emph{mild solution} of~\eqref{Gprob} if
\begin{equation} \label{ustrongdef}
u(x, t) = \T_{t - t_0}[u_0](x) + \int_{t_0}^t \T_{t - \tau} [G[u](\cdot, \tau)] (x) \, d\tau \quad \mbox{for a.e.~} x \in \R, \, t \in (t_0, t_1).
\end{equation}
When~$t_1 = +\infty$, we only require that~$u \in L^\infty_\loc([t_0, +\infty); L^\infty(\R))$ in addition to identity~\eqref{ustrongdef}.
\end{definition}

It is not hard to verify that, when~$G$ and~$u$ are sufficiently smooth,~$u$ is a mild solution of~\eqref{Gprob} if and only if it solves it in the pointwise sense. Without assuming regularity a priori (apart from the boundedness of~$u$), Definition~\ref{strongsoldef} still provides a well-defined notion of solution of~\eqref{Gprob}.

The next proposition provides some basic local- and global-in-time regularity estimates for mild solutions. Here, it does not harm the generality to assume the right-hand side to be a function~$f$ of~$x$ and~$t$ only. We also adopt the notation~$a \wedge b := \min \{ a, b\}$.

\begin{proposition} \label{regforstrongprop}
Let~$f: \R \times (t_0, t_1) \to \R$ be a bounded function and~$u$ be a mild solution of
$$
\begin{cases}
\partial_t u + (-\Delta)^s u = f & \quad \mbox{in } \R \times (t_0, t_1),\\
u = u_0 & \quad \mbox{on } \R \times \{ t_0 \}.
\end{cases}
$$
Assume that~$t_1 < +\infty$ and let~$T \ge t_1 - t_0$. The following estimates are valid:
\begin{enumerate}[label=$(\alph*)$,leftmargin=*]
\item \label{intreg} If~$u_0 \in L^\infty(\R)$, then, for every~$t_\star \in (t_0, t_1)$,~$\sigma \in (0, s)$, and~$\theta \in (0, 1)$, it holds
$$
\sup_{t \in (t_\star, t_1)} \| u(\cdot, t) \|_{C^{2 \sigma}(\R)} + \sup_{x \in \R} \| u(x, \cdot) \|_{C^\theta(t_\star, t_1)}
\le C_{T, \, t_\star} \Big( \| f \|_{L^\infty(\R \times (t_0, t_1))} + \| u_0 \|_{L^\infty(\R)} \Big),
$$
for some constant~$C_{t_\star} > 0$ depending only on~$s$,~$\sigma$,~$\theta$,~$T$, and~$t_\star$. For~$s \ne 1/2$, we can take~$\sigma = s$.
\item \label{globreg} If~$u_0 \in C^\alpha(\R)$ for some~$\alpha \in (0, 2 s \wedge 1)$, then it holds
$$
\sup_{t \in (t_0, t_1)} \| u(\cdot, t) \|_{C^{\alpha}(\R)} + \sup_{x \in \R} \| u(x, \cdot) \|_{C^{\frac{2s \wedge 1}{2s} \alpha}(t_0, t_1)}
%\| u \|_{C_{x, t}^{2s \wedge \alpha, (2s \wedge 1) \frac{\alpha}{2s}}(\R \times [0, T])}
\le C_T \Big( \| f \|_{L^\infty(\R \times (t_0, t_1))} + \| u_0 \|_{C^\alpha(\R)} \Big),
$$
for some constant~$C > 0$ depending only on~$s$,~$\alpha$, and~$T$.
\end{enumerate}
\end{proposition}

Proposition~\ref{regforstrongprop} is surely well-known to the experts and can be proved via rather straightforward, albeit lengthy, computations. We postpone the argument to Appendix~\ref{estforstrongapp}.

We now address the solvability of problem~\eqref{Gprob}. To this end, it is convenient to consider the map~$\NNN_{G, u_0}: L^\infty(\R\times (t_0, t_1)) \to L^\infty(\R \times (t_0, t_1))$ defined, for~$u_0 \in L^\infty(\R)$ and~$u \in L^\infty(\R \times (t_0, t_1))$, by
\begin{equation} \label{Ngphidef}
\NNN_{G, u_0}[u](x, t) := \T_{t - t_0}[u_0](x) + \int_{t_0}^t \T_{t - \tau} [G[u](\cdot, \tau)] (x) \, d\tau \quad \mbox{for a.e.~} x \in \R, \, t \in (t_0, t_1).
\end{equation}
We remark that the boundedness of~$\NNN_{G, u_0}[u]$ is a consequence, in particular, of hypothesis~\eqref{gbounded} on~$G$. The following result holds true.

\begin{proposition} \label{existinLinftyprop}
Let~$G: \R \times (t_0, +\infty) \times \R \to \R$ be a function satisfying~\eqref{gbounded} and~\eqref{gLip}. Then, given any~$u_0 \in L^\infty(\R)$, there exists a unique mild solution~$u$ to problem~\eqref{Gprob}, with~$t_1 = +\infty$.
\end{proposition}
\begin{proof}
Of course, the claim is equivalent to the existence and uniqueness of a mild solution~$u$ to problem~\eqref{Gprob} for every~$t_1 \in (t_0, +\infty)$.

Given such a~$t_1$, for~$u \in L^\infty(\R \times (t_0, t_1))$ being a mild solution of~\eqref{Gprob} is equivalent to being a fixed point for the map~$\NNN := \NNN_{G, u_0}$. The existence and uniqueness of such a fixed point follows, when~$T := t_1 - t_0$ is small, from the fact that~$\NNN$ is a contraction. Indeed, by~\eqref{gLip} and property~\ref{pmass=1},
\begin{align*}
\big| \NNN[v](x, t) - \NNN[w](x, t) \big| & \le \int_{t_0}^t \left( \int_\R p(x - y, t - \tau) \left| G(y, \tau, v(y, \tau)) - G(y, \tau, w(y, \tau)) \right| dy \right) d\tau \\
& \le L T \| v - w \|_{L^\infty(\R \times (t_0, t_1))},
\end{align*}
for a.e.~$x \in \R$,~$t \in (t_0, t_1)$, and all~$v, w \in L^\infty(\R \times (t_0, t_1))$. Hence,
$$
\| \NNN[v] - \NNN[w] \|_{L^\infty(\R \times (t_0, t_1))} \le \frac{3}{4} \, \| v - w \|_{L^\infty(\R \times (t_0, t_1))} \quad \mbox{for all } v, w \in L^\infty(\R\times(t_0, t_1)),
$$
provided~$T \le 3/(4L)$.

The case of a general~$T$ follows by applying the previous argument iteratively in the time intervals~$(t_0, t_0 + (2L)^{-1}), (t_0 + (2L)^{-1}, t_0 + L^{-1}), \ldots, (t_0 + (2 L)^{-1} (\lceil 2 L T \rceil - 1), t_1)$, by restricting the mild solution found in an interval to the right endpoint of said interval and using this function as the initial datum for the next interval. Note that this can be done thanks to the semigroup properties of~$\T_t$. In addition, we took advantage of the fact that the temporal slices of mild solutions are well-defined bounded continuous functions of~$x$, since the solutions are bounded continuous functions of~$(x, t)$, thanks to Proposition~\ref{regforstrongprop}\ref{intreg}.
\end{proof}

In the sequel, we will need to consider products of mild solutions with positive factors~$A(t)$ depending solely on the time variable. The next result establishes that these new functions are themselves mild solutions of a different initial value problem.
%Note that this fact has already been observed in~\cite{CR13} for the case~$A(t) = e^{a t}$.

\begin{lemma} \label{Arescalelem}
Let~$G: \R \times (t_0, +\infty) \times \R \to \R$ be a function satisfying~\eqref{gbounded} and~$u$ be a mild solution of~\eqref{Gprob}. Then, given a positive~$A \in C^1([t_0, t_1])$, the function~$\widetilde{u}(x, t) := A(t) u(x, t)$ is a mild solution of
\begin{equation} \label{Gtildeprob}
\begin{cases}
\partial_t \widetilde{u} + (-\Delta)^s \widetilde{u} = \widetilde{G}[\widetilde{u}] & \quad \mbox{in } \R \times (t_0, t_1),\\
\widetilde{u} = \widetilde{u}_0 & \quad \mbox{on } \R \times \{ t_0 \},
\end{cases}
\end{equation}
with~$\widetilde{u}_0 := A(t_0) u_0$ and
\begin{equation} \label{gtildedef}
\widetilde{G}(x, t, \widetilde{u}) := A(t) G(x, t, A(t)^{-1} \widetilde{u}) + A'(t) A(t)^{-1} \widetilde{u}.
\end{equation}
\end{lemma}
\begin{proof}
We follow the argument of~\cite[Subsection~2.3]{CR13}, where the result is established for the case of~$A(t)$ being an exponential factor.

As~$u$ is a mild solution of~\eqref{Gprob}, we have that~$u(\cdot, \tau) = \T_{\tau - t_0}[u_0] + \int_{t_0}^\tau \T_{\tau - \sigma}[G[u](\cdot, \sigma)] \, d\sigma$ in~$\R$ for a.e.~$\tau \in (t_0, t_1)$. By applying the operator~$\T_{t - \tau}$ to both sides of the previous identity and taking advantage of its semigroup properties, we find that~$\T_{t - \tau}[u(\cdot, \tau)] = \T_{t - t_0}[u_0] + \int_{t_0}^\tau \T_{t - \sigma}[G[u](\cdot, \sigma)] \, d\sigma$. Hence, after an integration by parts and another application of~\eqref{ustrongdef}, we get
\begin{align*}
\int_{t_0}^t \T_{t - \tau}[A'(\tau) u(\cdot, \tau)] \, d\tau & = \T_{t - t_0}[u_0] \int_{t_0}^t A'(\tau) \, d\tau + \int_{t_0}^t A'(\tau) \left( \int_{t_0}^\tau \T_{t - \sigma}[G[u](\cdot, \sigma)] \, d\sigma \right) d\tau \\
& = \T_{t - t_0}[u_0] \left( A(t) - A(t_0) \right) \\
& \quad + A(t) \int_{t_0}^t \T_{t - \sigma}[G[u](\cdot, \sigma)] \, d\sigma - \int_{t_0}^t A(\tau) \T_{t - \tau}[G[u](\cdot, \tau)] \, d\tau \\
& = A(t) u(\cdot, t) - \T_{t - t_0}[A(t_0) u_0] - \int_{t_0}^t \T_{t - \tau}[A(\tau) G[u](\cdot, \tau)] \, d\tau,
\end{align*}
in~$\R$, for a.e.~$t \in (t_0, t_1)$. That is,~$\widetilde{u}$ is a mild solution of~\eqref{Gtildeprob}.
\end{proof}

%\begin{remark} \label{exprescalermk}
%If~$u$ is a mild solution of~\eqref{Gprob} and~$a \in \R$, then the function~$\widetilde{u}(x, t) := e^{a t} u(x, t)$ is a mild solution of~\eqref{Gprob} with the same~$u_0$ as initial datum and with~$g$ replaced by
%\begin{equation} \label{gtildedef}
%\widetilde{g}(x, t, \widetilde{u}) := e^{a t} g(x, t, e^{- a t} \widetilde{u}) + a \widetilde{u}.
%\end{equation}
%Notice that, if~$g$ satisfies properties~\eqref{gbounded} and~\eqref{gLip}, then~$\widetilde{g}$ satisfies them too, possibly with different constants~$C_{T, M}$ and~$L$.
%\end{remark}

As a first application of the previous result, we have the following comparison principle.

\begin{proposition} \label{compprincforstrongprop}
Let~$G_1, G_2: \R \times (t_0, +\infty) \times \R \to \R$ be two functions satisfying~\eqref{gbounded},~\eqref{gLip}, and
\begin{equation} \label{g1leg2}
G_1(x, t, u) \le G_2(x, t, u) \quad \mbox{for all } x \in \R, \, t > t_0, \, u \in \R.
\end{equation}
For~$i = 1, 2$, let~$u_{i, 0} \in L^\infty(\R)$ be such that~$u_{1, 0} \le u_{2, 0}$ in~$\R$. Let~$u_i$ be the mild solution of
$$
\begin{cases}
\partial_t u_i + (-\Delta)^s u_i = G_i(\cdot, \cdot, u_i) & \quad \mbox{in } \R \times (t_0, t_1),\\
u_i = u_{i, 0} & \quad \mbox{on } \R \times \{ t_0 \},
\end{cases}
$$
for any~$t_1 > t_0$. Then,~$u_1 \le u_2$ in~$\R \times (t_0, t_1)$.
\end{proposition}
\begin{proof}
First of all, we remark that it suffices to prove the result when~$T:= t_1 - t_0$ is a small number. Indeed, if this is not the case, we divide~$(t_0, t_1)$ into a finite number of adjacent subintervals of sufficiently small length and apply the result iteratively in each subinterval.

Let now~$L_i$ be the Lipschitz constant of~$G_i$ (as given by property~\eqref{gLip}) and set~$a := \max \{ L_1, L_2 \}$. Consider the functions~$\widetilde{G}_i$ defined as in~\eqref{gtildedef}, corresponding to~$G_i$ and to the function~$A(t) = e^{a (t - t_0)}$. Observe that~$\widetilde{G}_i$ satisfies properties~\eqref{gbounded} and~\eqref{gLip}, possibly for different constants~$C_{T, M, i}$ and~$L_i$. Moreover, it is immediate to see that~$\widetilde{G}_i$ is monotone non-decreasing in the last component.

Consider the functions~$\widetilde{u}_i := e^{a (t - t_0)} u_i$. As per Lemma~\ref{Arescalelem},~$\widetilde{u}_i$ is the mild solution of problem~\eqref{Gprob} with initial datum~$u_{i, 0}$ and right-hand side~$\widetilde{G}_i$. Recalling how the existence of mild solution has been established in Proposition~\ref{existinLinftyprop} and the proof of the contraction lemma, assuming~$T \le 3/(8a)$ we have that~$\widetilde{u}_i$ may be obtained as the limit
\begin{equation} \label{uiasalimit}
\widetilde{u}_i = \lim_{k \rightarrow +\infty} (\NNN_i)^k[u_{i, 0}] \in L^\infty(\R \times (t_0, t_1)),
\end{equation}
where we wrote~$\NNN_i := \NNN_{\widetilde{G}_i, u_{i, 0}}$ and identified~$u_{i, 0}$ with its constant extension~$u_{i, 0}(\cdot, t) := u_{i, 0}$ for~$t > 0$. From definition~\eqref{Ngphidef}, the positivity of~$p$, the monotonicity of, say,~$\widetilde{G}_1$, and hypothesis~\eqref{g1leg2}, it follows that~$\NNN_1[v] \le \NNN_2[w]$ in~$\R \times (t_0, t_1)$ if~$v \le w$ in~$\R \times (t_0, t_1)$. By applying this relation iteratively and recalling~\eqref{uiasalimit}, we conclude that~$\widetilde{u}_1 \le \widetilde{u}_2$ and, consequently,~$u_1 \le u_2$ in~$\R \times (t_0, t_1)$.
\end{proof}

The next two lemmas provide us with a couple of barriers that will be used later on in conjunction with the comparison principle of Proposition~\ref{compprincforstrongprop}. Recall definition~\eqref{Phidef} of the weight function~$\Phi$.

\begin{lemma} \label{etabound0lem}
For~$m > 0$ and~$T \ge 1$, let~$\etaini_T = \etaini$ be the mild solution of the problem
\begin{equation} \label{etaprob0}
\begin{cases}
\partial_t \etaini + (-\Delta)^s \etaini + m \etaini = 0 & \quad \mbox{in } \R \times (T, +\infty) \\
\etaini = \Phi & \quad \mbox{on } \R \times \{ T \}.
\end{cases}
\end{equation}
Then,
$$
0 < \etaini \le C \Phi \quad \mbox{in } \R \times [T, +\infty),
$$
for some constant~$C > 0$ depending only on~$s$ and~$m$.
\end{lemma}
\begin{proof}
By Lemma~\ref{Arescalelem}---applied here with~$A(t) = e^{m (t - t_0)}$---and Definition~\ref{strongsoldef}, we have that
$$
\etaini(x, t) = e^{- m (t - T)} \int_\R p(y, t - T) \Phi(x - y, T) \, dy.
$$
From this, property~\ref{pmass=1}, and~\eqref{Phidef}, it immediately follows that
$$
0 < \etaini(x, t) \le T^{- \frac{2 s}{1 + 2 s}} e^{- m (t - T)} \le C \, t^{- \frac{2 s}{1 + 2 s}} \quad \mbox{for all } x \in \R, \, t \ge T.
$$
Note that the last inequality can be obtained for instance by considering separately the two cases of~$t \in [T, 2 T]$ and~$t \ge 2 T$.

To finish the proof we only need to show that~$\eta^{(1)}$ is controlled by~$|x|^{- 2 s}$. By~\ref{pbounds} and again~\ref{pmass=1},
\begin{align*}
\etaini(x, t) & \le e^{- m (t - T)} \left\{ \int_{|y - x| \ge \frac{|x|}{2}} \frac{p(y, t - T)}{|x - y|^{2 s}} \, dy + C \, \frac{t - T}{T^{\frac{2 s}{1 + 2 s}}} \int_{|y - x| < \frac{|x|}{2}} \frac{dy}{y^{1 + 2 s}} \right\} \\
& \le C e^{- m (t - T)} \left\{ |x|^{-2 s} + \frac{t - T}{T^{\frac{2 s}{1 + 2 s}}} |x|^{- 2 s} \right\} \le \frac{C e^{- m (t - T)}}{|x|^{2 s}} \Big\{ 1 + (t - T) \Big\} \le \frac{C}{|x|^{2 s}},
\end{align*}
where the last inequality is a consequence of the boundedness of the function~$\sigma \mapsto e^{-m \sigma} \{ 1 + \sigma \}$ in~$[0, +\infty)$. The proof is thus concluded.
\end{proof}

\begin{lemma} \label{etaboundlem}
For~$m > 0$, let~$\etadis$ be the mild solution of the problem
\begin{equation} \label{etaprob}
\begin{cases}
\partial_t \etadis + (-\Delta)^s \etadis + m \etadis = \Phi & \quad \mbox{in } \R \times (1, +\infty) \\
\etadis = 0 & \quad \mbox{on } \R \times \{ 1 \}.
\end{cases}
\end{equation}
Then,
\begin{equation} \label{etabound}
0 < \etadis \le C \Phi \quad \mbox{in } \R \times [1, +\infty),
\end{equation}
for some constant~$C > 0$ depending only on~$s$ and~$m$.
\end{lemma}
\begin{proof}
Taking advantage of Lemma~\ref{Arescalelem}, with~$A(t) = e^{m t}$, we represent the solution~$\etadis$ as
\begin{equation} \label{etarepres}
\etadis(x, t) = \int_0^{t - 1} \int_{\R} e^{- m \tau} p(y, \tau) \Phi(x - y, t - \tau) \, dy d\tau.
\end{equation}
Note that from this identity we immediately infer the positivity of~$\etadis$.

We claim that
\begin{alignat}{3}
\label{etaboundclaim1}
\etadis(x, t) & \le C \, t^{- \frac{2 s}{1 + 2 s}} && \qquad \mbox{for all } x \in \R, \, t \ge 1, \\
\label{etaboundclaim2}
\etadis(x, t) & \le C |x|^{- 2 s} && \qquad \mbox{for all } |x| \ge t^{\frac{1}{1 + 2 s}}, \, t \ge 1,
\end{alignat}
for some costant~$C \ge 1$ depending only on~$s$ and~$m$. Observe that, in view of~\eqref{Phidef}, estimate~\eqref{etabound} would then be proved.

We begin by establishing the time decay estimate~\eqref{etaboundclaim1}. In view of~\eqref{Phidef},~\eqref{etarepres}, and~\ref{pmass=1}, we have
$$
\etadis(x, t) \le \int_0^{t - 1} \frac{e^{- m \tau}}{(t - \tau)^{\frac{2 s}{1 + 2 s}}} \, d\tau.
$$
We split the last integral between the two domains~$(0, (t - 1)/2)$ and~$((t - 1)/2, t - 1)$. We get
$$
\int_0^{t - 1} \frac{e^{- m \tau}}{(t - \tau)^{\frac{2 s}{1 + 2 s}}} \, d\tau \le \left( \frac{t + 1}{2} \right)^{- \frac{2 s}{1 + 2 s}} \int_0^{\frac{t - 1}{2}} e^{- m \tau} \, d\tau + \int_{\frac{t - 1}{2}}^{t - 1} e^{- m \tau}  \, d\tau \le C \left( t^{- \frac{2 s}{1 + 2 s}} + e^{- \frac{m}{2} t} \right),
$$
and thus~\eqref{etaboundclaim1} follows.

Now we deal with~\eqref{etaboundclaim2}. Recalling~\eqref{Phidef}, it is easy to see that
$$
\Phi(x, t) \le \frac{4^s}{(1 + |x|)^{2 s}} \quad \mbox{for all } x \in \R, \, t > 1.
$$
By using this in combination with~\eqref{etarepres} and~\ref{pbounds}, we obtain
\begin{equation} \label{etaboundtech1}
\etadis(x, t) \le C \int_0^{t - 1} \left\{ \tau^{- \frac{1}{2 s}} \int_{|y| \le \tau^{\frac{1}{2 s}}} \frac{dy}{(1 + |x - y|)^{2 s}} + \tau \int_{|y| > \tau^{\frac{1}{2 s}}} \frac{ dy}{(1 + |x - y|)^{2 s} |y|^{1 + 2 s}} \right\} \frac{d\tau}{e^{m \tau}}.
\end{equation}
To estimate these two integrals, we distinguish between the two cases~$|y - x| \le |x|/2$ and~$|y - x| > |x| / 2$. In the first situation, we simply use that~$1 + |x - y| \ge 1$ and~$|y| \ge |x| / 2$, while in the second we take advantage of the inequality~$1 + |x - y| \ge |x|/2$. We get
$$
\int_{|y| \le \tau^{\frac{1}{2 s}}} \frac{dy}{(1 + |x - y|)^{2 s}} \le \int_{\substack{|y| \le \tau^{\frac{1}{2 s}} \\ |y - x| \le \frac{|x|}{2}}} dy + \frac{4^{s}}{|x|^{2 s}} \int_{\substack{|y| \le \tau^{\frac{1}{2 s}} \\ |y - x| > \frac{|x|}{2}}} dy \le 2 \left( \tau^{\frac{1}{2 s}} - \frac{|x|}{2} \right)_+ + \frac{C \, \tau^{\frac{1}{2 s}}}{|x|^{2 s}}
$$
and
\begin{align*}
\int_{|y| > \tau^{\frac{1}{2 s}}} \frac{dy}{(1 + |x - y|)^{2 s} |y|^{1 + 2 s}} & \le \int_{\substack{|y| > \tau^{\frac{1}{2 s}} \\ |y - x| \le \frac{|x|}{2}}} \frac{dy}{|y|^{1 + 2 s}} + \frac{4^{s}}{|x|^{2 s}} \int_{\substack{|y| > \tau^{\frac{1}{2 s}} \\ |y - x| > \frac{|x|}{2}}} \frac{dy}{|y|^{1 + 2 s}} \\
& \le C \left( \min \left\{ \frac{4^s}{|x|^{2 s}}, \frac{1}{\tau} \right\} + \frac{1}{|x|^{2 s} \tau} \right).
\end{align*}
By plugging these last two estimates into~\eqref{etaboundtech1}, we obtain
\begin{align*}
\etadis(x, t) & \le C \left( \frac{1}{|x|^{2 s}} \int_0^{\frac{|x|^{2 s}}{4^s}} e^{-m \tau} \tau \, d\tau + \int_{\frac{|x|^{2 s}}{4^s}}^{+\infty} e^{- m \tau} \, d\tau + \frac{1}{|x|^{2 s}} \int_{0}^{t - 1} e^{- m \tau} d\tau \right) \\
& \le C \left( |x|^{- 2 s} + e^{- \frac{m}{4^s} |x|^{2 s}} + |x|^{-2 s} \right).
\end{align*}
Claim~\eqref{etaboundclaim2} is then also true.
\end{proof}

We conclude the section with a result that deals with the linearization of equation~\eqref{maineq} at the layer solution~$w$. It provides a decay estimate (in an~$L^2$ sense) for all mild solutions whose time slices are~$L^2$-orthogonal to the derivative of~$w$. Its proof heavily relies on the non-degeneracy of the stationary Peierls-Nabarro equation~\eqref{ellPNeq}, established (in a quantitative form) in~\cite[Section~5]{DPV15}.

\begin{lemma} \label{L2lem}
Let~$u$ be a mild solution of
$$
\begin{cases}
\partial_t u + (-\Delta)^s u + W''(w) u = 0 & \quad \mbox{in } \R \times (t_0, t_1),\\
u=u_0 & \quad \mbox{on } \R \times \{ t_0 \},
\end{cases}
$$
such that
\begin{equation} \label{uinL2}
u \in C^0((t_0, t_1); L^2(\R))
\end{equation}
and
\begin{equation} \label{L2uortw'}
\int_{\R} u(x, t) w'(x) \, dx = 0 \quad \mbox{for all } t \in (t_0, t_1).
\end{equation}
Then, there exists a constant~$\lambda > 0$, depending only on~$s$ and~$W$, such that the function
$$
t \longmapsto e^{\lambda t} \int_\R u(x, t)^2 \, dx
$$
is non-increasing in~$(t_0, t_1)$.
\end{lemma}
\begin{proof}
First of all, in view of~\cite[Lemma~5.3]{DPV15}, there exists a constant~$\lambda > 0$, depending only on~$s$ and~$W$, for which
\begin{equation} \label{DPVnondeg}
[v]_{H^s(\R)}^2 + 2 \int_\R W''(w(x)) v(x)^2 \, dx \ge \lambda \, \| v \|_{L^2(\R)}^2 \mbox{ for all } v \in L^2(\R) \mbox{ s.t.~} \!\! \int_\R v(x) w'(x) \, dx = 0.
\end{equation}
Given an open interval~$I \subseteq \R$, we denote here with~$[\, \cdot \,]_{H^s(I)}$ the Gagliardo seminorm of the fractional Sobolev space~$H^s(I)$, that is, we write
$$
[v]_{H^s(I)} := \left( \iint_{I^2} \frac{|v(x) - v(y)|^2}{|x - y|^{1 + 2 s}} \, dx dy \right)^{\! \frac{1}{2}} \quad \mbox{for all } v \in L^2(I).
$$
To be rigorous,~\eqref{DPVnondeg} is proved in~\cite{DPV15} only for the case~$s \in (1/2, 1)$. However, the result is actually true for any~$s \in (0, 1)$, as noted, for instance, at the beginning of the proof of~\cite[Theorem~9.1]{DFV14}.\footnote{We also point out a misprint occurring in~\cite{DPV15} (and~\cite{PSV13}). There, it is stated that the layer solution~$w$ of~\eqref{ellPNeq} is (formally) a local minimizer of the functional
$$
v \longmapsto \frac{1}{2} \, [v]_{H^s(\R)}^2 + \int_\R W(v(x)) \, dx,
$$
with respect to compactly supported perturbations. A direct inspection shows that the factor in front of~$[v]_{H^s(\R)}^2$ should be~$1/4$ instead of~$1/2$. As a result, formulas~(5.5)-(5.7) in~\cite{DPV15} should be replaced by~\eqref{DPVnondeg} here.}

With~$\lambda$ given by~\eqref{DPVnondeg}, we consider the function~$\widetilde{u}(x, t) := e^{\frac{\lambda}{2} (t - t_0)} u(x, t)$. By Lemma~\ref{Arescalelem}, we know that~$\widetilde{u}$ is a mild solution of
\begin{equation} \label{g=-W''ueq}
\begin{cases}
\partial_t \widetilde{u} + (-\Delta)^s \widetilde{u} + \left( W''(w) - \frac{\lambda}{2} \right) \widetilde{u} = 0 & \quad \mbox{in } \R \times (t_0, t_1),\\
\widetilde{u}=u_0 & \quad \mbox{on } \R \times \{ t_0 \},
\end{cases}
\end{equation}
In addition,~$\widetilde{u}$,~$\partial_t \widetilde{u}$, and~$(-\Delta)^s \widetilde{u}$ are all continuous functions on~$\R \times (t_0, t_1)$ and therefore~$\widetilde{u}$ is a pointwise solution of~\eqref{g=-W''ueq}---see, e.g.,~\cite[Corollary~3.1]{VDQR17}.

Let~$\eta \in C^\infty_c(\R)$ be a cutoff function satisfying~$0 \le \eta \le 1$ in~$\R$,~$\mbox{supp}(\eta) \subset (-2, 2)$,~$\eta = 1$ in~$(-1, 1)$, and~$|\eta'| \le 4$. For~$R \ge 1$, set~$\eta_R(x) := \eta(x/R)$. For~$t_0 < \tau < t < t_1$ fixed, we multiply the equation in~\eqref{g=-W''ueq} against~$\eta_R^2 \widetilde{u}$ and integrate the resulting identity in~$(\tau, t) \times \R$. We get
\begin{equation} \label{equationagainstcutoffed}
\int_{\tau}^{t} \int_\R \eta_R(x)^2 \widetilde{u}(x, \sigma) \left\{ \partial_t \widetilde{u}(x, \sigma) + (-\Delta)^s \widetilde{u}(x, \sigma) + \left( W''(w(x)) - \frac{\lambda}{2} \right) \widetilde{u}(x, \sigma) \right\} dx d\sigma = 0.
\end{equation}

On the one hand,
\begin{equation} \label{timeder}
\begin{aligned}
2 \int_{\tau}^{t} \int_\R \eta_R(x)^2 \widetilde{u}(x, \sigma) \partial_t \widetilde{u}(x, \sigma) \, dx d\sigma & = \int_{\tau}^{t} \frac{d}{d\sigma} \left( \int_\R \eta_R(x)^2 \widetilde{u}(x, \sigma)^2 \, dx \right) d\sigma \\
& = \int_\R \eta_R(x)^2 \widetilde{u}(x, t)^2 \, dx - \int_\R \eta_R(x)^2 \widetilde{u}(x, \tau)^2 \, dx.
\end{aligned}
\end{equation}

Secondly, after a symmetrization, for any fixed~$\sigma \in (\tau, t)$ we write
\begin{equation} \label{lapl=iota}
2 \int_\R \eta_R(x)^2 \widetilde{u}(x, \sigma) (-\Delta)^s \widetilde{u}(x, \sigma) \, dx = \iint_{\R^2} \frac{\iota(x, y)}{|x - y|^{1 + 2 s}} \, dx dy,
\end{equation}
where
$$
\iota(x, y) := \left( \widetilde{u}(x, \sigma) - \widetilde{u}(y, \sigma) \right) \left( \eta_R(x)^2 \widetilde{u}(x, \sigma) - \eta_R(y)^2 \widetilde{u}(y, \sigma) \right).
$$
We claim that, for every~$x, y \in \R$, it holds
\begin{equation} \label{iotage}
\begin{aligned}
\iota(x, y) & \ge \left( 1 - R^{-s} \right)\max \{ \eta_R(x), \eta_R(y) \}^2 |\widetilde{u}(x, \sigma) - \widetilde{u}(y, \sigma)|^2 \\
& \quad - R^s \max \left\{ \widetilde{u}(x, \sigma)^2, \widetilde{u}(y, \sigma)^2 \right\} |\eta_R(x) - \eta_R(y)|^2.
\end{aligned}
\end{equation}
When both~$x$ and~$y$ lie outside of the support of~$\eta_R$, inequality~\eqref{iotage} is clearly valid, since both sides of it vanish. Suppose then that~$\max \{ \eta_R(x), \eta_R(y) \} > 0$. As~\eqref{iotage} is symmetric in~$x$ and~$y$, we can also assume without loss of generality that~$\eta_R(x) \ge \eta_R(y)$. Using the weighted Young's inequality, we compute
\begin{align*}
\iota (x, y) & = \eta_R(x)^2 |\widetilde{u}(x, \sigma) - \widetilde{u}(y, \sigma)|^2 + (\widetilde{u}(x, \sigma) - \widetilde{u}(y, \sigma)) \widetilde{u}(y, \sigma) (\eta_R(x) + \eta_R(y)) (\eta_R(x) - \eta_R(y)) \\
& \ge \left( \eta_R(x)^2 - \frac{\varepsilon^2 (\eta_R(x) + \eta_R(y))^2}{2} \right) |\widetilde{u}(x, \sigma) - \widetilde{u}(y, \sigma)|^2 - \frac{\widetilde{u}(y, \sigma)^2 |\eta_R(x) - \eta_R(y)|^2}{2 \varepsilon^2},
\end{align*}
for any~$\varepsilon > 0$. The choice~$\varepsilon := \sqrt{2} R^{- s / 2}\eta_R(x) / (\eta_R(x) + \eta_R(y)) \ge R^{-s/2}/\sqrt{2}$ leads to~\eqref{iotage}.

As~$\eta_R = 1$ in~$(-R, R)$, inequality~\eqref{iotage} gives in particular that
$$
\iint_{\R^2} \frac{\iota(x, y)}{|x - y|^{1 + 2 s}} \, dx dy \ge \left( 1 - R^{-s} \right) [\widetilde{u}(\cdot, \sigma)]_{H^s(-R, R)}^2 - 2 R^s \int_{\R} \widetilde{u}(x, \sigma)^2 \left( \int_\R \frac{|\eta_R(x) - \eta_R(y)|^2}{|x - y|^{1 + 2 s}} \, dy \right) dx.
$$
By the properties of~$\eta_R$, we have
$$
\int_\R \frac{|\eta_R(x) - \eta_R(y)|^2}{|x - y|^{1 + 2 s}} \, dy \le C \left( R^{-2} \int_0^R r^{1 - 2 s} \, dr + \int_R^{+\infty} \frac{dr}{r^{1 + 2 s}} \right) \le C R^{-2 s},
$$
for all~$x \in \R$ and for some constant~$C$ depending only on~$s$. Consequently, by combining the above two estimates with~\eqref{lapl=iota}, we conclude that
$$
\int_\tau^t \int_\R \eta_R(x)^2 \widetilde{u}(x, \sigma) (-\Delta)^s \widetilde{u}(x, \sigma) \, dx d\sigma \ge \frac{1 - R^{-s}}{2} \int_\tau^t [\widetilde{u}(\cdot, \sigma)]_{H^s(-R, R)}^2 \, d\sigma - C R^{- s} \| \widetilde{u} \|_{L^2(\R \times (\tau, \sigma))}^2.
$$

Putting together the last inequality,~\eqref{equationagainstcutoffed},~\eqref{timeder}, and letting~$R \rightarrow +\infty$, we get
$$
\int_\R \widetilde{u}(x, t)^2 \, dx - \int_\R \widetilde{u}(x, \tau)^2 \, dx + \int_\tau^t \left\{ [\widetilde{u}(\cdot, \sigma)]_{H^s(\R)}^2 + \int_\R \left( 2 \, W''(w(x)) - \lambda \right) \widetilde{u}(x, \sigma)^2 \, dx \right\} d\sigma \le 0.
$$
Observe that the limit can be taken in a rigorous way thanks to hypothesis~\eqref{uinL2} and the boundedness of~$W''$. Note now that, in view of assumption~\eqref{L2uortw'}, we have that~$\widetilde{u}(\cdot, \sigma)$ is orthogonal in~$L^2(\R)$ to~$w'$ for all~$\sigma \in (\tau, t)$. Hence, the non-degeneracy inequality~\eqref{DPVnondeg} gives that the quantity within curly brackets in the last formula is non-negative for all~$\sigma \in (\tau, t)$. Consequently,
$$
\int_\R \widetilde{u}(x, t)^2 \, dx \le \int_\R \widetilde{u}(x, \tau)^2 \, dx \quad \mbox{for all } t_0 < \tau < t < t_1.
$$
Recalling the definition of~$\widetilde{u}$, we are led to the conclusion of the lemma.
\end{proof}

\section{Solving for~$\psi$. Proofs of Proposition~\ref{mainlinprop} and Theorem~\ref{nonlinearthm}.} \label{psisec}

\noindent
In the present section and the next one, we complete the proof of Theorem~\ref{mainthm2} initiated in Section~\ref{outsec}. Here, we address the solvability in~$\Adot_T$ of problems~\eqref{psiprobthm} and~\eqref{psiDirprob}, showing in particular the validity of Proposition~\ref{mainlinprop} and Theorem~\ref{nonlinearthm}. Prior to this, we present a couple of lemmas containing estimates for the functions~$\E$ and~$\NN$.

Throughout the section, we assume~$h$ to be in~$\bar{B}_1(\HH_{T, \mu})$, for some~$\mu$ satisfying~\eqref{mulimit0}.

\subsection{Some preliminary estimates} \label{ENsubsec}

We include in this subsection a couple of technical results that will be often used both in this section and the next.

We begin with the following lemma, which contains some estimates for the error terms~$\E_1$,~$\E_2$, and~$\E_{0, j}$. See in particular~\eqref{ElePhinew}, which motivates our choice for the weight function~$\Phi$ and, as a result, for the norm~$\| \cdot \|_{\A_T}$.

\begin{lemma} \label{ElePhilem}
There exist two generic constants~$T_0, C \ge 1$ such that
\begin{alignat}{3}
\label{EklePhi}
|\E_k(x, t)| & \le C \Phi(x, t) && \qquad \mbox{for } k = 1, 2, \\
\label{E0elllePhi}
\int_\R |\E_{0, j}(x, t)| Z_j(x, t) \, dx & \le C \, t^{- \frac{2 s}{1 + 2s}} && \qquad \mbox{for } j = 1, \ldots, N, \\
\label{E0t1-t2}
\left| \partial_t \E_1(x, t) \right| & \le C \, t^{- \frac{2 s}{1 + 2 s}} \Phi(x, t), && \\
\label{E0ellt1-t2}
\left| \partial_t \E_{0, j}(x, t) \right| & \le C \, t^{- \frac{2 s}{1 + 2 s}} && \qquad \mbox{for } j = 1, \ldots, N,
\end{alignat}
for all~$x \in \R$ and~$t \ge T_0$. In particular,
\begin{equation} \label{ElePhinew}
|\E(x, t)| \le C \Phi(x, t),
\end{equation}
for all~$x \in \R$ and~$t \ge T_0$.
\end{lemma}
\begin{proof}
We first prove that~$\E_1$ satisfies~\eqref{EklePhi}. To do this, we set
\begin{equation} \label{Itdef}
I(t) := \left( - 2 \beta_N t^{\frac{1}{1 + 2 s}}, 2 \beta_N t^{\frac{1}{1 + 2 s}} \right),
\end{equation}
and distinguish between the cases~$x \in \R \setminus I(t)$ and~$x \in I(t)$.

First, we let~$t \ge T_0$ and suppose that~$x \in \R \setminus I(t)$. Without loss of generality, we also assume that~$x > 0$, so that~$x \ge 2 \beta_n \, t^{\frac{1}{1 + 2 s}}$. As~$h \in \bar{B}_1(\HH_{T, \mu})$ with~$\mu > 2 s / (1 + 2 s)$, taking~$T_0$ sufficiently large we have
$$
|h_i(t)| \le t^{1 - \mu} \le \frac{\beta_N}{2} \, t^{\frac{1}{1 + 2 s}}
$$
and, therefore, recalling~\eqref{Itdef},
$$
x - \xi_i(t) \ge x - |\xi_i^0(t)| - |h_i(t)| \ge x - |\beta_i| t^{\frac{1}{1 + 2 s}} - \frac{\beta_N}{2} \, t^{\frac{1}{1 + 2 s}} \ge x - \frac{3 \beta_N}{2} \, t^{\frac{1}{1 + 2 s}} \ge \frac{x}{4},
$$
for every index~$i = 1, \ldots, N$. By this, the Lipschitz continuity and periodicity of~$W'$, the fact that~$W'(0) = 0$, and estimate~\eqref{wasympt} for the asymptotic behavior of~$w$, recalling definitions~\eqref{E1def} and~\eqref{zdef} we get
\begin{align*}
|\E_1(x, t)| & \le \left| W' \! \left( \sum_{i = 1}^N \left( w(x - \xi_i(t)) - 1 \right) \right) - W'(0) \right| + \sum_{j = 1}^N \left| W'(w(x - \xi_j(t)) - 1) - W'(0) \right| \\
& \le 2 \| W'' \|_{L^\infty(\R)} \sum_{i = 1}^N \left( 1 - w(x - \xi_i(t)) \right) \le C (x - \xi_i(t))^{- 2 s} \le C |x|^{- 2 s},
\end{align*}
for some generic constant~$C \ge 1$. This gives estimate~\eqref{EklePhi} for~$\E_1$ and~$x \in \R \setminus I(t)$---recall~\eqref{Phidef}.

To check that such a bound also holds when~$x \in I(t)$, let~$j \in \{ 1, \ldots, N \}$ be the unique integer for which
\begin{equation} \label{uniquej}
\frac{\xi_{j - 1}^0(t)  + \xi_j^0(t)}{2} = \frac{\beta_{j - 1} + \beta_j}{2} \, t^{\frac{1}{1 + 2 s}} \le x < \frac{\beta_j + \beta_{j + 1}}{2} \, t^{\frac{1}{1 + 2 s}} = \frac{\xi_{j}^0(t)  + \xi_{j + 1}^0(t)}{2},
\end{equation}
where we adopt the convention that~$\beta_0 := 3 \beta_1 = - 3 \beta_N$ and~$\beta_{N + 1} := 3 \beta_N$. Assuming~\eqref{uniquej}, we have
\begin{equation} \label{x-xiget}
|x - \xi_i(t)| \ge \frac{1}{2} \min \left\{ \beta_j - \beta_{j - 1}, \beta_{j + 1} - \beta_j \right\} t^{\frac{1}{1 + 2 s}} - |h_i(t)| \ge \frac{1}{C} \, t^{\frac{1}{1 + 2 s}} \quad \mbox{for all } i \ne j,
\end{equation}
provided~$T_0$ is large enough. Using this,~\eqref{wasympt}, and the periodicity and regularity of~$W'$, we compute
\begin{align*}
|\E_1(x, t)| & \le \left| W' \! \left( w(x - \xi_j(t)) + \sum_{i \ne j} w(x - \xi_i(t)) \right) - W'(w(x - \xi_j(t))) \right| + \sum_{i \ne j} |W'(w(x - \xi_i(t)))| \\
& \le 2 \| W'' \|_{L^\infty(\R)} \left( \sum_{i < j} \left( 1 - w(x - \xi_i(t)) \right) + \sum_{i > j} w(x - \xi_i(t)) \right) \le C \, t^{- \frac{2 s}{1 + 2 s}},
\end{align*}
and~\eqref{EklePhi} is true for~$\E_1$ also when~$x \in I(t)$.

We now verify that~\eqref{EklePhi} is fulfilled by~$\E_2$ as well---recall~\eqref{E2def} and~\eqref{Zidef} for the definition of~$\E_2$. This is an immediate consequence of the stronger estimate
\begin{align*}
|\E_2(x, t)| & \le \sum_{i = 1}^N |\dot{\xi}_i(t) |w'(x - \xi_i(t)) \le C \, t^{ - \frac{2 s}{1 + 2 s}} \left( \chi_{I(t)}(x) + \frac{\chi_{\R \setminus I(t)}(x)}{|x|^{1 + 2 s}} \right),
\end{align*}
which holds for all~$x \in \R$ and~$t \ge T_0$, provided~$T_0$ is sufficiently large.

Observe that~\eqref{ElePhinew} is an immediate consequence of~\eqref{EklePhi}, thanks to the decomposition~\eqref{Edecomp}.

We proceed to check~\eqref{E0elllePhi}. For~$x$ satisfying~\eqref{uniquej}, estimate~\eqref{x-xiget} holds true and therefore, recalling definition~\eqref{E0idef}, we have
$$
|\E_{0, j}(x, t)| \le \| W''' \|_{L^\infty(\R)} \left\{ \sum_{i < j} \left( 1 - w(x - \xi_i(t)) \right) + \sum_{i > j} w(x - \xi_i(t)) \right\} \le C \, t^{- \frac{2 s}{1 + 2 s}}.
$$
Hence, applying the change of variables~$y := x - \xi_j(t)$, we estimate
$$
\int_{\frac{\xi_{j - 1}^0(t)  + \xi_j^0(t)}{2}}^{\frac{\xi_{j}^0(t)  + \xi_{j + 1}^0(t)}{2}} |\E_{0, j}(x, t)| Z_j(x, t) \, dx \le C \, t^{- \frac{2 s}{1 + 2 s}} \int_\R w'(y) \, dy = C \, t^{- \frac{2 s}{1 + 2 s}}.
$$
Conversely, by the same change of variables,
\begin{align*}
& \int_{\R \setminus \left( \frac{\xi_{j - 1}^0(t) + \xi_j^0(t)}{2}, \, \frac{\xi_{j}^0(t) + \xi_{j + 1}^0(t)}{2} \right)} |\E_{0, j}(x, t)| Z_j(x, t) \, dx \\
& \hspace{12pt} \le 2 \| W'' \|_{L^\infty(\R)} \int_{\R \setminus \left( - \frac{\xi_j^0(t) - \xi_{j - 1}^0(t)}{2} - h_j(t), \, \frac{\xi_{j + 1}^0(t) - \xi_{j}^0(t)}{2} - h_j(t) \right)} \! w'(y) \, dy \le C \int_{\frac{1}{C} t^{\frac{1}{1 + 2 s}}}^{+\infty} \frac{dy}{y^{1 + 2 s}} \le C \, t^{- \frac{2 s}{1 + 2 s}},
\end{align*}
provided~$T_0$ is large enough. The combination of the last two inequalities gives~\eqref{E0elllePhi}.

We now address the validity of~\eqref{E0t1-t2}. First, we compute the derivative of~$\E_1$ with respect to~$t$:
$$
\partial_t \E_1(x, t) = - \sum_{i = 1}^N \Big\{ W''(z(x, t)) - W''(w(x - \xi_i(t))) \Big\} w'(x - \xi_i(t)) \dot{\xi}_i(t)
$$
When~$x \in \R \setminus I(t)$ (and assuming also, without loss of generality, that~$x < 0$), we simply estimate
$$
|\partial_t \E_1(x, t)| \le 2 \| W'' \|_{L^\infty(\R)} \sum_{i = 1}^N w'(x - \xi_i(t)) |\dot{\xi}_i(t)| \le C x^{ - 1 - 2 s} t^{- \frac{2 s}{1 + 2 s}} \le C \, t^{-1} \Phi(x, t).
$$
On the other hand, for~$x \in I(t)$ we let~$j$ be defined by~\eqref{uniquej} and compute
\begin{align*}
|\partial_t \E_1(x, t)| & \le \left| W'' \! \left( w(x - \xi_j(t)) + \sum_{i \ne j} w(x - \xi_i(t)) \right) - W''(w(x - \xi_j(t))) \right| w'(x - \xi_j(t)) |\dot{\xi}_j(t)| \\
& \quad + \sum_{i \ne j} \Big| W''(z(x, t)) - W''(w(x - \xi_i(t))) \Big| w'(x - \xi_i(t)) |\dot{\xi}_i(t)| \\
& \le \frac{C}{t^{\frac{2 s}{1 + 2 s}}} \left\{ \| W''' \|_{L^\infty(\R)} \left( \sum_{i < j} \left( w(x - \xi_i(t)) - 1 \right) + \sum_{i > j} w(x - \xi_i(t)) \right) + \frac{\| W'' \|_{L^\infty(\R)}}{t} \right\} \\
& \le C \, t^{- \frac{2 s}{1 + 2 s}} \Phi(x, t).
\end{align*}
The last two inequalities lead to~\eqref{E0t1-t2}.

Finally, for~$j = 1, \ldots, N$ we simply have
$$
\left| \partial_t \E_{0, j}(x, t) \right| = \left| W'''(z(x, t)) \sum_{i = 1}^N w'(x - \xi_i(t)) \dot{\xi}_i(t) - W'''(w(x - \xi_j(t))) w'(x - \xi_j(t)) \dot{\xi}_j(t) \right| \le C \, t^{- \frac{2 s}{1 + 2 s}},
$$
which is~\eqref{E0ellt1-t2}. This concludes the proof of Lemma~\ref{ElePhilem}.
\end{proof}

We proceed with a second lemma, containing some computations for the nonlinear term~$\NN$. Recall the definition~\eqref{Atnualphadef} of the space of H\"older continuous functions~$\A_T^\alpha$.

\begin{lemma} \label{Ndecaylem}
There exists a generic constant~$C \ge 1$ such that
\begin{equation} \label{Npsi1psi2est}
\left\| \NN[\psi_1] - \NN[\psi_2] \right\|_{\A_{T}} \le C \, T^{- \frac{2 s}{1 + 2 s}} \max \left\{ \| \psi_1 \|_{\A_{T}}, \| \psi_2 \|_{\A_{T}} \right\} \| \psi_1 - \psi_2 \|_{\A_{T}},
\end{equation}
for every~$\psi_1, \psi_1 \in \A_{T}$ and~$T \ge 1$. In particular,
\begin{equation} \label{Npsiest}
\| \NN[\psi] \|_{\A_{T}} \le C \, T^{- \frac{2s}{1 + 2 s}} \| \psi \|_{\A_{T}}^2,
\end{equation}
for every~$\psi \in \A_{T}$ and~$T \ge 1$. In addition, given any~$\alpha \in (0, 1)$,
\begin{equation} \label{Nt1-Nt2}
|\NN[\psi](x, t_1) - \NN[\psi](x, t_2)| \le C \| \psi \|_{\A_{T}} \| \psi \|_{\A_{T}^\alpha} \, t^{- \frac{4 s}{1 + 2 s}} |t_1 - t_2|^\alpha,
\end{equation}
for every~$\psi \in \A_T^\alpha$,~$x \in \R$,~$t \ge T \ge 1$, and~$t_1, t_2 \in [t, t + 1]$.
\end{lemma}
\begin{proof}
We begin to deal with~\eqref{Npsi1psi2est} and~\eqref{Npsiest}. Notice that it suffices to establish~\eqref{Npsi1psi2est}, as~\eqref{Npsiest} follows by taking~$\psi_1 = \psi$ and~$\psi_2 = 0$ in~\eqref{Npsi1psi2est}, since~$\NN[0] = 0$. Recalling the definition~\eqref{Ndef} of~$N$, we have
$$
\NN[\psi_1] - \NN[\psi_2] = W'(z + \psi_1) - W'(z + \psi_2) - W''(z) \left( \psi_1 - \psi_2 \right).
$$
Thanks to the Lipschitz continuity of~$W''$, we may then estimate
\begin{align*}
\left| \NN[\psi_1] - \NN[\psi_2] \right| & \le \left| \int_{\psi_2}^{\psi_1} \left| W''(z + \tau) - W''(z) \right| d\tau \right| \le \| W''' \|_{L^\infty(\R)} \left| \int_{\psi_2}^{\psi_1} |\tau| \, d\tau \right| \\
& \le C \max \left\{ |\psi_1|, |\psi_2| \right\} |\psi_1 - \psi_2|.
\end{align*}
Estimate~\eqref{Npsi1psi2est} plainly follows from this, recalling definition~\eqref{Phidef} of~$\Phi$.

To establish~\eqref{Nt1-Nt2}, we let~$t_1, t_2 \in [t, t+ 1]$ and compute, using that~$W''' \in C^1(\R)$,
\begin{align*}
& |\NN[\psi](x, t_1) - \NN[\psi](x, t_2)| \\
& \hspace{40pt} = \left| \int_0^{\psi(x, t_1)} \left( \int_0^\tau W'''(z(x, t_1) + \sigma) \, d\sigma \right) d\tau - \int_0^{\psi(x, t_2)} \left( \int_0^\tau W'''(z(x, t_2) + \sigma) \right) d\tau \right| \\
& \hspace{40pt} \le \left| \int_0^{\psi(x, t_1)} \left( \int_0^\tau |W'''(z(x, t_1) + \sigma) - W'''(z(x, t_2) + \sigma)| \, d\sigma \right) d\tau \right| \\
& \hspace{40pt} \quad + \left| \int_{\psi(x, t_1)}^{\psi(x, t_2)} \left( \int_0^\tau W'''(z(x, t_2) + \sigma) \right) d\tau \right| \\
& \hspace{40pt} \le C \Big\{ |\psi(x, t_1)|^2 \, |z(x, t_1) - z(x, t_2)| + \max \{ |\psi(x, t_1)|, |\psi(x, t_2)| \} \, |\psi(x, t_1) - \psi(x, t_2)| \Big\}.
\end{align*}
Since~$w' \in L^\infty(\R)$ and~$h \in \bar{B}_1(\HH_{T, \mu})$, one easily checks that~$|z(x, t_1) - z(x, t_2)| \le C \, t^{- \frac{2 s}{1 + 2 s}} |t_1 - t_2|$. Recalling definition~\eqref{CTnualphadef}, we are immediately led to~\eqref{Nt1-Nt2}. The proof is now complete.
\end{proof}

\subsection{Linear theory for $\psi$}

In order to address the nonlinear initial value problem~\eqref{psiDirprob}, we develop in this subsection a solvability theory for the corresponding linear problem~\eqref{psiprobthm}---namely, we establish Proposition~\ref{mainlinprop}.

As a first step towards its proof, we have the following existence and uniqueness result in a bounded time interval~$(T, T + \tau)$. Notice that the right-hand side~$g$ is a general function in~$\A_{(T, T + \tau)}$. Thus, the found solution~$\psi$ belongs only to~$\A_{(T, T+ \tau)}$ and not necessarily to~$\Adot_{(T, T + \tau)}$. As a result, the corresponding estimate is governed by a constant that depends on (an upper bound~$\bar{\tau}$ on)~$\tau$.

\begin{lemma} \label{locweightedLinftytoLinftylem}
Let~$T \ge 1$,~$\bar{\tau} \ge \tau > 0$,~$h \in \bar{B}_1(\HH_{T, \mu})$,~$g \in \A_{(T, T + \tau)}$, and~$\psi_0 \in \A_{\{ T \}}$. Then, there exists a unique solution~$\psi \in \A_{(T, T + \tau)}$ of
\begin{equation} \label{psiprob}
\begin{cases}
\partial_t \psi + (-\Delta)^s \psi + W''(z) \psi = g & \quad \mbox{in } \R \times (T, T + \tau) \\
\psi = \psi_0 & \quad \mbox{on } \R \times \{ T \}.
\end{cases}
\end{equation}
In addition,~$\psi$ satisfies
\begin{equation} \label{psilocinAnuest}
\| \psi \|_{\A_{(T, T + \tau)}} \le C_{\bar{\tau}} \left( \| g \|_{\A_{(T, T + \tau)}} + \| \psi_0 \|_{\A_{\{ T \}}} \right),
\end{equation}
for some constant~$C_{\bar{\tau}}$ depending only on structural quantities and~$\bar{\tau}$.
\end{lemma}
\begin{proof}
The existence and uniqueness of a solution~$\psi \in L^\infty(\R \times (T, T + \tau))$ to~\eqref{psiprob} is a consequence of Proposition~\ref{existinLinftyprop}. We thus only need to establish that~$\psi \in \A_{(T, T + \tau)}$ and~\eqref{psilocinAnuest} is true.

Consider the positive functions~$\eta^{(1)}_{T}$ and~$\eta^{(2)}$ introduced, respectively, in Lemmas~\ref{etabound0lem} and~\ref{etaboundlem}, for~$m = 1$. Define
$$
\overline{\psi}(x, t) := e^{M (t - T)} \left( C_1 \, \eta^{(1)}_{T}(x, t) + C_2 \, \eta^{(2)}(x, t) \right) \quad \mbox{for } x \in \R, \, t \ge T,
$$
with~$M := \| W'' \|_{L^\infty(\R)} + 1$,~$C_1 := \| \psi_0 \|_{\A_{\{ T \}}}$, and~$C_2 := \| g \|_{\A_{(T, T + \tau)}} $. Thanks to Lemma~\ref{Arescalelem}, we have that~$\overline{\psi}$ is a mild solution of
$$
\begin{cases}
\partial_t \overline{\psi} + (-\Delta)^s \overline{\psi} = \overline{G}[\overline{\psi}] & \quad \mbox{in } \R \times (T, +\infty) \\
\overline{\psi} = \overline{\psi}_0 & \quad \mbox{on } \R \times \{ T \},
\end{cases}
$$
for some~$\overline{\psi}_0 \ge C_1 \psi(\cdot, T)$ in~$\R$ and with
$$
\overline{G}(x, t, u) := C_2 e^{M(t - T)} \Phi(x, t) + \| W'' \|_{L^\infty(\R)} \, u_+ + W''(z(x, t)) \, u_-.
$$
Notice that here we took advantage of the positivity of~$\overline{\psi}$ in~$\R \times (T, T + \tau)$. As~$\psi_0(x) \le \overline{\psi}_0(x)$ for all~$x \in \R$ and~$g(x, t) - W''(z(x, t)) \, u \le \overline{G}(x, t, u)$ for all~$x \in \R$,~$t \in (T, T + \tau)$, and~$u \in \R$, we conclude, using the comparison principle of Proposition~\ref{compprincforstrongprop}, that~$\psi \le \overline{\psi}$ in~$\R \times (T, T + \tau)$. Since an analogous bound from below can be obtained by comparing~$\psi$ to~$- \overline{\psi}$, applying Lemmas~\ref{etabound0lem}-\ref{etaboundlem} we infer that~$\psi \in \A_{(T, T + \tau)}$ and that~\eqref{psilocinAnuest} holds true.
\end{proof}

The key tool that we need to prove Proposition~\ref{mainlinprop} is an a priori estimate like~\eqref{psilocinAnuest} but with a constant independent of~$\bar{\tau}$. We do this in the next proposition, at the price of assuming that~$\psi$ belongs to~$\Adot_{(T, T + \tau)}$, i.e., that it satisfies the orthogonality conditions~\eqref{psiort}.

\begin{proposition} \label{weightedLinftytoLinftyprop}
Let~$T, \tau \ge 1$,~$h \in \bar{B}_1(\HH_{T, \mu})$,~$g \in \A_{(T, T + \tau)}$, and~$\psi_0 \in \A_{\{ T \}}$. Let~$\psi$ be the mild solution of problem~\eqref{psiprob} and suppose that~$\psi \in \Adot_{(T, T + \tau)}$. Then, there exist two generic constants~$T_0, C_\sharp \ge 1$ such that
$$
\left\| \psi \right\|_{\A_{(T, T + \tau)}} \le C_\sharp \left( \left\| g \right\|_{\A_{(T, T + \tau)}} + \| \psi_0 \|_{\A_{\{ T \}}} \right),
$$
provided~$T \ge T_0$.
\end{proposition}

\begin{proof}

We argue by contradiction and suppose that, for every~$j \in \N$, there exist two positive real numbers
\begin{equation} \label{b>a+1}
b_j - 1 > a_j \ge j,
\end{equation}
an array of trajectories~$\xi^{(j)} = \xi^0 + h^{(j)}$, with~$h^{(j)} \in \bar{B}_1(\HH_{a_j, \mu})$, and three functions~$\psi^{(j)}, g^{(j)} \in \A_{(a_j, b_j)}$, and~$\psi_0^{(j)} \in L^\infty(\R)$, with~$\| \psi^{(j)}_0 / \Phi(\cdot, a_j) \|_{L^\infty(\R)} < +\infty$, satisfying
\begin{equation} \label{eqforpsij}
\begin{cases}
\partial_t \psi^{(j)} + (-\Delta)^s \psi^{(j)} + b^{(j)} \psi^{(j)} = g^{(j)} & \quad \mbox{in } \R \times (a_j, b_j), \\
\psi^{(j)} = \psi_0^{(j)} & \quad \mbox{on } \R \times \{ a_j \},
\end{cases}
\end{equation}
with~$b^{(j)}(x, t) := W''(z^{(j)}(x, t))$ and~$z^{(j)}(x, t) := \sum_{i = 1}^N w(x - \xi_i^{(j)}(t))$, as well as the orthogonality condition
\begin{equation} \label{psijortcond}
\int_{\R} \psi^{(j)}(x, t) w'(x - \xi^{(j)}_i(t)) \, dx = 0 \quad \mbox{for a.e.~} t \in (a_j, b_j) \mbox{ and every } i = 1, \ldots, N,
\end{equation}
but for which
$$
\left\| \frac{\psi^{(j)}}{\Phi} \right\|_{L^\infty(\R \times (a_j, b_j))} > j \left\{ \left\| \frac{g^{(j)}}{\Phi} \right\|_{L^\infty(\R \times (a_j, b_j))} + \left\| \frac{\psi_0^{(j)}}{\Phi(\cdot, a_j)} \right\|_{L^\infty(\R)} \right\}.
$$
The equation being linear, after a renormalization we may also suppose that
\begin{equation} \label{psij=1}
\left\| \frac{\psi^{(j)}}{\Phi} \right\|_{L^\infty(\R \times (a_j, b_j))} = 1
\end{equation}
and consequently that
\begin{equation} \label{gjdecay}
\left\| \frac{g^{(j)}}{\Phi} \right\|_{L^\infty(\R \times (a_j, b_j))} + \left\| \frac{\psi_0^{(j)}}{\Phi(\cdot, a_j)} \right\|_{L^\infty(\R)} < \frac{1}{j}.
\end{equation}

For~$R \ge 1$, consider the sets
$$
\mathscr{B}_j(R) := \left\{ (x, t) \in \R \times (a_j, b_j) : |x - \xi_i^{(j)}(t)| < R \mbox{ for some } i = 1, \ldots, N \right\}
$$
We claim that
\begin{equation} \label{weightedlemclaim}
\lim_{j \rightarrow +\infty} \left\| \frac{\psi^{(j)}}{\Phi} \right\|_{L^\infty(\mathscr{B}_j(R))} = 0.
\end{equation}
For the moment, we assume~\eqref{weightedlemclaim} to hold true and show that, under its validity, we reach a contradiction.

Let
\begin{equation} \label{mandadef}
m := \frac{W''(0)}{2} > 0 \quad \mbox{and} \quad d^{(j)}(x, t) := b^{(j)}(x, t) \chi_{\left( \R \times (a_j, b_j) \right) \setminus \mathscr{B}_j(R)} + m \chi_{\mathscr{B}_j(R)}.
\end{equation}
If~$R$ is chosen sufficiently large, but independently of~$j$, then
$$
b^{(j)} \ge m \quad \mbox{in } \big( \R \times (a_j, b_j) \big) \setminus \mathscr{B}_j(R).
$$
Accordingly,
\begin{equation} \label{a>m}
d^{(j)} \ge m \quad \mbox{in } \R \times (a_j, b_j).
\end{equation}
We rewrite~\eqref{eqforpsij} as
$$
\begin{cases}
\partial_t \psi^{(j)} + (- \Delta)^s \psi^{(j)} = G^{(j)}[\psi^{(j)}] & \quad \mbox{in } \R \times (a_j, b_j) \\
\psi^{(j)} = \psi_0^{(j)} & \quad \mbox{on } \R \times \{ a_j \},
\end{cases}
$$
with~$G^{(j)}(x, t, u) := f^{(j)}(x, t) - d^{(j)}(x, t) \, u \,$ and~$f^{(j)} := g^{(j)} + (m - b^{(j)}) \chi_{\mathscr{B}_j(R)} \psi^{(j)}$.

Consider now the positive mild solutions~$\etaini_{a_j}$ and~$\etadis$ of problems~\eqref{etaprob0} (with~$T = a_j$) and~\eqref{etaprob}, respectively, with~$m$ as in~\eqref{mandadef}. Let
\begin{equation} \label{Cj12def}
\begin{aligned}
C_j^{(1)} & := \left\| \frac{\psi_0^{(j)}}{\Phi(\cdot, a_j)} \right\|_{L^\infty(\R)}, \\
C_j^{(2)} & := \left\| \frac{g^{(j)}}{\Phi} \right\|_{L^\infty(\R \times (a_j, b_j))} + 2 \| W'' \|_{L^\infty(\R)} \left\| \frac{\psi^{(j)}}{\Phi} \right\|_{L^\infty(\mathscr{B}_j(R))},
\end{aligned}
\end{equation}
and define~$\overline{\psi}^{(j)} := C_j^{(1)} \etaini_{a_j} + C_j^{(2)} \etadis$. Thanks to its positivity, this function solves
$$
\begin{cases}
\partial_t \overline{\psi}^{(j)} + (-\Delta)^s \overline{\psi}^{(j)} = \overline{G}^{(j)}[\, \overline{\psi}^{(j)}] & \quad \mbox{in } \R \times (a_j, +\infty) \\
\overline{\psi}^{(j)} = \overline{\psi}_0^{(j)} & \quad \mbox{on } \R \times \{ a_j \},
\end{cases}
$$
for some~$\overline{\psi}_0^{(j)} \in L^\infty(\R)$ satisfying~$\overline{\psi}_0^{(j)} \ge C_j^{(1)} \Phi(\cdot, a_j)$ in~$\R$ and with
$$
\overline{G}^{(j)}(x, t, u) := C_j^{(2)} \Phi(x, t) - m u_+ + d^{(j)}(x, t) u_-.
$$
Observe that
$$
|\psi_0^{(j)}(x)| \le C_j^{(1)} \Phi(x, a_j) \quad \mbox{and} \quad |f^{(j)}(x, t)| \le C_j^{(2)} \Phi(x, t) \quad \mbox{for a.e.~} x \in \R, \, t \in (a_j, b_j).
$$
From these inequalities and~\eqref{a>m} we infer that~$\psi_0^{(j)} \le \overline{\psi}_0^{(j)}$ in~$\R$ and~$G^{(j)} \le \overline{G}^{(j)}$ in~$\R \times (a_j, b_j) \times \R$. Hence, we can apply the comparison principle of Proposition~\ref{compprincforstrongprop} and get that
$$
\psi^{(j)} \le \overline{\psi}^{(j)} \quad \mbox{in } \R \times (a_j, b_j),
$$
By considering~$- \overline{\psi}^{(j)}$ instead of~$\overline{\psi}^{(j)}$, we get the analogous lower bound. By this and Lemmas~\ref{etabound0lem}-\ref{etaboundlem}, we conclude that
$$
|\psi^{(j)}| \le C \left\{ C_j^{(1)} + C_j^{(2)} \right\} \Phi \quad \mbox{in } \R \times (a_j, b_j),
$$
for some constant~$C > 0$ independent of~$j$. Recalling~\eqref{Cj12def},~\eqref{gjdecay}, and~\eqref{weightedlemclaim}, we then find that
$$
\left\| \frac{\psi^{(j)}}{\Phi} \right\|_{L^\infty(\R \times (a_j, b_j))} \le C \left\{ \left\| \frac{\psi_0^{(j)}}{\Phi(\cdot, a_j)} \right\|_{L^\infty(\R)} + \left\| \frac{g^{(j)}}{\Phi} \right\|_{L^\infty(\R \times (a_j, b_j))} + \left\| \frac{\psi^{(j)}}{\Phi} \right\|_{L^\infty(\mathscr{B}_j(R))} \right\} \longrightarrow 0,
$$
as~$j \rightarrow +\infty$. But this is in contradiction with~\eqref{psij=1} and, accordingly, the conclusion of Proposition~\ref{weightedLinftytoLinftyprop} follows, under the validity of claim~\eqref{weightedlemclaim}.

To finish the proof of Proposition~\ref{weightedLinftytoLinftyprop}, we therefore only need to prove that~\eqref{weightedlemclaim} holds true. Again, we argue by contradiction and suppose that there exist a constant~$\delta_\star > 0$ and a sequence of points~$(x_j, t_j) \in \mathscr{B}_j(R)$, such that
\begin{equation} \label{ratiogedelta}
\frac{\left| \psi^{(j)}(x_j, t_j) \right|}{\Phi(x_j, t_j)} \ge \delta_\star \quad \mbox{for all } j.
\end{equation}
Note that, by Lemma~\ref{ElePhilem} and~\eqref{gjdecay}, we necessarily have that~$t_j \ge a_j + 1/2$. Up to a subsequence, we may assume that there exists a unique~$\ell \in \{ 1, \ldots, N \}$ for which
\begin{equation} \label{xjtjleR}
|x_j - \xi_\ell^{(j)}(t_j)| < R \quad \mbox{for all } j.
\end{equation}
In particular, there exists a point~$\bar{x} \in B_R$ such that
\begin{equation} \label{barxdef}
\lim_{j \rightarrow +\infty} \left( x_j - \xi_\ell^{(j)}(t_j) \right) = \bar{x},
\end{equation}
up to extracting a further subsequence.

Define
$$
\widetilde{\psi}^{(j)}(x, t) := (t + t_j)^{\frac{2 s}{1 + 2 s}} \psi^{(j)}(x, t + t_j).
$$
Recalling~\eqref{eqforpsij}, Lemma~\ref{Arescalelem} guarantees that~$\widetilde{\psi}^{(j)}$ is a mild solution of
\begin{equation} \label{psitildeprob}
\begin{cases}
\partial_t \widetilde{\psi}^{(j)} + (-\Delta)^s \widetilde{\psi}^{(j)} + \widetilde{b}^{(j)} \widetilde{\psi}^{(j)} = \widetilde{g}^{(j)} & \quad \mbox{in } \R \times (a_j - t_j, b_j - t_j), \\
\widetilde{\psi}^{(j)} = \widetilde{\psi}_0^{(j)} & \quad \mbox{on } \R \times \{ a_j - t_j \},
\end{cases}
\end{equation}
with
\begin{align*}
\widetilde{b}^{(j)}(x, t) & = b^{(j)}(x, t + t_j) - \frac{2 s}{1 + 2 s} (t + t_j)^{- 1}, \\
\widetilde{g}^{(j)}(x, t) & = (t + t_j)^{\frac{2 s}{1 + 2 s}} g^{(j)}(x, t + t_j), \\
\widetilde{\psi}_0^{(j)}(x) & = a_j^{\frac{2 s}{1 + 2 s}} \psi_0^{(j)}(x).
\end{align*}
Furthermore, by~\eqref{psij=1},~\eqref{gjdecay}, the definition~\eqref{Phidef} of~$\Phi$, and~\eqref{b>a+1}, we have that
\begin{equation} \label{psigbtildele}
\begin{gathered}
\| \widetilde{\psi}^{(j)} \|_{L^\infty(\R \times (a_j - t_j, b_j - t_j))} \le 1, \quad \| \widetilde{\psi}_0^{(j)} \|_{L^\infty(\R)} + \| \widetilde{g}^{(j)} \|_{L^\infty(\R \times (a_j - t_j, b_j - t_j))} \le \frac{1}{j}, \\
\mbox{and} \quad \sup_{x \in \R, \, t \in (a_j - t_j, b_j - t_j)} \left| \, \widetilde{b}^{(j)}(x, t) - b^{(j)}(x, t + t_j) \right| \le \frac{1}{j}.
\end{gathered}
\end{equation}
From this, the fact that
\begin{equation} \label{bjleW''}
\| b^{(j)} \|_{L^\infty(\R \times (a_j, b_j))} \le \| W'' \|_{L^\infty(\R)},
\end{equation}
and Proposition~\ref{regforstrongprop}, we infer that, for every~$\varepsilon \in (0, 1/2]$,
\begin{equation} \label{Holderunifinj}
\| \widetilde{\psi}^{(j)} \|_{C^\alpha(\R \times (a_j - t_j + \varepsilon, b_j - t_j))} \le C_\varepsilon,
\end{equation}
for some exponent~$\alpha \in (0 , 1)$, depending only on~$s$, and some constant~$C_\varepsilon > 0$, depending only on~$s$,~$W$, and~$\varepsilon$. In addition, condition~\eqref{psijortcond} translates into
\begin{equation} \label{jtildeortcond}
\int_{\R} \widetilde{\psi}^{(j)}(x, t) w'(x - \xi_i^{(j)}(t + t_j)) \, dx = 0 \quad \mbox{for a.e.~} t \in (a_j - t_j, b_j - t_j) \mbox{ and every } i = 1, \ldots, N,
\end{equation}
while~\eqref{ratiogedelta},~\eqref{xjtjleR},~\eqref{Phidef},~\eqref{xi0def}, and the fact that~$h^{(j)} \in \bar{B}_1(\HH_{a_j, \mu})$ imply that
\begin{equation} \label{psitildegedelta}
|\widetilde{\psi}^{(j)}(x_j, 0)| \ge \delta_\star,
\end{equation}
up to possibly taking a smaller~$\delta_\star > 0$, still uniform in~$j$. Using that~$\widetilde{\psi}^{(j)}$ is a mild solution of~\eqref{psitildeprob} as well as estimates~\eqref{psigbtildele},~\eqref{bjleW''}, and property~\ref{pmass=1} in Section~\ref{solsec}, we immediately deduce that
\begin{equation} \label{tildepsijgoesto0}
\| \widetilde{\psi}^{(j)}(\cdot, t) \|_{L^\infty(\R)} \le \frac{1}{j} + \left( \frac{2}{j} + \| W'' \|_{L^\infty(\R)} \right) (t - a_j + t_j) \quad \mbox{for all } t \in (a_j - t_j, b_j - t_j).
\end{equation}

Set now
$$
\widehat{\psi}^{(j)}(x, t) := \widetilde{\psi}^{(j)}(x + x_j - \bar{x} + \xi_\ell^{(j)}(t + t_j) - \xi_\ell^{(j)}(t_j), t).
$$
By~\eqref{psigbtildele},~\eqref{Holderunifinj},~\eqref{xi0def}, and the fact that~$h^{(j)} \in \bar{B}_1(\HH_{a_j, \mu})$, we infer that
\begin{equation} \label{Holderhatunifinj}
\| \widehat{\psi}^{(j)}\|_{L^\infty(\R \times (a_j - t_j, b_j - t_j))} \le 1 \quad \mbox{and} \quad \| \widehat{\psi}^{(j)} \|_{C^\alpha(\R \times (a_j - t_j + \varepsilon, b_j - t_j))} \le C_\varepsilon.
\end{equation}
Furthermore, as~$\widetilde{\psi}^{(j)}$ is a mild solution of~\eqref{psitildeprob}, applying Definition~\ref{strongsoldef} it is not hard to check that~$\widehat{\psi}^{(j)}$ satisfies\footnote{Formally,~$\widehat{\psi}^{(j)}$ solves the equation
\begin{equation} \tag*{$(\dagger)_j$} \label{driftedeq}
\partial_t \widehat{\psi}^{(j)} + (-\Delta)^s \widehat{\psi}^{(j)} + \widehat{a}^{(j)} \partial_x \widehat{\psi}^{(j)} + \widehat{b}^{(j)} \widehat{\psi}^{(j)} = \widehat{g}^{(j)} \quad \mbox{in } \R \times (a_j - t_j, b_j - t_j),
\end{equation}
with drift coefficient~$\widehat{a}^{(j)}(t) := - \dot{\xi}_\ell(t + t_j)$. Identity~\eqref{psihatjstrong} essentially corresponds to a mild formulation of~\ref{driftedeq}. Below, we will let~$j \rightarrow +\infty$ and find that~$\widehat{\psi}^{(j)}$ converges to a solution~$\widehat{\psi}$ of
\begin{equation} \tag*{$(\dagger)_\infty$} \label{eqforpsihat}
\partial_t \widehat{\psi} + (-\Delta)^s \widehat{\psi} + W''(w) \widehat{\psi} = 0.
\end{equation}
It does not seem possible to get compactness and take such a limit directly through~\ref{driftedeq}, as we do not control the spatial derivative~$\partial_x \widehat{\psi}^{(j)}$---at least when~$s \le 1/2$. To circumvent this issue, we recover the needed compactness by transferring the uniform~$C^\alpha$ bound~\eqref{Holderunifinj} on~$\widetilde{\psi}^{(j)}$ to~$\widehat{\psi}^{(j)}$ (which is done in~\eqref{Holderhatunifinj}, thanks to the boundedness of~$\dot{\xi}_\ell$) and pass~\eqref{psihatjstrong} to the limit to obtain the mild formulation~\eqref{psihatstrong} of equation~\ref{eqforpsihat} (an operation that only requires the pointwise convergence of~$\widehat{\psi}^{(j)}$ to~$\widehat{\psi}$).}
\begin{equation} \label{psihatjstrong}
\begin{aligned}
& \widehat{\psi}^{(j)}(x, t) = \int_\R p(x + \xi_\ell^{(j)}(t + t_j) - \xi_\ell^{(j)}(\sigma + t_j) - y, t - \sigma) \widehat{\psi}^{(j)}(y, \sigma) \, dy \\
& \hspace{10pt} + \int_{\sigma}^t \int_\R p(x + \xi_\ell^{(j)}(t + t_j) - \xi_\ell^{(j)}(\tau + t_j) - y, t - \tau) \left\{ \widehat{g}^{(j)}(y, \tau) - \widehat{b}^{(j)}(y, \tau) \widehat{\psi}^{(j)}(y, \tau) \right\} dy d\tau,
\end{aligned}
\end{equation}
for all~$x \in \R$,~$a_j - t_j \le \sigma < t < b_j - t_j$, and with
\begin{align*}
\widehat{b}^{(j)}(x, t) & := \widetilde{b}^{(j)}(x + x_j - \bar{x} + \xi_\ell^{(j)}(t + t_j) - \xi_\ell^{(j)}(t_j), t), \\
\widehat{g}^{(j)}(x, t) & := \widetilde{g}^{(j)}(x + x_j - \bar{x} + \xi_\ell^{(j)}(t + t_j) - \xi_\ell^{(j)}(t_j), t).
\end{align*}
Also,~\eqref{jtildeortcond} (with~$i = \ell$),~\eqref{psitildegedelta}, and~\eqref{tildepsijgoesto0} respectively yield
\begin{align}
\label{jhatortcond}
\int_{\R} \widehat{\psi}^{(j)}(x, t) w'(x + x_j - \bar{x} - \xi_\ell^{(j)}(t_j)) \, dx  & = 0 \quad \mbox{for every } t \in (a_j - t_j, b_j - t_j), \\
\label{psijhatgedelta}
|\widehat{\psi}^{(j)}(\bar{x}, 0)| & \ge \delta_\star,
\end{align}
and
\begin{equation} \label{hatpsijgoesto0}
\| \widehat{\psi}^{(j)}(\cdot, t) \|_{L^\infty(\R)} \le \frac{1}{j} + \left( \frac{2}{j} + \| W'' \|_{L^\infty(\R)} \right) (t - a_j + t_j) \quad \mbox{for all } t \in (a_j - t_j, b_j - t_j).
\end{equation}

Let now~$\bar{a} \in [-\infty, 0)$ and~$\bar{b} \in [0, +\infty]$ be such that
$$
\lim_{j \rightarrow +\infty} \left( a_{j} - t_{j} \right) = \bar{a} \quad \mbox{and} \quad \lim_{j \rightarrow +\infty} \left( b_{j} - t_{j} \right) = \bar{b},
$$
up to a subsequence. Note that~\eqref{b>a+1} gives that~$\bar{b} - \bar{a} \ge 1$. By~\eqref{Holderhatunifinj} and Ascoli-Arzel\`a theorem, up to extracting a further subsequence,~$\{ \widehat{\psi}^{(j)} \}$ converges locally uniformly in~$\R \times (\bar{a}, \bar{b})$ to a continuous function~$\widehat{\psi}$. From~\eqref{Holderhatunifinj} we get that
\begin{equation} \label{hatpsile1}
\| \widehat{\psi} \|_{L^\infty(\R \times (\bar{a}, \bar{b}))} \le 1.
\end{equation}
By~\eqref{xi0def}, the fact that~$h^{(j)} \in \bar{B}_1(\HH_{a_j, \mu})$,~\eqref{psigbtildele}, and~\eqref{barxdef}, we also see that
\begin{align*}
\lim_{j \rightarrow +\infty} \left( \xi^{(j)}_\ell(t + t_j) - \xi^{(j)}_\ell(\tau + t_j)\right) = 0, \quad \lim_{j \rightarrow +\infty} \widehat{b}^{(j)}(x, t) = W''(w(x)), \quad \mbox{and} \quad \lim_{j \rightarrow +\infty} |\widehat{g}^{(j)}(x, t)| = 0,
\end{align*}
for a.e.~$x \in \R$ and~$\bar{a} < \sigma \le \tau < t < \bar{b}$. Accordingly, using Lebesgue's dominated convergence theorem we may let~$j \rightarrow +\infty$ in~\eqref{psihatjstrong} and obtain that~$\widehat{\psi}$ satisfies
\begin{equation} \label{psihatstrong}
\widehat{\psi}(x, t) = \int_\R p(x - y, t - \sigma) \widehat{\psi}(y, \sigma) \, dy - \int_{\sigma}^t \int_\R p(x - y, t - \tau) W''(w(y)) \widehat{\psi}(y, \tau) \, dy d\tau,
\end{equation}
for all~$x \in \R$ and~$\bar{a} < \sigma < t < \bar{b}$. Moreover, by taking the limit in~\eqref{jhatortcond},~\eqref{psijhatgedelta}, and~\eqref{hatpsijgoesto0}, we obtain
\begin{align}
\label{hatortcond}
\int_{\R} \widehat{\psi}(x, t) w'(x) \, dx & = 0 \quad \mbox{for every } t \in (\bar{a}, \bar{b}), \\
\label{psihatgedelta}
|\widehat{\psi}(\bar{x}, 0)| & \ge \delta_\star,
\end{align}
and
\begin{equation} \label{hatpsigoesto0}
\| \widehat{\psi}(\cdot, t) \|_{L^\infty(\R)} \le \| W'' \|_{L^\infty(\R)} (t - \bar{a}) \quad \mbox{for all } t \in (\bar{a}, \bar{b}).
\end{equation}
Notice that the last expression yields meaningful information only when~$\bar{a} > -\infty$.

We now claim that
\begin{equation} \label{psihatiszero}
\widehat{\psi}(x, t) = 0 \quad \mbox{for all } x \in \R, \, t \in (\bar{a}, \bar{b}).
\end{equation}
As~$0 \in (\bar{a}, \bar{b}]$, it is clear that~\eqref{psihatiszero} would contradict~\eqref{psihatgedelta}. Thus, to finish the proof we only need to prove~\eqref{psihatiszero}. To do this, we distinguish between the two cases~$\bar{a} > -\infty$ and~$\bar{a} = -\infty$.

If~$\bar{a} > -\infty$, we deduce from~\eqref{psihatstrong} and~\eqref{hatpsigoesto0} that~$\widehat{\psi}$ is a mild solution of
$$
\begin{cases}
\partial_t \widehat{\psi} + (-\Delta)^s \widehat{\psi} + W''(w) \widehat{\psi} = 0 & \quad \mbox{in } \R \times (\bar{a}, \bar{b}), \\
\widehat{\psi} = 0 & \quad \mbox{on } \R \times \{ \bar{a} \}.
\end{cases}
$$
Claim~\eqref{psihatiszero} then follows from the uniqueness of mild solutions, established in Proposition~\ref{existinLinftyprop}.

Assume now that~$\bar{a} = -\infty$. Then,~\eqref{psihatstrong} says that~$\widehat{\psi}$ is a mild solution of equation
\begin{equation} \label{psihateqentire}
\partial_t \widehat{\psi} + (-\Delta)^s \widehat{\psi} + W''(w) \widehat{\psi} = 0 \quad \mbox{in } \R \times (\sigma, \bar{b}),
\end{equation}
for every~$\sigma < 0$ (with initial datum at~$t = \sigma$ given by~$\widehat{\psi}(\cdot, \sigma)$ itself). Let
$$
m := \frac{W''(0)}{2} \quad \mbox{and} \quad d(x) := \max \left\{ m, W''(w(x)) \right\}.
$$
We rewrite~\eqref{psihateqentire} as
$$
\partial_t \widehat{\psi} + (-\Delta)^s \widehat{\psi} = G[\widehat{\psi}] \quad \mbox{in } \R \times (\sigma, \bar{b}),
$$
with~$G(x, t, u) := \left( d(x) - W''(w(x)) \right) \widehat{\psi}(x, t) - d(x) u$. Let then~$M > 0$ be sufficiently large to have~$W''(w(x)) > m$ for all~$x \in \R \setminus [-M, M]$. Also let~$\overline{C} > 0$ be such that~$w'(x) \ge \overline{C}^{-1}$ for all~$x \in [-M, M]$. Consider the function
$$
\overline{\psi}(x, t) := \overline{C} w'(x) + e^{-m(t - \sigma)}.
$$
It is easy to see that~$\overline{\psi}$ solves
$$
\begin{cases}
\partial_t \overline{\psi} + (-\Delta)^s \overline{\psi} = \overline{G}[\overline{\psi}] & \quad \mbox{in } \R \times (\sigma, +\infty), \\
\overline{\psi} = \overline{C} w' + 1 & \quad \mbox{on } \R \times \{ \sigma \},
\end{cases}
$$
with~$\overline{G}(x, t, u) := \overline{f}(x, t) - d(x) u$ and~$\overline{f}(x, t) = \overline{C} \left( d(x) - W''(w(x)) \right) w'(x) + (d(x) - m) e^{- m(t- \sigma)}$. Observe that, due to~\eqref{hatpsile1} and our choices of~$d$,~$M$, and~$\overline{C}$, it holds
$$
\overline{f}(x, t) \ge \left( d(x) - W''(w(x) \right) \widehat{\psi}(x, t) \quad \mbox{for all } x \in \R, \, t \in (\sigma, \bar{b}).
$$
Hence,~$G \le \overline{G}$ in~$\R \times (\sigma, \bar{b}) \times \R$. Since we also have that~$\widehat{\psi}(\cdot, \sigma) \le 1 \le \overline{\psi}(\cdot, \sigma)$ in~$\R$, we may apply the comparison principle of Proposition~\ref{compprincforstrongprop} and deduce, recalling the decay estimate~\eqref{w'asympt} for~$w'$, that
$$
\widehat{\psi}(x, t) \le \overline{\psi}(x, t) \le \frac{C}{(1 + |x|)^{1 + 2 s}} + e^{-m(t - \sigma)} \quad \mbox{for all } x \in \R, \, t \in [\sigma, \bar{b}),
$$
for some constant~$C > 0$ depending only on~$s$ and~$W$. By the arbitrariness of~$\sigma < 0$ and since we can also obtain an analogous lower bound for~$\widehat{\psi}$, we conclude that
\begin{equation} \label{psihatspacedecayuniv}
|\widehat{\psi}(x, t)| \le \frac{C}{(1 + |x|)^{1 + 2 s}} \quad \mbox{for all } x \in \R, \, t \in (-\infty, \bar{b}).
\end{equation}
This gives in particular that~$\widehat{\psi} \in C^0((-\infty, \bar{b}); L^2(\R))$, which, along with equation~\eqref{psihateqentire} and the orthogonality condition~\eqref{hatortcond}, enables us to apply Lemma~\ref{L2lem} and deduce that
$$
e^{\lambda t} \int_\R \widehat{\psi}(x, t)^2 \, dx \le e^{\lambda \sigma} \int_{\R} \widehat{\psi}(x, \sigma)^2 \, dx \quad \mbox{for all } -\infty < \sigma < t < \bar{b},
$$
for some constant~$\lambda > 0$ depending only on~$s$ and~$W$. Thanks to~\eqref{psihatspacedecayuniv}, the integral on the right-hand side of the above inequality is bounded uniformly in~$\sigma$. Hence, letting~$\sigma \rightarrow -\infty$ 
we are led to~\eqref{psihatiszero}. The proof of Proposition~\ref{weightedLinftytoLinftyprop} is thus complete.
\end{proof}

Proposition~\ref{weightedLinftytoLinftyprop} provides an important a priori estimate for solutions~$\psi$ of~\eqref{psiprobthm} fulfilling the orthogonality conditions~\eqref{psiort}. The next lemma, when applied with an initial datum~$\psi_0 \in \Adot_{\{ T \}}$, shows that such orthogonality conditions are satisfied if and only if the coefficients~$(c_j)$ on the right-hand side of the equation in~\eqref{psiprobthm} are determined by the system~\eqref{Ac=bt>Tige1}, with~$(b_i)$ as in~\eqref{bidef}.

\begin{lemma} \label{ortequivsystem}
Let~$T \ge 1$,~$\tau \in (0, +\infty]$,~$h \in \bar{B}_1(\HH_{T, \mu})$,~$f \in L^\infty(\R \times (T, T + \tau))$, and~$\psi$ be the mild solution of
\begin{equation} \label{psiprobt0t1}
\begin{dcases}
\partial_t \psi + (-\Delta)^s \psi + W''(z) \psi = f + \sum_{j = 1}^N c_j Z_j & \quad \mbox{in } \R \times (T, T + \tau) \\
\psi = \psi_0 \vphantom{\sum_{j = 1}^N c_j Z_j} & \quad \mbox{in } \R \times \{ T \},
\end{dcases}
\end{equation}
for some initial datum~$\psi_0 \in L^\infty(\R)$ and some coefficients~$c \in L^\infty((T, T + \tau); \R^N)$. Let~$i \in \{ 1, \ldots, N\}$. Then, the map
$$
(T, T + \tau) \ni t \longmapsto \int_{\R} \psi(x, t) Z_i(x, t) \, dx \in \R \quad \mbox{is constant}
$$
if and only if~$c$ satisfies
$$
\sum_{j = 1}^N A_{i j}(t) \, c_j(t) = b_i(t) \quad \mbox{for all } t \in (T, T + \tau),
$$
with~$A_{i j}$ and~$b_i$ respectively given by~\eqref{Aijdef} and~\eqref{bidef}.
\end{lemma}
\begin{proof}
Set
$$
I_i(\sigma) := \int_{\R} \psi(x, \sigma) Z_i(x, \sigma) \, dx \quad \mbox{for } \sigma \in (T, T + \tau).
$$
From the fact that~$\psi$ is a mild solution of~\eqref{psiprobt0t1} it follows that
\begin{equation} \label{psistrongbetweentauandt}
\psi(x, m) = \T_{m - \rho}[\psi(\cdot, \rho)](x) + \int_{\rho}^m \T_{m - \sigma}[G(\cdot, \sigma)](x) \, d\sigma \quad \mbox{for all } x \in \R \mbox{ and } m \in (\rho, T + \tau),
\end{equation}
with
$$
G(x, \sigma) := f(x, \sigma) + \sum_{j = 1}^{N} c_j(\sigma) Z_j(x, \sigma) - W''(z(x, \sigma)) \psi(x, \sigma).
$$
Let now~$\rho, t \in (T, T + \tau)$ be fixed, with~$\rho < t$. We multiply identity~\eqref{psistrongbetweentauandt} (with~$m = t$) by~$Z_i(x, t)$ and integrate as~$x$ ranges in~$\R$. We get
\begin{equation} \label{psiintegratedagainstwi}
I_i(t) = \int_\R \T_{t - \rho}[\psi(\cdot, \rho)](x) Z_i(x, t) \, dx + \int_{\rho}^t \int_\R \T_{t - \sigma}[G(\cdot, \sigma)](x) Z_i(x, t) \, dx d\sigma.
\end{equation}
From the parity of~$p(\cdot, m)$ we deduce that, for any two functions~$U \in L^\infty(\R)$,~$V \in L^1(\R)$ and any~$m > 0$,
\begin{equation} \label{TUV}
\int_\R \T_{m}[U](x) V(x) \, dx = \int_\R \int_\R p(x - y, m) U(y) V(x) \, dx dy = \int_\R \T_{m}[V](y) U(y) \, dy.
\end{equation}
Using~\eqref{TUV} twice (with~$U = \psi(\cdot, \rho)$,~$V = Z_i(\cdot, t)$,~$m = t - \rho$ and~$U = G(\cdot, \sigma)$,~$V = Z_i(\cdot, t)$,~$m = t - \sigma$), identity~\eqref{psiintegratedagainstwi} becomes
\begin{equation} \label{psiagainstwiafterfubini}
I_i(t) = \int_\R \T_{t - \rho}[Z_i(\cdot, t)](y) \psi(y, \rho) \, dy + \int_{\rho}^t \int_\R \T_{t - \sigma}[Z_i(\cdot, t)](y) G(y, \sigma) \, dy d\sigma.
\end{equation}

Next, observe that
$$
\partial_t Z_i(x, \sigma) - (-\Delta)^s Z_i(x, \sigma) = H_i(x, \sigma) \quad \mbox{for all } x \in \R \mbox{ and } \sigma > T,
$$
with
$$
H_i(x, \sigma) := - w''(x - \xi_i(\sigma)) \dot{\xi}_i(\sigma) + W''(w(x - \xi_i(\sigma))) Z_i(x, \sigma).
$$
From this it is not hard to see that~$Z_i$ can be represented as
$$
Z_i(y, m) = \T_{t - m}[Z_i(\cdot, t)](y) - \int_m^t \T_{\ell - m}[H_i(\cdot, \ell)](y) \, d\ell \quad \mbox{for all } y \in \R \mbox{ and } m \in (T, t).
$$
We use this identity twice (with~$m = \rho$ and~$m = \sigma$) to rewrite~\eqref{psiagainstwiafterfubini} as
\begin{align*}
I_i(t) & = I_i(\rho) + \int_\rho^t \int_\R \T_{\ell - \rho}[H_i(\cdot, \ell)](y) \psi(y, \rho) \, dy d\ell \\
& \quad + \int_{\rho}^t \int_\R \left\{ Z_i(y, \sigma) + \int_\sigma^t \T_{\ell - \sigma}[H_i(\cdot, \ell)](y) \, d\ell \right\} G(y, \sigma) \, dy d\sigma.
\end{align*}
Applying again~\eqref{TUV} (this time with~$U = \psi(\cdot, \rho)$,~$V = H_i(\cdot, \ell)$,~$m = \ell - \rho$) and~\eqref{psistrongbetweentauandt} (with~$m = \ell$), after a computation we obtain
\begin{align*}
& I_i(t) - I_i(\rho) = \int_\rho^t \int_\R \Big\{ \psi(x, \sigma) H_i(x, \sigma) + Z_i(x, \sigma) G(x, \sigma) \Big\} \, dx d\sigma \\
& \hspace{10pt} - \int_\rho^t \int_\R \left\{ \int_{\rho}^\ell \T_{\ell - \sigma}[G(\cdot, \sigma)](x) \, d\sigma \right\} H_i(x, \ell) \, dx d\ell + \int_\rho^t \int_\R \left\{ \int_\sigma^t \T_{\ell - \sigma}[H_i(\cdot, \ell)](y) \, d\ell \right\} G(y, \sigma) \, dy d\sigma.
\end{align*}
After a swap in the time variables of integration~($\sigma$ and~$\ell$) and another application of~\eqref{TUV}, we see that the last two integrals are equal and therefore cancel each other out. Hence, we conclude that~$I_i$ is constant in~$(T, T + \tau)$ if and only if
$$
\int_\R \Big\{ \psi(x, t) H_i(x, t) + Z_i(x, t) G(x, t) \Big\} \, dx = 0 \quad \mbox{for all } t \in (T, T + \tau).
$$
The claim of Lemma~\ref{ortequivsystem} then follows after an inspection of the quantity on the left-hand side.
\end{proof}

We are now almost in position to prove Proposition~\ref{mainlinprop}. We will do it via a fixed point argument based on the a priori estimate of Proposition~\ref{weightedLinftytoLinftyprop}, applied with~$g$ given by the right-hand side of the equation in~\eqref{psiprobthm}. To this aim, we need to estimate the~$\| \cdot \|_{\A_T}$ norm of the sum~$\sum_{j = 1}^N c_j Z_j$. The next lemma is a first step in this direction.

\begin{lemma} \label{clem}
Let~$T \ge 1$,~$\tau \in (0, +\infty]$,~$\psi, f \in L^\infty(\R \times (T, T + \tau))$. Let~$A = (A_{i j})$ and~$b = (b_i)$ be given by~\eqref{Aijdef} and~\eqref{bidef}, respectively. Then, there exists a generic constant~$T_0 \ge 1$ such that, if~$T \ge T_0$, there exists a unique solution~$c = (c_j)$ of
\begin{equation} \label{Ac=b}
A(t) \, c(t) = b(t) \quad \mbox{for all } t \in (T, T + \tau).
\end{equation}
If, moreover,~$\psi, f \in \A_{(T, T + \tau)}$, then it holds
\begin{equation} \label{c=fintw+epsj}
c_j(t) = - \gamma \int_{\R} f(x, t) w'(x - \xi_j(t)) \, dx + \varepsilon_j(t) \quad \mbox{for all } t \in (T, T + \tau),
\end{equation}
where~$\gamma$ is as in~\eqref{gammadef} and~$\varepsilon_j: (T, T + \tau) \to \R$ satisfies
\begin{equation} \label{epsjdecay}
|\varepsilon_j(t)| \le C \left( \left\| \psi \right\|_{\A_{(T, T + \tau)}} + \left\| f \right\|_{\A_{(T, T + \tau)}} \right) t^{- \frac{4 s}{1 + 2 s}} \quad \mbox{for all } t \in (T, T + \tau), \, j = 1, \ldots, N,
\end{equation}
for some generic constant~$C$.
\end{lemma}
\begin{proof}
We claim that, if~$T_0$ is large enough, the matrix~$(A_{i j}(t))$ is invertible for every~$t > T \ge T_0$ and therefore~\eqref{Ac=b} is uniquely solvable. To verify this, first notice that, changing variables appropriately,
\begin{equation} \label{Aii=beta}
A_{i i}(t) = \int_{\R} w'(x - \xi_i(t))^2 \, dx = \int_{\R} w'(y)^2 \, dy = \gamma^{-1},
\end{equation}
for every~$i = 1, \ldots, N$. On the other hand, when~$i \ne j$, we change coordinates as before and write
$$
A_{i j}(t) = \int_{\R} w'(x - \xi_i(t)) w'(x - \xi_j(t)) \, dx = \int_{\R} w'(y) w'(y + \xi_i(t) - \xi_j(t)) \, dy.
$$
Let then~$\beta_\star := \inf_{i \ne j} |\beta_i - \beta_j|/4$. Recalling~\eqref{xi0def} and the fact that~$h \in \bar{B}_1(\HH_{T, \mu})$, for~$|y| \le \beta_\star \, t^{\frac{1}{1 + 2 s}}$ and~$t > T \ge T_0$ we have
$$
|y + \xi_i(t) - \xi_i(y)| \ge |\xi_i^0(t) - \xi_j^0(t)| - |h_i(t) - h_j(t)| - |y| \ge 2 \beta_\star \, t^{\frac{1}{1 + 2 s}} - |y| \ge \beta_\star \, t^{\frac{1}{1 + 2 s}},
$$
provided~$T_0$ is sufficiently large and generic. Using this and estimate~\eqref{w'asympt} on~$w'$, we obtain
\begin{equation} \label{Aijdecay}
\begin{aligned}
|A_{i j}(t)| & \le C \int_{\R} \frac{dy}{(1 + |y|)^{1 + 2 s} (1 + |y + \xi_i(t) - \xi_j(t)|)^{1 + 2 s}} \\
& \le C \left\{ \int_{- \beta_\star \, t^{\frac{1}{1 + 2 s}}}^{\beta_\star \, t^{\frac{1}{1 + 2 s}}} \frac{dy}{|y + \xi_i(t) - \xi_j(t)|^{1 + 2 s}} + 2 \int_{\beta_\star \, t^{\frac{1}{1 + 2 s}}}^{+\infty} \frac{dy}{y^{1 + 2 s}} \right\} \le C \, t^{- \frac{2s}{1+2s}},
\end{aligned}
\end{equation}
for all~$t > T \ge T_0$. In particular, by choosing~$T_0$ large enough, the matrix~$(A_{i j}(t))$ has positive determinant and is therefore invertible, for every~$t > T \ge T_0$. Consequently, there exists a unique solution~$c(t)$ to~\eqref{Ac=b}.

We now show that, if~$\psi, f \in \A_{(T, T + \tau)}$, then~$c_j$ satisfies~\eqref{c=fintw+epsj}-\eqref{epsjdecay}. First, we estimate the decay of the~$b_i$'s. By definition~\eqref{Phidef} of~$\Phi$ and a change of variables, we have
\begin{equation} \label{fw'decay}
\begin{aligned}
\int_{\R} |f(x, t)| w'(x - \xi_i(t)) \, dx & \le \left\| f \right\|_{\A_{(T, T + \tau)}} \int_{\R} \Phi(x, t) w'(x - \xi_i(t)) \, dx \\
& \le \left\| f \right\|_{\A_{(T, T + \tau)}} t^{- \frac{2s}{1+2s}} \int_{\R} w'(y) \, dy = \left\| f \right\|_{\A_{(T, T + \tau)}} t^{- \frac{2s}{1+2s}},
\end{aligned}
\end{equation}
as~$w'$ is positive and has integral~$1$. On the other hand, by estimate~\eqref{E0elllePhi} in Lemma~\ref{ElePhilem}, we have
\begin{equation} \label{W''decays}
\int_{\R} \left| \E_{0, i}(x, t) \right| |\psi(x, t)| Z_i(x, t) \, dx \le C \left\| \psi \right\|_{\A_{(T, T + \tau)}} t^{- \frac{4 s}{1 + 2 s}}.
\end{equation}
Finally, using definitions~\eqref{xi0def} and~\eqref{Phidef}, that~$h \in \bar{B}_1(\HH_{T, \mu})$, and that~$w'' \in L^1(\R)$, by estimate~\eqref{w''decay} of Proposition~\ref{w''decayprop}, we easily obtain
$$
|\dot{\xi}_i(t)| \int_{\R} |\psi(x, t)| |w''(x - \xi_i(t))| \, dx \le C \left\| \psi \right\|_{\A_{(T, T + \tau)}} t^{- \frac{4 s}{1+2s}}.
$$
Putting together this,~\eqref{W''decays}, and recalling the definition~\eqref{bidef} of~$b_i$, we conclude that
\begin{equation} \label{bitech}
b_i(t) = - \int_{\R} f(x, t) Z_i(x, t) \, dx + \left\| \psi \right\|_{\A_{(T, T + \tau)}} O \! \left(t^{- \frac{4 s}{1 + 2 s}}\right),
\end{equation}
for every~$i = 1, \ldots, N$.

To deduce from this the decay of~$c_j$, we recall~\eqref{Aii=beta} and~\eqref{Aijdecay} to obtain the following estimate on the inverse of the matrix~$A$:
$$
A^{-1}(t) = \left( \gamma^{-1} I + O \! \left( t^{-\frac{2s}{1+2s}} \right) \right)^{-1} = \gamma I + O \! \left( t^{-\frac{2s}{1+2s}} \right).
$$
Therefore, by~\eqref{bitech} and~\eqref{fw'decay}, we immediately get
$$
c_j(t) = - \gamma \int_{\R} f(x, t) Z_i(x, t) \, dx + \left\{ \left\| \psi \right\|_{\A_{(T, T + \tau)}} + \left\| f \right\|_{\A_{(T, T + \tau)}} \right\} O \! \left(t^{- \frac{4 s}{1 + 2 s}}\right),
$$
which gives~\eqref{c=fintw+epsj}-\eqref{epsjdecay}.
\end{proof}

Given~$f, \psi \in \A_{(T, T + \tau)}$, let~$c$ be the unique solution to system~\eqref{Ac=b}, as given by Lemma~\ref{clem}. Note that~$c$ depends on both~$\psi$ and~$f$. We consider the function~$\CC_f[\psi]$ defined by
\begin{equation} \label{Cfdef}
\CC_f[\psi](x, t) := \sum_{j = 1}^N c_j(t) Z_j(x, t) \quad \mbox{for } x \in \R \mbox{ and } t \in (T, T + \tau).
\end{equation}
In particular, we denote with~$\CC_0[\psi]$ the function corresponding to the choice~$f = 0$.

\begin{lemma} \label{Cfdecaylem}
Let~$T \ge 1$,~$\tau \in (0, +\infty]$,~$h \in \bar{B}_1(\HH_{T, \mu})$, and~$\psi, f \in \A_{(T, T + \tau)}$. Then, there exist two generic constants~$T_0, C_\star \ge 1$ such that
\begin{equation} \label{Cfdecay}
\left\| \CC_f[\psi] \right\|_{\A_{(T, T + \tau)}} \le C_\star \left( T^{- \frac{2 s}{1 + 2 s}} \| \psi \|_{\A_{(T, T + \tau)}} + \| f \|_{\A_{(T, T + \tau)}} \right),
\end{equation}
provided~$T \ge T_0$. In particular,
$$
\left\| \CC_0[\psi] \right\|_{\A_{(T, T + \tau)}} \le C_\star \, T^{- \frac{2 s}{1 + 2 s}} \| \psi \|_{\A_{(T, T + \tau)}}.
$$
\end{lemma}
\begin{proof}
We first claim that
\begin{equation} \label{w'overPsidecay}
|Z_j(x, t)| \le C \, t^{\frac{2s}{1+2s}} \Phi(x, t) \quad \mbox{for all } x \in \R, \, t > T, \mbox{ and } j = 1, \ldots, N.
\end{equation}
To check this, take~$|x| \ge 4 \xi_N^0(t) = 4 \beta_N \, t^{\frac{1}{1 + 2 s}}$. For such a choice, using that~$h \in \bar{B}_1(\HH_{T, \mu})$ we have
$$
|x - \xi_j(t)| \ge |x| - |\xi^0_j(t)| - |h_j(t)| \ge |x| - 2 \beta_N \, t^{\frac{1}{1 + 2 s}} \ge \frac{|x|}{2},
$$
provided~$T_0$ is sufficiently large, and hence
$$
|Z_j(x, t)| = |w'(x - \xi_j(t))| \le \frac{C}{(1 + |x - \xi_j(t)|)^{1 + 2 s}} \le \frac{C}{|x|^{1 + 2 s}} \le C \Phi(x, t),
$$
recalling~\eqref{Phidef}. On the other hand, for~$|x| < 4 \beta_N \, t^{\frac{1}{1 + 2 s}}$, we simply estimate
$$
|Z_j(x, t)| \le C \le C \, t^{\frac{2 s}{1 + 2 s}} \Phi(x, t).
$$
In both cases,~\eqref{w'overPsidecay} follows.

By putting together~\eqref{w'overPsidecay},~\eqref{epsjdecay}, and~\eqref{fw'decay}, inequality~\eqref{Cfdecay} is readily established.
\end{proof}

Thanks to all these results, we are now ready to establish the main result of the subsection.

\begin{proof}[Proof of Proposition~\ref{mainlinprop}]
Let~$T'$ be larger than the positive constants~$T_0$ found in Proposition~\ref{weightedLinftytoLinftyprop} and Lemmas~\ref{clem}-\ref{Cfdecaylem}. Given any~$\psi_0 \in \A_{\{ T' \}}$ and~$g \in \A_{(T', T' + 2)}$, denote with~$\mathbf{R}_{\psi_0}(g) = \mathbf{R}_{T', \psi_0}(g)$ the mild solution of problem~\eqref{psiprob} (with~$T = T'$ and~$\tau = 2$). In view of Lemma~\ref{locweightedLinftytoLinftylem}, we have that~$\mathbf{R}_{\psi_0}$ maps~$\A_{(T', T' + 2)}$ into itself.

For~$f, \psi \in \A_{(T', T' + 2)}$, set~$\mathbf{B}_{f, \psi_0}(\psi) := \mathbf{R}_{\psi_0}(f + \CC_f[\psi])$, with~$\CC_f$ defined by~\eqref{Cfdef}. By Lemma~\ref{Cfdecaylem}, we know that~$\B_{f, \psi_0}$ also maps~$\A_{(T', T' + 2)}$ into itself. We claim that~$\B_{f, \psi_0}$ has a unique fixed point in~$\A_{(T', T' + 2)}$, provided~$T_0$ is large enough. To see this, it suffices to notice that, for~$T_0$ large,~$\B_{f, \psi_0}$ is a contraction. Indeed, given~$\psi^{(1)}, \psi^{(2)} \in \A_{(T', T' + 2)}$, from Lemmas~\ref{locweightedLinftytoLinftylem} and~\ref{Cfdecaylem} we get 
\begin{align*}
& \| \B_{f, \psi_0}(\psi^{(1)}) - \B_{f, \psi_0}(\psi^{(2)}) \|_{\A_{(T', T' + 2)}} \\
& \hspace{15pt} = \| \mathbf{R}_{\psi_0}(f + \CC_f[\psi^{(1)}]) - \mathbf{R}_{\psi_0}(f + \CC_f[\psi^{(2)}]) \|_{\A_{(T', T' + 2)}} = \| \mathbf{R}_0(\CC_f[\psi^{(1)}] - \CC_f[\psi^{(2)}]) \|_{\A_{(T', T' + 2)}} \\
& \hspace{15pt} = \| \mathbf{R}_0(\CC_0[\psi^{(1)} - \psi^{(2)}]) \|_{\A_{(T', T' + 2)}} \le C_2 \| \CC_0[\psi^{(1)} - \psi^{(2)}] \|_{\A_{(T', T' + 2)}} \\
& \hspace{15pt} \le C_2 C_\star T_0^{- \frac{2 s}{1 + 2 s}} \| \psi^{(1)} - \psi^{(2)} \|_{\A_{(T', T' + 2)}} \le \frac{1}{2} \, \| \psi^{(1)} - \psi^{(2)} \|_{\A_{(T', T' + 2)}},
\end{align*}
provided~$T_0$ is sufficiently large. Hence,~$\B_{f, \psi_0}$ is a contraction and, by the Banach fixed point theorem, it admits a unique fixed point~$\psi = \psi[f, \psi_0, T']$ in~$\A_{(T', T' + 2)}$, for all~$T' \ge T_0$.

Let now~$T \ge T_0$,~$\psi_0 \in \Adot_{\{ T \}}$, and~$f \in \A_{T}$. The function~$\psi: \R \times (T, +\infty) \to \R$ defined iteratively by
$$
\psi :=
\begin{cases}
\psi[f, \psi_0, T] & \quad \mbox{in } \R \times (T, T + 1], \\
\psi[f, \psi(\cdot, T + k), T + k] & \quad \mbox{in } \R \times (T + k, T + k + 1], \mbox{ for } k \in \N,
\end{cases}
$$
is a mild solution of problem~\eqref{psiprobthm}, with~$(c_j)$ given by Lemma~\ref{clem}. From Lemma~\ref{ortequivsystem} and the fact that~$\psi_0$ fulfills the orthogonality conditions~\eqref{psi0ort}, we get that the function~$\psi$ satisfies~\eqref{psiort} with~$I = (T, +\infty)$. Accordingly,~$\psi \in \Adot_{(T, T + \tau)}$ for all~$\tau > 0$, and thus Proposition~\ref{weightedLinftytoLinftyprop} and Lemma~\ref{Cfdecaylem} yield that
\begin{align*}
\left\| \psi \right\|_{\A_{T}} & \le C_\sharp \left( \left\| f \right\|_{\A_{T}} + \| \CC_f[\psi] \|_{\A_T} + \| \psi_0 \|_{\A_{\{ T \}}} \right) \\
& \le C_\sharp \left( (C_\star + 1) \left\| f \right\|_{\A_{T}} + C_\star T_0^{- \frac{2 s}{1 + 2 s}} \| \psi \|_{\A_{T}}  + \| \psi_0 \|_{\A_{\{ T \}}} \right).
\end{align*}
By taking~$T_0$ sufficiently large, we can reabsorb the~$\A_T$ norm of~$\psi$ to the left-hand side and conclude that the bound~\eqref{psioverPhiestthm} holds true. In particular,~$\psi \in \Adot_T$.

Let now~$\widetilde{\psi} \in \Adot_T$ be another mild solution of~\eqref{psiprobthm}, for some~$\widetilde{c}: (T, +\infty) \to \R^N$. Lemma~\ref{ortequivsystem} yields that~$\widetilde{c}$ solves system~\eqref{Ac=bt>Tige1}, with~$A$ given by~\eqref{Aijdef} and with~$b$ replaced by the vector-valued function~$\widetilde{b}$ defined by the right-hand side of~\eqref{bidef}, with~$\widetilde{\psi}$ instead of~$\psi$. In view of this,~$\psi - \widetilde{\psi}$ is the mild solution of~\eqref{psiprobthm} with~$f = 0$,~$\psi_0 = 0$, and with~$c_j - \widetilde{c}_j$ in place of~$c_j$. From, say, estimate~\eqref{psioverPhiestthm} it follows that~$\widetilde{\psi} = \psi$. Consequently, the proof of Proposition~\ref{mainlinprop} is complete.
\end{proof}

\subsection{Nonlinear theory for $\psi$}

We can now take advantage of the linear theory that we just developed and of the estimates derived in Subsection~\ref{ENsubsec} to establish Theorem~\ref{nonlinearthm}.

\begin{proof}[Proof of Theorem~\ref{nonlinearthm}]
Let~$T$ be larger than the constant~$T_0 \ge 1$ found in Proposition~\ref{mainlinprop}. For~$f \in \A_T$ and~$\psi_0 \in \Adot_{\{ T \}}$, denote with~$\mathbf{S}_{\psi_0}[f]$ the unique solution to problem~\eqref{psiprobthm} lying in~$\Adot_T$, as given by Proposition~\ref{mainlinprop}. Write
$$
\C_{\psi_0}[\psi] := \mathbf{S}_{\psi_0}[- \NN[\psi] - \E] \quad \mbox{for } \psi \in \A_T,
$$
Thanks to estimates~\eqref{ElePhinew} of Lemma~\ref{ElePhilem} and~\eqref{Npsiest} of Lemma~\ref{Ndecaylem}, we know that~$- \NN[\psi] - \E$ belongs to~$\A_T$ if~$\psi$ does. Consequently,~$\C_{\psi_0}$ maps~$\Adot_T$ into itself. Moreover, it is clear that a function~$\psi \in \Adot_T$ is a mild solution of~\eqref{psiDirprob} if and only if it is a fixed point of~$\C_{\psi_0}$.

To prove that~$\C_{\psi_0}$ has indeed a fixed point, we consider the closed ball
$$
\X := \Big\{ \psi \in \Adot_T : \| \psi \|_{\A_T} \le 2 C_{\E} C_\flat \Big\},
$$
of the Banach space~$\Adot_T$, for a fixed~$T \ge T_0$, and where we denote with~$C_{\E}$ and~$C_\flat$, respectively, the constants~$C$ appearing in estimates~\eqref{ElePhinew} of Lemma~\ref{ElePhilem} and~\eqref{psioverPhiestthm} of Proposition~\ref{mainlinprop}. We claim that
\begin{equation} \label{Kclaim1}
\C_{\psi_0} \mbox{ maps } \X \mbox{ into } \X
\end{equation}
and
\begin{equation} \label{Kclaim2}
\C_{\psi_0} \mbox{ is a contraction in } \X,
\end{equation}
provided~$T_0$ is sufficiently large.

First, we check~\eqref{Kclaim1}. In view of Proposition~\ref{mainlinprop}, we know that, given~$\psi \in \X$,
$$
\| \C_{\psi_0}[\psi] \|_{\A_{T}} = \| \mathbf{S}_{\psi_0}[- \NN[\psi] - \E] \|_{\A_{T}} \le C_\flat \left( \| \NN[\psi] \|_{\A_{T}} + \| \E \|_{\A_{T}} + \| \psi_0 \|_{\A_{\{ T \}}} \right).
$$
Taking advantage of inequalities~\eqref{Npsiest} in Lemma~\ref{Ndecaylem} and~\eqref{ElePhinew} in Lemma~\ref{ElePhilem}, we further estimate
$$
\| \C_{\psi_0}[\psi] \|_{\A_T} \le C_0 \left( T_0^{- \frac{2 s}{1 + 2 s}} \| \psi \|_{\A_{T}}^2 + \| \psi_0 \|_{\A_{\{ T \}}} \right) + C_\flat C_{\E},
$$
for some generic constant~$C_0 > 0$. Using that~$\psi \in \X$, we immediately see from the last inequality that~$\| \C_{\psi_0}[\psi] \|_{\A_T} \le 2 C_{\E} C_\flat$, provided~$T_0$ is sufficiently large and~$\| \psi_0 \|_{\A_{\{ T \}}}$ sufficiently small. Claim~\eqref{Kclaim1} then follows.

We now establish~\eqref{Kclaim2}. Let~$\psi^{(1)}, \psi^{(2)} \in \X$. Using the linearity properties of~$\mathbf{S}_{\hspace{1pt}\cdot}[\, \cdot \,]$, Proposition~\ref{mainlinprop}, inequality~\eqref{Npsi1psi2est} of Lemma~\ref{Ndecaylem}, and the definition of~$\X$, we compute
\begin{align*}
\| \C_{\psi_0}[\psi^{(1)}] - \C_{\psi_0}[\psi^{(2)}] \|_{\A_{T}} & = \| \mathbf{S}_{\psi_0}[- \NN[\psi^{(1)}] - \E]  - \mathbf{S}_{\psi_0}[- \NN[\psi^{(2)}] - \E] \|_{\A_{T}} \\
& = \| \mathbf{S}_0[\NN[\psi^{(2)}] - \NN[\psi^{(1)}]] \|_{\A_{T}} \le C_\flat \| \NN[\psi^{(2)}] - \NN[\psi^{(1)}] \|_{\A_{T}} \\
& \le C \, T_0^{- \frac{2 s}{1 + 2 s}} \| \psi^{(1)} - \psi^{(2)} \|_{\A_{T}},
\end{align*}
for some generic constant~$C > 0$. By taking~$T_0 \ge (2 C)^{(1 + 2 s) / (2 s)}$, we infer that
$$
\| \C_{\psi_0}[\psi^{(1)}] - \C_{\psi_0}[\psi^{(2)}] \|_{\A_{T}} \le \frac{1}{2} \, \| \psi^{(1)} - \psi^{(2)} \|_{\A_{T}},
$$
and~\eqref{Kclaim2} follows.

In view of~\eqref{Kclaim1} and~\eqref{Kclaim2}, by the Banach fixed point theorem, we conclude that there exists a unique fixed point~$\psi$ for~$\C_{\psi_0}$ in~$\X$. Accordingly,~$\psi$ is the unique mild solution of~\eqref{psiDirprob} which lies in~$\Adot_T$ and satisfies~\eqref{psidecay} with~$C = 2 C_{\E} C_\flat$.

To conclude the proof of Theorem~\ref{nonlinearthm}, we are left with showing that, if~$\psi_0$ is odd, then~$\psi$ satisfies property~\eqref{psiodd}. To verify this, let~$\widehat{\psi}(x, t) := - \psi(- x, t)$. We claim that
\begin{equation} \label{hatpsisol}
\widehat{\psi} \mbox{ is a mild solution of problem } \eqref{psiDirprob},
\end{equation}
for some coefficients~$\widehat{c} = (\widehat{c}_j): (T, +\infty) \to \R^N$. Observe that, in view of the unique solvability of~\eqref{psiDirprob}, this would immediately give that~$\widehat{\psi} = \psi$ in~$\R \times [T, +\infty)$, which is~\eqref{psiodd}. To check~\eqref{hatpsisol}, we first notice that, by~\eqref{hodd} and~\eqref{betaodd}, we have that
\begin{equation} \label{xiodd}
\xi_i(t) = - \xi_{N - i + 1}(t) \quad \mbox{for all } t > T \mbox{ and } i = 1, \ldots, N.
\end{equation}
Using this,~\eqref{wodd}, and the periodicity and symmetry of~$W$, it is not hard to verify that the identities
\begin{equation} \label{oddeveniden}
\begin{alignedat}{3}
z(- x, t) & = N - z(x, t), && \\
W''(z(- x, t)) & = W''(z(x, t)), && \\
Z_j(- x, t) & = Z_{N - j + 1}(x, t) && \qquad \mbox{for } j = 1, \ldots, N, \\
\E_k(- x, t) & = - \E_k(x, t) && \qquad \mbox{for } k = 1, 2, \\
\E_{0, j}(- x, t) & = \E_{0, N - j + 1}(x, t) && \qquad \mbox{for } j = 1, \ldots, N, \\
\NN[\psi](- x, t) & = - \NN[\widehat{\psi}](x, t), &&
\end{alignedat}
\end{equation}
hold true for every~$x \in \R$ and~$t > T$. Claim~\eqref{hatpsisol} then easily follows from them and the oddness of~$\psi_0$. The proof of the theorem is thus complete.
\end{proof}

Applying the regularity theory developed in Section~\ref{solsec}, we easily obtain a global~$C^\alpha$ estimate for~$\psi$, when the initial datum~$\psi_0$ is regular. Recall definitions~\eqref{CTnualphadef} and~\eqref{ATalphainit}.

\begin{lemma} \label{Holderpsilem}
Let~$\psi \in \Adot_T$ be the solution of problem~\eqref{psiDirprob}, as per Theorem~\ref{nonlinearthm}. Suppose that the initial datum~$\psi_0$ belongs to~$\Adot^1_{\{ T \}}$. Then,
\begin{equation} \label{Holderestforpsi}
\| \psi \|_{\A_T^\alpha} \le C \left( 1 + \| \psi_0 \|_{\A_{\{ T \}}^1} \right),
\end{equation}
for two generic constants~$\alpha \in (0, 1)$ and~$C > 0$.
\end{lemma}
\begin{proof}
Let~$\alpha \in (0, \min \{ 2 s, 1, 1/(2 s)) \})$ and~$t \ge T$. If~$t \in [T, T + 2]$, then we apply Proposition~\ref{regforstrongprop}\ref{globreg}, along with~\eqref{psidecay}, Lemmas~\ref{ElePhilem},~\ref{Ndecaylem},~\ref{Cfdecaylem}, and Theorem~\ref{nonlinearthm}, obtaining
\begin{align*}
\| \psi \|_{C^\alpha(\R \times [t, t + 1])} & \le C \left\{ \left\| - W''(z) \psi - \NN[\psi] - \E + \sum_{j = 1}^N c_j Z_j \right\|_{L^\infty(\R \times (T, t + 1))} + \| \psi_0 \|_{C^1(\R)} \right\} \\
& \le C \, T^{- \frac{2 s}{1 + 2 s}} \left( \| \psi \|_{\A_T} + T^{ - \frac{2 s}{1 + 2 s}} \| \psi \|_{\A_T}^2 + 1 + \| \psi_0 \|_{\A_{\{ T \}}^1} \right) \le C \left( 1 + \| \psi_0 \|_{\A_{\{ T \}}^1} \right) t^{- \frac{2 s}{1 + 2 s}},
\end{align*}
for some generic constant~$C > 0$. If, on the other hand,~$t > T + 2$, then we use instead Proposition~\ref{regforstrongprop}\ref{intreg} and deduce that
\begin{align*}
\| \psi \|_{C^\alpha(\R \times [t, t + 1])} & \le C \left\{ \left\| - W''(z) \psi - \NN[\psi] - \E + \sum_{j = 1}^N c_j Z_j \right\|_{L^\infty(\R \times (t - 1, t + 1))} + \| \psi(\cdot, t - 1) \|_{L^\infty(\R)} \right\} \\
& \le C \, t^{- \frac{2 s}{1 + 2 s}} \left( 1 + \| \psi \|_{\A_T} \right) \le C \, t^{- \frac{2 s}{1 + 2 s}}.
\end{align*}
In either case, we are led to~\eqref{Holderestforpsi}.
\end{proof}

In light of Theorem~\ref{nonlinearthm}, corresponding to each~$h \in \bar{B}_1(\HH_{T, \mu})$ there exists a unique~$\psi = \Psi[h] \in \Adot_T$ that solves the nonlinear initial value problem~\eqref{psiDirprob} with initial datum~$\psi_0$ and satisfies estimate~\eqref{psidecay}. The next result addresses the continuity properties of such map~$\Psi$, when~$\psi_0 \in \Adot_{\{ T \}}^1$.

\begin{lemma} \label{Psicontlem}
Let~$\psi_0 \in \Adot^1_{\{ T \}}$ and~$\{ h^{(k)} \}_{k \in \N} \subset \bar{B}_1(\HH_{T, \mu})$ be a sequence converging to~$h$ in~$\bar{B}_1(\HH_{T, \mu})$. Then,~$\Psi[h^{(k)}] \rightarrow \Psi[h]$ locally uniformly in~$\R \times [T, +\infty)$.
\end{lemma}
\begin{proof}
Write~$\psi^{(k)} := \Psi[h^{(k)}]$. In view of Lemma~\ref{Holderpsilem}, there exist two constants~$\alpha \in (0, 1)$, depending only on structural quantities, and~$C' > 0$, which also depends on the~$\A_{\{ T \}}^1$ norm of~$\psi_0$, such that~$\| \psi^{(k)} \|_{\A_T^\alpha} \le C'$ for all~$k \in \N$. By Ascoli-Arzel\`a theorem and a standard diagonal procedure, there exists a subsequence~$\{ \psi^{(k_\ell)} \} \subseteq \{ \psi^{(k)} \}$ converging to some~$\psi \in \A_T^\alpha$ locally uniformly~in~$\R \times [T, +\infty)$. Using that~$h^{(k_\ell)} \rightarrow h$ and Lebesgue's dominated convergence theorem, we may let~$\ell \rightarrow +\infty$ in the mild formulations of the initial value problems satisfied by the~$\psi^{(k_\ell)}$'s and obtain that~$\psi$ is a mild solution of~\eqref{psiDirprob} (with coefficients corresponding to~$h$). In addition,~$\psi \in \Adot_T$ and it satisfies~\eqref{psidecay}. Consequently, by the uniqueness statement of Theorem~\ref{nonlinearthm}, we conclude that~$\psi = \Psi[h]$ and thus that the full sequence~$\Psi[h^{(k)}]$ converges to it.
\end{proof}

\begin{remark} \label{LipvsHolrmk}
The previous lemma shows that~$\Psi$ is a continuous map from~$\bar{B}_1(\HH_{T, \mu})$ to~$\Adot_T$, endowed with the~$L^\infty_\loc(\R \times [T, +\infty))$ topology. A more refined argument---similar to the one employed in part~(b) of the proof of~\cite[Proposition~4.1]{dG18}---actually leads to quantitative information on the modulus of continuity of~$\Psi$, with respect to a slightly weaker norm than~$\| \cdot \|_{\A_T}$. It is worth pointing out that~$\Psi$ is Lipschitz when~$s \in (1/2, 1)$, while for~$s \in (0, 1/2]$ it only seems to be H\"older continuous. This loss of regularity is caused by the fact that the norm~$\| \cdot \|_{\HH_{T, \mu}}$ allows the elements of~$\bar{B}_1(\HH_{T, \mu})$ to be unbounded when~$s \in (0, 1/2]$.
\end{remark}

\section{Solving for~$h$. Proof of Theorem~\ref{Msystsolvthm}.} \label{hsec}

\noindent
In this section, we show the existence of a solution~$h \in \bar{B}_1(\HH_{T, \mu})$ of the nonlinear problem~\eqref{hdot+Rh=Fh}, provided its initial datum~$h^{(0)}$ is suitably small. That is, we prove Theorem~\ref{Msystsolvthm}.

Recall that the right-hand side~$F$ of the equation in~\eqref{hdot+Rh=Fh} is a nonlinear vector-valued function of~$h$. In order to solve~\eqref{hdot+Rh=Fh}, we thus proceed to investigate some properties of~$F$.

Throughout the section, we always assume~$\psi$ to be the solution of problem~\eqref{psiDirprob} given by Theorem~\ref{nonlinearthm}, for an initial datum~$\psi_0$ lying in~$\Adot^1_{\{ T \}}$---and, hence, in particular of class~$C^1(\R)$.

\subsection{Properties of the nonlinear term $F$.}

We begin the subsection by studying the decay rate of~$F(t)$ as~$t \rightarrow +\infty$. Notice that the function~$F^{(2)}$ appearing in the definition~\eqref{Fdef} of~$F$ involves the inverse of the matrix~$M$ given by~\eqref{Mdef}. We have the following preliminary result which addresses the invertibility of~$M$ and provides some bounds for its inverse.

\begin{lemma} \label{Mlem}
There exists a generic constant~$T_0 \ge 1$ such that, for every~$T \ge T_0$, the matrix~$M(t)$ defined by~\eqref{Mdef} is invertible for all~$t > T$. In addition, it holds
\begin{equation} \label{Minv}
M^{-1}(t) = \gamma I + O \! \left( t^{- \frac{2 s}{1 + 2 s}} \right) \quad \mbox{for all } t > T.
\end{equation}
\end{lemma}
\begin{proof}
First, recalling~\eqref{Aii=beta} and~\eqref{Aijdecay}, we have that
$$
A(t) = \gamma^{-1} I + O \! \left( t^{- \frac{2 s}{1 + 2 s}} \right).
$$
On the other hand, by definition~\eqref{Phidef}, the fact that~$w'' \in L^1(\R)$ (by Proposition~\ref{w''decayprop}), and estimate~\eqref{psidecay}, we get that
$$
\left| \int_{\R} \psi(x, t) w''(x - \xi_i(t)) \, dx \right| \le t^{- \frac{2 s}{1 + 2 s}} \| \psi \|_{\A_{T}} \int_{\R} |w''(y)| \, dy \le C \, t^{- \frac{2 s}{1 + 2 s}}.
$$
The last two formulas and definition~\eqref{Mdef} yield that
$$
M(t) = \gamma^{-1} I + O \! \left( t^{- \frac{2 s}{1 + 2 s}} \right).
$$
In particular, for~$T_0$ large enough the matrix~$M(t)$ is invertible for every~$t > T \ge T_0$ and~\eqref{Minv} holds true.
\end{proof}

The next proposition deals with the rate of decay of~$F$ and of its H\"older modulus of continuity.

\begin{proposition} \label{Fdecaylem}
Let~$h \in \bar{B}_1(\HH_{T, \mu})$, with~$\mu > 3 s / (1 + 2 s)$. Then, there exist constants~$T_0, C, C_{\psi_0} \ge 1$, and~$\alpha \in (0, 1)$ such that
$$
\begin{alignedat}{3}
|F(t)| & \le C \, t^{- \frac{4 s}{1 + 2 s}} && \quad \mbox{for all } t > T, \\
\sup_{t_1, t_2 \in [t, t + 1]} \frac{|F(t_1) - F(t_2)|}{|t_1 - t_2|^\alpha} & \le C_{\psi_0} \, t^{- \frac{4 s}{1 + 2 s}} && \quad \mbox{for all } t > T,
\end{alignedat}
$$
provided~$T \ge T_0$. The constants~$T_0$, $C$, and~$\alpha$ are generic, while~$C_{\psi_0}$ also depends on~$\| \psi_0 \|_{\A^1_{\{ T \}}}$.
\end{proposition}

Proposition~\ref{Fdecaylem} is an immediate consequence of the next two lemmas, which address, respectively, the decay properties of~$F^{(1)}$ and~$F^{(2)}$---recall that~$F = F^{(1)} + F^{(2)}$.

\begin{lemma} \label{F1lem}
Let~$h \in \bar{B}_1(\HH_{T, \mu})$, with~$\mu > 2 s / (1 + 2 s)$. Then, there exist generic constants~$T_0, C \ge 1$ such that
\begin{equation} \label{F1decay}
|F^{(1)}(t)| + t |\dot{F}^{(1)}(t)| \le C \, t^{- 2 \mu + \frac{2 s}{1 + 2 s}} \quad \mbox{for all } t > T,
\end{equation}
provided~$T \ge T_0$.
\end{lemma}
\begin{proof}
We begin by showing that the bound~\eqref{F1decay} holds for~$|F^{(1)}(t)|$. For~$i = 1, \ldots, N$, we write
\begin{equation} \label{R2ndorderTay}
\begin{aligned}
F^{(1)}_i & = \int_0^1 \langle D_\xi \mathscr{R}_i(\xi^0) - D_\xi \mathscr{R}_i(\xi^0 + \tau h), h \rangle \, d\tau = - \int_0^1 \left\{ \int_0^1 \langle D^2_{\xi \xi} \mathscr{R}_i(\xi^0 + \sigma \tau h) h, h \rangle \, d\sigma \right\} \tau \, d\tau.
\end{aligned}
\end{equation}
Notice that
\begin{equation} \label{DR=}
D_{\xi_k} \mathscr{R}_i(\xi) = \gamma \left\{ \delta_{i k} \sum_{j \ne i} \frac{1}{|\xi_i - \xi_j|^{1 + 2 s}} - \frac{1 - \delta_{i k}}{|\xi_i - \xi_k|^{1 + 2 s}} \right\}
\end{equation}
and
\begin{align*}
D^2_{\xi_k \xi_\ell} \mathscr{R}_i(\xi) & = - (1 + 2 s) \gamma \Bigg\{ \delta_{i k} \delta_{i \ell} \sum_{j \ne i} \frac{\xi_i - \xi_j}{|\xi_i - \xi_j|^{3 + 2 s}} + (1 - \delta_{i k}) \delta_{k \ell} \frac{\xi_i - \xi_k}{|\xi_i - \xi_k|^{3 + 2 s}} \\
& \quad - \delta_{i k} (1 - \delta_{i \ell}) \frac{\xi_i - \xi_\ell}{|\xi_i - \xi_\ell|^{3 + 2 s}} - \delta_{i \ell} (1 - \delta_{i k}) \frac{\xi_i - \xi_k}{|\xi_i - \xi_k|^{3 + 2 s}}  \Bigg\}.
\end{align*}
The last identity leads to the bound
\begin{equation} \label{D2Rdecay}
\left| D^2_{\xi \xi} \mathscr{R}(\xi^0(t) + \sigma \tau h(t)) \right| \le C \, t^{- \frac{2 + 2 s}{1 + 2 s}} \quad \mbox{for all } t \ge T \mbox{ and } \sigma, \tau \in [0, 1].
\end{equation}
By this and~\eqref{R2ndorderTay}, we get
$$
|F^{(1)}(t)| \le C \, t^{- \frac{2 + 2 s}{1 + 2 s}} |h(t)|^2 \le C \| h \|_{\HH_{T, \mu}}^2  t^{- \frac{2 + 2 s}{1 + 2 s} + 2 (1 - \mu)} \quad \mbox{for all } t > T,
$$
from which estimate~\eqref{F1decay} for~$|F^{(1)}(t)|$ immediately follows---recall that~$h \in \bar{B}_1(\HH_{T, \mu})$.

We now address the bound for~$|\dot{F}^{(1)}(t)|$. Differentiating relation~\eqref{R2ndorderTay} with respect to~$t$, we get
\begin{align*}
\dot{F}_i^{(1)} & = - \sum_{k, \ell, m = 1}^N \int_0^1 \left\{ \int_0^1 \bigg( D^3_{\xi_k \xi_\ell \xi_m} \mathscr{R}_i(\xi^0 + \sigma \tau h) \! \left( \dot{\xi}_m^0 + \sigma \tau \dot{h}_m \right) \! h_k h_\ell \bigg) d\sigma \right\} \tau \, d\tau \\
& \quad - 2 \sum_{k, \ell = 1}^N \int_0^1 \left\{ \int_0^1 \bigg( D^2_{\xi_k \xi_\ell} \mathscr{R}_i(\xi^0 + \sigma \tau h) h_k \dot{h}_\ell \bigg) d\sigma \right\} \tau \, d\tau.
\end{align*}
A computation gives that
$$
\left| D^3_{\xi \xi \xi} \mathscr{R}(\xi^0(t) + \sigma \tau h(t)) \right| \le C \, t^{- \frac{3 + 2 s}{1 + 2 s}} \quad \mbox{for all } t \ge T \mbox{ and } \sigma, \tau \in [0, 1].
$$
From this,~\eqref{D2Rdecay}, and the fact that~$h \in \bar{B}_1(\HH_{T, \mu})$, we infer that
$$
|\dot{F}^{(1)}(t)| \le C \| h \|_{\HH_{T, \mu}}^2 \left( t^{- \frac{3 + 2 s}{1 + 2 s} - \frac{2 s}{1 + 2 s} + 2(1 - \mu)} + t^{- \frac{2 + 2 s}{1 + 2 s} + 1 - 2 \mu} \right) \le C \, t^{- \frac{2 + 2 s}{1 + 2 s} + 1 - 2 \mu},
$$
which is the desired bound for~$|\dot{F}^{(1)}(t)|$.
\end{proof}

\begin{lemma} \label{F2lem}
Let~$h \in \bar{B}_1(\HH_{T, \mu})$, with~$\mu > 2 s / (1 + 2 s)$. Then, there exist constants~$T_0, C, C_{\psi_0} \ge 1$, and~$\alpha \in (0, 1)$ such that
\begin{alignat}{3}
\label{F2decay}
|F^{(2)}(t)| & \le C \, t^{- \frac{4 s}{1 + 2 s}} && \quad \mbox{for all } t > T, \\
\label{F2derdecay}
\sup_{t_1, t_2 \in [t, t + 1]} \frac{|F^{(2)}(t_1) - F^{(2)}(t_2)|}{|t_1 - t_2|^\alpha} & \le C_{\psi_0} \, t^{- \frac{4 s}{1 + 2 s}} && \quad \mbox{for all } t > T,
\end{alignat}
provided~$T \ge T_0$. The constants~$T_0$, $C$, and~$\alpha$ are generic, while~$C_{\psi_0}$ also depends on~$\| \psi_0 \|_{\A^1_{\{ T \}}}$.
\end{lemma}

\begin{proof}
We first establish~\eqref{F2decay}. Let
\begin{equation} \label{F2tildedef}
\widetilde{F}^{(2)}(t) := M^{-1}(t) \, d(t).
\end{equation}
In view of~\eqref{Minv}, we have that
\begin{equation} \label{F2tildedecay}
\widetilde{F}^{(2)}(t) = \left( \gamma + O \! \left( t^{- \frac{2 s}{1 + 2 s}} \right) \right) d(t) \quad \mbox{for all } t > T.
\end{equation}
We analyze one by one the terms composing~$d = (d_i)_{i = 1}^N$, as defined by~\eqref{didef}. Estimates~\eqref{E0elllePhi} and~\eqref{EklePhi} in Lemma~\ref{ElePhilem} respectively yield
$$
\int_{\R} \left| \E_{0, i}(x, t) \right| |\psi(x, t)| Z_i(x, t) \, dx \le t^{- \frac{2 s}{1 + 2 s}} \| \psi \|_{\A_T} \int_{\R} \left| \E_{0, i}(x, t) \right| Z_i(x, t) \, dx  \le C \| \psi \|_{\A_{T}} t^{- \frac{4 s}{1 + 2 s}}
$$
and, after a change of variables,
$$
\int_{\R} \left| \E_1(x, t) \right| Z_i(x, t) \, dx \le C \, t^{- \frac{2 s}{1 + 2 s}} \int_{\R} w'(y) \, dy = C \, t^{- \frac{2 s}{1 + 2 s}},
$$
for all~$t > T$ and for some generic constant~$C > 0$. On the other hand,~\eqref{Npsiest} of Lemma~\ref{Ndecaylem} gives
$$
\int_{\R} \left| \NN[\psi](x, t) \right| Z_i(x, t) \, dx \le t^{- \frac{2 s}{1 + 2 s}} \| \NN[\psi] \|_{\A_t} \int_{\R} w'(y) \, dy \le C \| \psi \|_{\A_{T}}^2 t^{- \frac{4 s}{1 + 2 s}}.
$$
Hence, recalling definition~\eqref{didef} and estimate~\eqref{psidecay} of Theorem~\ref{nonlinearthm} we infer that
\begin{equation} \label{dest}
|d(t)| \le C \, t^{- \frac{2 s}{1 + 2 s}} \quad \mbox{for all } t > T,
\end{equation}
and, more precisely, that~\eqref{F2tildedecay} yields, for~$i = 1, \ldots, N$,
\begin{equation} \label{ydot=0.5}
\widetilde{F}^{(2)}_i(t) = \gamma \int_{\R} \E_1(x, t) Z_i(x, t) \, dx + O \! \left( t^{- \frac{4 s}{1 + 2 s}} \right) \quad \mbox{for all } t > T.
\end{equation}

To obtain~\eqref{F2decay}, we need to further analyze the term involving~$\E_1$. First of all, changing variables appropriately, we see that
\begin{equation} \label{E1Zafterchange}
\int_{\R} \E_1(x, t) Z_i(x, t) \, dx = \int_{\R} \widehat{\E}_1(y, t) w'(y) \, dy,
\end{equation}
with
\begin{align*}
\widehat{\E}_1(y, t) := & \,\, \E_1(y + \xi_i(t), t) \\
= & \,\, W' \! \left( w(y) + \sum_{j \ne i} w(y + \xi_i(t) - \xi_j(t)) \right) - W'(w(y)) - \sum_{j \ne i} W'(w(y + \xi_i(t) - \xi_j(t))).
\end{align*}
Now, thanks to the decay estimate~\eqref{wasympt}, definition~\eqref{xi0def}, the fact that~$h \in \bar{B}_1(\HH_{T, \mu})$, and since, by~\eqref{EklePhi},~$|\widehat{\E}_1(\cdot, t)| \le C \, t^{- 2 s / (1 + 2 s)}$, we compute
$$
\int_{\frac{\xi_{i + 1}(t) - \xi_i(t)}{2}}^{+\infty} |\widehat{\E}_1(y, t)| w'(y) \, dy \le \frac{C}{t^{\frac{2 s}{1 + 2 s}}} \left\{ 1 - w \! \left( \frac{\xi_{i + 1}(t) - \xi_i(t)}{2} \right) \right\} \le \frac{C \, t^{- \frac{2 s}{1 + 2 s}}}{(\xi_{i + 1}(t) - \xi_i(t))^{2 s}} \le C \, t^{- \frac{4 s}{1 + 2 s}}.
$$
Analogously,
$$
\int_{-\infty}^{- \frac{\xi_i(t) - \xi_{i - 1}(t)}{2}} |\widehat{\E}_1(y, t)| w'(y) \, dy \le C \, t^{- \frac{4 s}{1 + 2 s}},
$$
and thus, recalling~\eqref{E1Zafterchange}, identity~\eqref{ydot=0.5} becomes
\begin{equation} \label{ydot=2}
\widetilde{F}^{(2)}_i(t) = \gamma \int_{I_i(t)} \widehat{\E}_1(y, t) w'(y) \, dy + O \! \left( t^{- \frac{ 4 s}{1 + 2 s}} \right) \quad \mbox{for all } t > T,
\end{equation}
where we set
$$
I_i(t) := \left( - \frac{\xi_i(t) - \xi_{i - 1}(t)}{2}, \frac{\xi_{i + 1}(t) - \xi_i(t)}{2} \right).
$$

For~$j \ne i$, write
$$
\ell_{i j}(y, t) := \begin{cases}
w(y + \xi_i(t) - \xi_j(t)) - 1 & \quad \mbox{if } j < i \\
w(y + \xi_i(t) - \xi_j(t)) & \quad \mbox{if } j > i
\end{cases}
\qquad \mbox{and} \qquad
\ell_i(y, t) := \sum_{j \ne i} \ell_{i j}(y, t).
$$
Taking advantage of the periodicity of~$W'$, we may express~$\widehat{\E}_1$ as
$$
\widehat{\E}_1(y, t) = W'(w(y) + \ell_i(y, t)) - W'(w(y)) - \sum_{j \ne i} W'(\ell_{i j}(y, t)).
$$
By Taylor expansion and the fact that~$W'(0) = 0$, we have that
$$
W'(w(y) + \ell_i(y, t)) - W'(w(y)) = W''(w(y)) \ell_i(y, t) + \frac{W'''(w(y) + \theta_i(y, t) \ell_i(y, t))}{2} \, \ell_i(y, t)^2
$$
and
$$
W'(\ell_{i j}(y, t)) = W''(0) \ell_{i j}(y, t) + \frac{W'''(\theta_{i j}(y, t) \ell_{i j}(y, t))}{2} \, \ell_{i j}(y, t)^2,
$$
for some functions~$\theta_i, \theta_{i j}: \R \times (T, +\infty) \to [0, 1]$. Note that, for every~$y \in I_i(t)$ and~$j < i$,
\begin{equation} \label{y+xi+xilarge}
y + \xi_i(t) - \xi_j(t) \ge \frac{\xi_i(t) - \xi_{i - 1}(t)}{2} \ge C^{-1} \, t^{\frac{1}{1 + 2 s}},
\end{equation}
and thus, using~\eqref{wasympt},
$$
|\ell_{i j}(y, t)| \le \frac{C}{\left( y + \xi_i(t) - \xi_j(t) \right)^{2 s}} \le C \, t^{- \frac{2s}{1 + 2 s}}.
$$
Since the same bound also holds for every~$j > i$, we conclude that
$$
\widehat{\E}_1(y, t) = - \left\{ W''(0) - W''(w(y)) \right\} \ell_i(y, t) + O \! \left( t^{- \frac{4 s}{1 + 2 s}} \right) \quad \mbox{for all } t > T, \, y \in I_i(t).
$$
By this and the fact that~$w' \in L^1(\R)$, we get that~\eqref{ydot=2} can be simplified to
\begin{equation} \label{ydot=3}
\widetilde{F}^{(2)}_i(t) = - \gamma \int_{I_i(t)} \left\{ W''(0) - W''(w(y)) \right\} \ell_i(y, t) w'(y) \, dy + O \! \left( t^{- \frac{ 4 s}{1 + 2 s}} \right) \quad \mbox{for all } t > T.
\end{equation}

Set now
\begin{equation} \label{Lijdef}
L_{i j}(y, t) := - \frac{1}{2 s W''(0)} \frac{y + \xi_i(t) - \xi_j(t)}{|y + \xi_i(t) - \xi_j(t)|^{1 + 2 s}} \quad \mbox{and} \quad L_i(y, t) := \sum_{j \ne i} L_{i j}(y, t).
\end{equation}
In view of the validity of~\eqref{y+xi+xilarge} for all~$y \in I_i(t)$ and~$j < i$---and of a similar inequality for~$j > i$---, by Proposition~\ref{uimprovasymptprop} we have that
$$
\ell_{i}(y, t) = L_{i}(y, t) + O \! \left( t^{-\frac{4 s}{1 + 2 s}} \right) \quad \mbox{for all } t > T \mbox{ and } y \in I_i(t).
$$
As~$\int_{\R} \left| W''(0) - W''(w(y)) \right| w'(y) \, dy \le C$, equation~\eqref{ydot=3} can be rewritten as
\begin{equation} \label{ydot=4}
\widetilde{F}^{(2)}_i(t) = - \gamma \int_{I_i(t)} \left\{ W''(0) - W''(w(y)) \right\} L_i(y, t) w'(y) \, dy + O \! \left( t^{- \frac{ 4s}{1 + 2 s}} \right) \quad \mbox{for all } t > T.
\end{equation}
A Taylor expansion now yields that
\begin{equation} \label{Lijtaylor}
L_{i j}(y, t) = L_{i j}(0, t) + \frac{1}{W''(0)} \left\{ \frac{y}{|\xi_i(t) - \xi_j(t)|^{1 + 2 s}} + \frac{1 + 2 s}{2} \frac{\vartheta_{i j}(y, t) y + \xi_i(t) - \xi_j(t)}{|\vartheta_{i j}(y, t) y + \xi_i(t) - \xi_j(t)|^{3 + 2 s}} \, y^2 \right\},
\end{equation}
for some~$\vartheta_{i j}: \R \times (T, +\infty) \to [0, 1]$. Since~$w'$ and~$W''(w)$ are even, setting
$$
a_i(t) := \min \left\{ \frac{\xi_i(t) - \xi_{i - 1}(t)}{2}, \frac{\xi_{i + 1}(t) - \xi_i(t)}{2} \right\} \quad \mbox{and} \quad \widehat{I}_i(t) := \left( - a_i(t), a_i(t) \right),
$$
and using that~$|W''(0) - W''(w(y))| \le C (1 + |y|)^{- 2 s}$ and~$0 < w'(y) \le C (1 + |y|)^{- 1 - 2 s}$, we have
\begin{align*}
\left| \int_{I_i(t)} \frac{\left\{ W''(0) - W''(w(y)) \right\} y \, w'(y)}{|\xi_i(t) - \xi_j(t)|^{1 + 2 s}} \, dy \right| & = \left| \int_{I_i(t) \setminus \widehat{I}_i(t)} \frac{\left\{ W''(0) - W''(w(y)) \right\} y \, w'(y)}{|\xi_i(t) - \xi_j(t)|^{1 + 2 s}} \, dy \right| \\
& \le \frac{C}{t} \int_{I_i(t) \setminus \widehat{I}_i(t)} \frac{dy}{|y|^{4 s}} \le C \, t^{- \frac{6 s}{1 + 2 s}}.
\end{align*}
Similarly,
\begin{align*}
\left| \int_{I_i(t)} \left\{ W''(0) - W''(w(y)) \right\} \frac{\vartheta_{i j}(y, t) y + \xi_i(t) - \xi_j(t)}{|\vartheta_{i j}(y, t) y + \xi_i(t) - \xi_j(t)|^{3 + 2 s}} \, y^2 \, w'(y) \, dy \right| & \le \frac{C}{t^{\frac{2 + 2 s}{1 + 2 s}}} \int_{I_i(t)} \frac{dy}{(1 + |y|)^{4 s - 1}} \\
& \le C \, t^{- \frac{4 s}{1 + 2 s}}.
\end{align*}
By using the last two estimates in conjunction with~\eqref{Lijtaylor}, identity~\eqref{ydot=4} becomes
$$
\widetilde{F}^{(2)}_i(t) = - \gamma L_i(0, t) \int_{I_i(t)} \left\{ W''(0) - W''(w(y)) \right\} w'(y) \, dy + O \! \left( t^{- \frac{ 4 s}{1 + 2 s}} \right) \quad \mbox{for all } t > T.
$$
Note that~$|L_i(0, t)| \le C \, t^{- 2 s / (1 + 2 s)}$ and
\begin{align*}
\int_{I_i(t)} \left\{ W''(0) - W''(w(y)) \right\} w'(y) \, dy & = W''(0) \left\{ w \! \left( \frac{\xi_{i + 1}(t) - \xi_i(t)}{2} \right) - w \! \left( - \frac{\xi_i(t) - \xi_{i - 1}(t)}{2} \right) \right\} \\
& \quad - W' \! \left( w \! \left( \frac{\xi_{i + 1}(t) - \xi_i(t)}{2} \right) \right) + W' \! \left( w \! \left( - \frac{\xi_i(t) - \xi_{i - 1}(t)}{2} \right) \right) \\
& = W''(0) + O \! \left( t^{- \frac{2 s}{1 + 2 s}} \right).
\end{align*}
Hence, taking into account definition~\eqref{Lijdef} we conclude that
$$
\widetilde{F}^{(2)}_i(t) = \frac{\gamma}{2 s} \sum_{j \ne i} \frac{\xi_i(t) - \xi_j(t)}{|\xi_i(t) - \xi_j(t)|^{1 + 2 s}} + O \! \left( t^{- \frac{ 4 s}{1 + 2 s}} \right) \quad \mbox{for all } t > T, \, i = 1, \ldots, N,
$$
which, recalling definitions~\eqref{Rmapdef},~\eqref{F2def}, and~\eqref{F2tildedef}, yields estimate~\eqref{F2decay}.

We now deal with~\eqref{F2derdecay}. Let~$t > T$,~$t_1, t_2 \in [t, t + 1]$, and~$i \in \{ 1, \ldots, N\}$ be fixed. Recalling definition~\eqref{F2tildedef} of~$\widetilde{F}^{(2)}$, we write
\begin{equation} \label{Ft1-Ft2}
\widetilde{F}_i^{(2)}(t_1) - \widetilde{F}_i^{(2)}(t_2) = \sum_{j = 1}^N \left\{ \left( M_{i j}^{-1}(t_1) - M_{i j}^{- 1}(t_2) \right) d_j(t_1) + M_{i j}^{-1}(t_2) \Big( d_j(t_1) - d_j(t_2) \Big) \right\}.
\end{equation}

We first claim that
\begin{equation} \label{M-1holderclaim}
|M^{-1}(t_1) - M^{- 1}(t_2)| \le C_{\psi_0} \, t^{- \frac{2 s}{1 + 2 s}} |t_1 - t_2|^\alpha,
\end{equation}
for some generic exponent~$\alpha \in (0, 1)$ and some constant~$C_{\psi_0} \ge 1$ depending only on structural quantities and on~$\| \psi_0 \|_{\A^1_{\{ T \}}}$. To verify~\eqref{M-1holderclaim}, we first observe that, for every~$i, j \in \{ 1, \ldots, N\}$,
$$
M^{-1}_{i j}(t_1) - M^{- 1}_{i j}(t_2) = - \sum_{k, \ell = 1}^N M_{i k}^{-1}(t_1) \left( M_{k \ell}(t_1) - M_{k \ell}(t_2) \right) M_{\ell j}^{-1}(t_2).
$$
By this and~\eqref{Minv}, it is clear that~\eqref{M-1holderclaim} would follow if we establish that
\begin{equation} \label{Mholderclaim}
|M_{i j}(t_1) - M_{i j}(t_2)| \le C_{\psi_0} \, t^{- \frac{2 s}{1 + 2 s}} |t_1 - t_2|^\alpha.
\end{equation}
On the one hand, recalling~\eqref{Aijdef}, we simply compute
\begin{equation} \label{Aholder}
\begin{aligned}
|A_{i j}(t_1) - A_{i j}(t_2)| & \le \int_{\R} w'(x - \xi_i(t_1)) \left| w'(x - \xi_j(t_1)) - w'(x - \xi_j(t_2)) \right| \, dx \\
& \quad + \int_{\R} \left| w'(x - \xi_i(t_1)) - w'(x - \xi_i(t_2)) \right| w'(x - \xi_j(t_2)) \, dx \\
& \le C \Big\{ |\xi_j(t_1) - \xi_j(t_2)| + |\xi_i(t_1) - \xi_i(t_2)| \Big\} \int_\R w'(y) \, dy \\
& \le C \left| \int_{t_1}^{t_2} \left| \dot{\xi}^0(\tau) + \dot{h}(\tau) \right| d\tau \right| \le C \, t^{- \frac{2 s}{1 + 2 s}} |t_1 - t_2|^\alpha,
\end{aligned}
\end{equation}
where in the last inequality we used that~$|t_1 - t_2| \le 1$. On the other hand, we have
\begin{align*}
& \left| \int_{\R} \psi(x, t_1) w''(x - \xi_i(t_1)) \, dx - \int_{\R} \psi(x, t_2) w''(x - \xi_i(t_2)) \, dx \right| \\
& \hspace{20pt} \le \int_{\R} |\psi(x, t_1) - \psi(x, t_2)| |w''(x - \xi_i(t_1))| \, dx + \int_{\R} |\psi(x, t_2)| |w''(x - \xi_i(t_1)) - w''(x - \xi_i(t_2))| \, dx.
\end{align*}
Taking advantage of Lemma~\ref{Holderpsilem} and recalling definition~\eqref{CTnualphadef}, we have
\begin{align*}
\int_{\R} |\psi(x, t_1) - \psi(x, t_2)| |w''(x - \xi_i(t_1))| \, dx & \le C \| \psi \|_{\A_T^\alpha} \, t^{- \frac{2 s}{1+ 2 s}} |t_1 - t_2|^\alpha \int_\R |w''(y)| \, dy \\
& \le C_{\psi_0} \, t^{- \frac{2 s}{1+ 2 s}} |t_1 - t_2|^\alpha,
\end{align*}
while~\eqref{psidecay} and estimate~\eqref{w'''decay} in Proposition~\ref{w''decayprop} give that
\begin{align*}
& \int_{\R} |\psi(x, t_2)| |w''(x - \xi_i(t_1)) - w''(x - \xi_i(t_2))| \, dx \\
& \hspace{80pt} \le C \| \psi \|_{\A_T} t^{- \frac{2 s}{1+ 2 s}} \int_\R \left| \int_{t_1}^{t_2} |w'''(x - \xi_i(\tau))| \left| \dot{\xi}^0_i(\tau) + \dot{h}_i(\tau) \right| d\tau \right| dx \\
& \hspace{80pt} \le C \, t^{- \frac{4 s}{1+ 2 s}} |t_1 - t_2| \int_\R |w'''(y)| \, dy \le C \, t^{- \frac{4 s}{1 + 2 s}} |t_1 - t_2|^\alpha.
\end{align*}
Accordingly,
$$
\left| \int_{\R} \psi(x, t_1) w''(x - \xi_i(t_1)) \, dx - \int_{\R} \psi(x, t_2) w''(x - \xi_i(t_2)) \, dx \right| \le C_{\psi_0} \, t^{- \frac{2 s}{1 + 2 s}} |t_1 - t_2|^\alpha
$$
Recalling definition~\eqref{Mdef}, the combination of this and~\eqref{Aholder} leads us to estimate~\eqref{Mholderclaim} and thus to claim~\eqref{M-1holderclaim}.

In light of~\eqref{M-1holderclaim},~\eqref{dest}, and~\eqref{Minv}, we can rewrite~\eqref{Ft1-Ft2} as
\begin{equation} \label{Ft1-Ft2bis}
\widetilde{F}^{(2)}(t_1) - \widetilde{F}^{(2)}(t_2) = \left\{ \gamma I + O \! \left( t^{- \frac{2 s}{1 + 2 s}} \right) \right\} \Big( d(t_1) - d(t_2) \Big) + |t_1 - t_2|^\alpha \, O \! \left( t^{- \frac{4 s}{1 + 2 s}} \right).
\end{equation}
Thanks to Lemmas~\ref{ElePhilem},~\ref{Ndecaylem}, and~\ref{Holderpsilem}, we have
\begin{align*}
|\E_1(x, t_1) - \E_1(x, t_2)| & \le C \, t^{- \frac{4 s}{1 + 2 s}} |t_1 - t_2|^\alpha, \\
|\E_{0, i}(x, t_1) - \E_{0, i}(x, t_2)| & \le C \, t^{- \frac{2 s}{1 + 2 s}} |t_1 - t_2|^\alpha. \\
|\NN[\psi](x, t_1) - \NN[\psi](x, t_2)| & \le C_{\psi_0} \, t^{- \frac{4 s}{1 + 2 s}} |t_1 - t_2|^\alpha, \\
|\E_1(x, t_1)| \le C \, t^{- \frac{2 s}{1 + 2 s}}, \quad & \mbox{and} \quad |\NN[\psi](x, t_1)| \le C \, t^{- \frac{4 s}{1 + 2 s}}.
\end{align*}
Moreover, by~\eqref{w''decay},
$$
\int_{\R} |Z_i(x, t_1) - Z_i(x, t_2)| \, dx \le \int_\R \left| \int_{\xi_i(t_2)}^{\xi_i(t_1)} w''(x - y) \, dy \right| dx \le C |\xi_i(t_1) - \xi_i(t_2)| \le C \, t^{- \frac{2 s}{1 + 2 s}} |t_1 - t_2|^\alpha.
$$
Recalling definition~\eqref{didef} and exploiting the above bounds in combination with~\eqref{E0elllePhi}, the boundedness of~$\E_{0, i}$, and, once again, Lemma~\ref{Holderpsilem}, we obtain
$$
|d(t_1)- d(t_2)| \le C_{\psi_0} \, t^{- \frac{ 4 s}{1 + 2 s}} |t_1 - t_2|^\alpha.
$$
Hence, from~\eqref{Ft1-Ft2bis} we get that
$$
|\widetilde{F}^{(2)}(t_1) - \widetilde{F}^{(2)}(t_2)| \le C_{\psi_0} \, t^{- \frac{4 s}{1 + 2 s}} |t_1 - t_2|^\alpha.
$$
Since,~\eqref{DR=} yields
$$
|\partial_t \mathscr{R}_i(\xi(t))| \le C \sum_{j = 1}^N \left| D_{\xi_j} \mathscr{R}_i(\xi(t)) \right| |\dot{\xi}_j(t)| \le C \, t^{- \frac{1 + 4 s}{1 + 2 s}},
$$
in view of~\eqref{F2def}, we conclude that estimate~\eqref{F2decay} holds true.
\end{proof}

The following lemma addresses an important symmetry property of~$F$, which holds true under the assumption that the initial datum~$\psi_0$ for~\eqref{psiDirprob} is odd.

\begin{lemma} \label{Foddlem}
Let~$h \in \bar{B}_1(\HH_{T, \mu})$ and suppose that~$\psi_0$ is odd. Then,
$$
F_{N - i + 1}(t) = - F_i(t) \quad \mbox{for all } t > T \mbox{ and } i = 1, \ldots, N.
$$
\end{lemma}
\begin{proof}
First, notice that from the symmetry relation~\eqref{xiodd}, formula~\eqref{DR=}, and definitions~\eqref{Rmapdef} and~\eqref{F1def}, we obtain that~$F^{(1)}_{N - i + 1}(t) = - F^{(1)}_i(t)$ and~$\mathscr{R}_i(\xi(t)) = - \mathscr{R}_{N - i + 1}(\xi(t))$ for all~$t > T$ and~$i = 1, \ldots, N$. To conclude the proof, we thus need to show that the same relation holds for~$\widetilde{F}^{(2)}$ as well, namely that
\begin{equation} \label{F2tildeodd}
\widetilde{F}^{(2)}_i(t) = - \widetilde{F}^{(2)}_{N - i + 1}(t) \quad \mbox{for all } t > T \mbox{ and } i = 1, \ldots, N.
\end{equation}
Recall definitions~\eqref{Fdef},~\eqref{F2def}, and~\eqref{F2tildedef}.

In order to establish~\eqref{F2tildeodd}, we observe that, by Theorem~\ref{nonlinearthm} and the fact that~$\psi_0$ is odd, the solution~$\psi$ satisfies~\eqref{psiodd}. Using this, identities~\eqref{oddeveniden}, and definitions~\eqref{Mdef}-\eqref{didef}, a simple computation reveals that~$M_{N - i + 1 \, N - j + 1}(t) = M_{i j}(t)$ and~$d_{N - i + 1}(t) = - d_i(t)$ for all~$t > T$ and~$i, j = 1, \ldots, N$. From this, claim~\eqref{F2tildeodd} immediately follows. The lemma is thus proved.
\end{proof}

\subsection{Linear theory for~$h$}

In this subsection, we develop a solvability theory for the linear counterpart of system~\eqref{hdot=f}, for a right-hand side belonging to~$\FF_{T, \mu_0}^{\alpha_0}$. Recalling definitions~\eqref{HTmudef} and~\eqref{FTmudef}, we have the following statement.

\begin{proposition} \label{linearhprop}
Let~$T \ge 1$,~$\alpha_0 \in (0, 1)$, and~$\mu_0 \in (0, 1 + \delta)$, with~$\delta > 0$ as in~\eqref{deltadef}. Let~$h^{(0)} \in \R^N$ be satisfying~\eqref{h0odd}. Then, for every~$f \in \FF^{\alpha_0}_{T, \mu_0}$, there exists a unique solution~$h = \mathbf{H}_{h^{(0)}}[f] \in \HH_{T, \mu_0}$ of
$$
\begin{cases}
\dot{h} + D_\xi \mathscr{R}(\xi^0) h = f & \quad \mbox{in } (T, +\infty)\\
h(T) = h^{(0)}. &
\end{cases}
$$
In addition, the map~$\mathbf{H}_{h^{(0)}}: \FF^{\alpha_0}_{T, \mu_0} \to \HH_{T, \mu_0}$ is continuous and it holds
\begin{equation} \label{Hmapbounded}
\| \mathbf{H}_{h^{(0)}}[f] \|_{\HH_{T, \mu_0}} \le C \left( \| f \|_{\FF_{T, \mu_0}} + T^{\mu_0 - 1} |h^{(0)}| \right) \quad \mbox{for all } f \in \FF^{\alpha_0}_{T, \mu_0},
\end{equation}
for some generic constant~$C > 0$.
\end{proposition}
\begin{proof}
By standard ODE theory, there exists a unique solution~$h \in C^1([T, T_1); \R^N)$ of
\begin{equation} \label{hdot+Rh=fuptoT1}
\begin{cases}
\dot{h} + D_\xi \mathscr{R}(\xi^0(t)) h = f \quad \mbox{in } (T, T_1)\\
h(T) = h^{(0)},
\end{cases}
\end{equation}
up to a certain maximal time~$T_1 \in (T, +\infty]$. In consequence of the uniqueness of~$h$ and of the facts that~$f$ and~$h^{(0)}$ respectively fulfill the symmetry assumptions~\eqref{fodd} and~\eqref{h0odd}, it is easy to see that~$h$ satisfies
\begin{equation} \label{hodduptoT1}
h_i(t) = - h_{N - i + 1}(t) \quad \mbox{for all } t \in (T, T_1) \mbox{ and } i = 1, \ldots, N.
\end{equation}

We now claim that there exists a generic constant~$C > 0$ such that
\begin{equation} \label{hest}
t^{\mu_0 - 1} |h(t)| + t^{\mu_0} |\dot{h}(t)| \le C \left( \| f \|_{\FF_{T, \mu_0}} + T^{\mu_0 - 1} |h^{(0)}| \right) \quad \mbox{for all } t \in (T, T_1).
\end{equation}
Note that this would imply in particular that~$T_1 = +\infty$. As a result,~$h$ would belong to~$\HH_{T, \mu_0}$ and the map~$\mathbf{H}_{h^{(0)}}: \FF_{T, \mu_0}^{\alpha_0} \to \HH_{T, \mu_0}$ would satisfy~\eqref{Hmapbounded}.

To prove~\eqref{hest}, we pick a vector~$\eta \in \R^N$ and observe that, taking advantage of~\eqref{DR=},
\begin{align*}
\langle D_\xi \mathscr{R}(\xi^0) \eta, \eta \rangle & = \gamma \sum_{i, k = 1}^N \left\{ \delta_{i k} \sum_{j \ne i} \frac{1}{|\xi_i^0 - \xi_j^0|^{1 + 2 s}} - \frac{1 - \delta_{i k}}{|\xi_i^0 - \xi_k^0|^{1 + 2 s}} \right\} \! \eta_i \eta_k = \gamma \sum_{\substack{1 \le i, j \le N \\ i \ne j}} \frac{\eta_i^2 - \eta_i \eta_j}{|\xi^0_i - \xi^0_j|^{1 + 2 s}} \\
& = \frac{\gamma}{2} \sum_{\substack{1 \le i, j \le N \\ i \ne j}} \frac{\eta_i^2 + \eta_j^2 - 2 \eta_i \eta_j}{|\xi^0_i - \xi^0_j|^{1 + 2 s}} = \gamma \sum_{1 \le i < j \le N} \frac{|\eta_i - \eta_j|^2}{|\xi^0_i - \xi^0_j|^{1 + 2 s}}.
\end{align*}
Recalling Proposition~\ref{explicitprop}, we get
$$
\langle D_\xi \mathscr{R}(\xi^0(t)) \eta, \eta \rangle = \frac{\gamma}{t} \sum_{1 \le i < j \le N} \frac{|\eta_i - \eta_j|^2}{(\beta_j - \beta_i)^{1 + 2 s}} \ge \frac{\gamma}{(2 \beta_N)^{1 + 2 s} \, t} \sum_{i, j = 1}^N |\eta_i - \eta_j|^2.
$$
Assuming now~$\eta$ to satisfy~$\sum_{i = 1}^N \eta_i = 0$, we see that
$$
\sum_{i, j = 1}^N |\eta_i - \eta_j|^2 = \sum_{i, j = 1}^N \left( \eta_i^2 + \eta_j^2  - 2 \eta_i \eta_j \right) = 2 N \sum_{i = 1}^N \eta_i^2 - 2 \left( \sum_{i = 1}^N \eta_i \right)^{\! 2} = 2 N |\eta|^2,
$$
which leads us to conclude that
\begin{equation} \label{systemiselliptic}
\langle D_\xi \mathscr{R}(\xi^0(t)) \eta, \eta \rangle \ge \frac{2 \delta}{t} \, |\eta|^2 \quad \mbox{for all } t > T \mbox{ and } \eta \in \R^N \mbox{ satisfying } \sum_{i = 1}^N \eta_i = 0,
\end{equation}
with~$\delta > 0$ given by~\eqref{deltadef}.

We will now apply~\eqref{systemiselliptic} with~$\eta = h$. First, observe that condition~\eqref{hodduptoT1} implies in particular that~$\sum_{i = 1}^N h_i(t) = 0$ for all~$t \in (T, T_1)$. Hence, by multiplying the equation in~\eqref{hdot+Rh=fuptoT1} against~$h$ and making use of~\eqref{systemiselliptic}, we obtain
$$
\frac{d}{dt} |h|^2 = 2 \langle \dot{h}, h \rangle = - 2 \langle D_\xi \mathscr{R}(\xi^0) h, h \rangle + 2 \langle f, h \rangle \le -\frac{2 \delta}{t} |h|^2 + \frac{t}{2 \delta} |f|^2 \quad \mbox{for all } t \in (T, T_1),
$$
where the last estimate follows from Cauchy-Schwarz and the weighted Young's inequalities. Thus,
$$
\frac{d}{dt} \left( t^{2 \delta} |h|^2 \right) = t^{2 \delta} \left( \frac{2 \delta}{t} |h|^2 + \frac{d}{dt} |h|^2 \right) \le \frac{t^{2 \delta + 1}}{2 \delta} \, |f|^2 \quad \mbox{for all } t \in (T, T_1).
$$
By integrating this inequality, using that~$h(T) = h^{(0)}$, and recalling that, by our hypothesis on~$\delta$, it holds~$2 \delta + 1 - 2 \mu_0 > -1$, we find that, for every~$t \in (T, T_1)$,
$$
t^{2 \delta} |h(t)|^2 - T^{2 \delta} |h^{(0)}|^2 \le \frac{1}{2 \delta} \int_T^t \tau^{2 \delta + 1 - 2 \mu_0} \left( \tau^{\mu_0} |f(\tau)| \right)^2 \, d\tau \le C \, t^{2 \delta + 2 - 2 \mu_0} \| f \|_{\FF_{T, \mu_0}}^2, 
$$
which gives
\begin{equation} \label{estforhtech}
t^{\mu_0 - 1} |h(t)| \le C \left( \| f \|_{\FF_{T, \mu_0}} + T^{\mu_0 - 1} |h^{(0)}| \right) \quad \mbox{for all } t \in (T, T_1).
\end{equation}

To establish~\eqref{hest}, we only need to control~$|\dot{h}|$. Using once again the equation in~\eqref{hdot+Rh=fuptoT1},~\eqref{DR=}, and~\eqref{estforhtech} we have
$$
|\dot{h}(t)| \le |f(t)| + |D_\xi \mathscr{R}(\xi^0(t))| |h(t)| \le \frac{t^{\mu_0} |f(t)| + C \, t^{\mu_0 - 1} |h(t)|}{t^{\mu_0}} \le \frac{C}{t^{\mu_0}} \left( \| f \|_{\FF_{T, \mu_0}} + T^{\mu_0 - 1} |h^{(0)}| \right),
$$
for all~$t \in (T, T_1)$. This, combined with~\eqref{estforhtech}, gives~\eqref{hest}.

The continuity of~$\mathbf{H}_{h^{(0)}}$ follows from~\eqref{Hmapbounded}, observing that~$\mathbf{H}_{h^{(0)}}[f^{(1)}] - \mathbf{H}_{h^{(0)}}[f^{(2)}] = \mathbf{H}_0[f^{(1)} - f^{(2)}]$ for every~$f^{(1)}, f^{(2)} \in \FF_{T, \mu_0}^{\alpha_0}$. The proof of the proposition is thus concluded.
\end{proof}

\subsection{Nonlinear theory for $h$}

In this conclusive subsection, we show that, when~$h^{(0)}$ is small, the nonlinear system~\eqref{hdot=f} admits a solution~$h \in \bar{B}_1(\HH_{T, \mu})$, thus proving Theorem~\ref{Msystsolvthm}. Our argument is based on the Schauder fixed point theorem. To apply it, we first need a compactness result.

Given~$h \in \bar{B}_1(\HH_{T, \mu})$, let~$F = \mathbf{F}[h]$ be the vector-valued function defined by~\eqref{Fdef}-\eqref{F2def}. Proposition~\ref{Fdecaylem} and Lemma~\ref{Foddlem} ensure that, if~$\mu > 3 s / (1 + 2 s)$, then~$\mathbf{F}$ maps~$\bar{B}_1(\HH_{T, \mu})$ into~$\FF_{T, 4 s / (1 + 2 s)}^{\alpha}$ and the estimates
\begin{equation} \label{Fhest}
\| \mathbf{F}[h] \|_{\FF_{T, \frac{4 s}{1 + 2 s}}} \le C \quad \mbox{for all } h \in \bar{B}_1(\HH_{T, \mu})
\end{equation}
and
\begin{equation} \label{Fhderest}
\| \mathbf{F}[h] \|_{\FF_{T, \frac{4 s}{1 + 2 s}}^{\alpha}} \le C_{\psi_0} \quad \mbox{for all } h \in \bar{B}_1(\HH_{T, \mu})
\end{equation}
hold true for some~$\alpha \in (0, 1)$,~$C, C_{\psi_0} \ge 1$, and provided~$T$ is larger than a generic constant~$T_0 \ge 1$. In addition to this, we have the following lemma.

\begin{lemma} \label{Fcompactlem}
Let~$\mu > 3 s / (1 + 2 s)$ and~$\mu_0 \in (0, 4 s / (1 + 2 s))$. Then, there exists a generic~$\alpha_0 \in (0, 1)$ for which the map~$\mathbf{F}: \bar{B}_1(\HH_{T, \mu}) \to \FF_{T, \mu_0}^{\alpha_0}$ is compact.
\end{lemma}

The proof of Lemma~\ref{Fcompactlem} will be an easy consequence of the following general fact.

\begin{lemma} \label{compemblem}
Let~$0 < \mu_0 < \mu_1$ and~$0 < \alpha_0 < \alpha_1 < 1$. Then, the inclusion~$\FF_{T, \mu_1}^{\alpha_1} \hookrightarrow \FF_{T, \mu_0}^{\alpha_0}$ is compact.
\end{lemma}
\begin{proof}
Let~$C_\star > 0$ and~$\{ f^{(k)} \} \subset \FF_{T, \mu_1}^{\alpha_1}$ be a sequence satisfying~$\| f^{(k)} \|_{\FF_{T, \mu_1}^{\alpha_1}} \le C_\star$ for every~$k \in \N$. By Ascoli-Arzel\`a theorem, a simple interpolation inequality, and a standard diagonal argument, there exists a subsequence~$\{ f^{(k_\ell)} \}$ of~$\{ f^{(k)} \}$ converging to a function~$f$ in~$C_\loc^{\alpha_0}([T, +\infty); \R^N)$ as~$\ell \rightarrow +\infty$. Clearly,~$f \in \FF_{T, \mu_1}^{\alpha_1} \subset \FF_{T, \mu_0}^{\alpha_0}$ with~$\| f \|_{\FF_{T, \mu_1}^{\alpha_1}} \le C_\star$.

Let now~$\varepsilon > 0$. Pick~$T_\varepsilon \ge T$ in a way that~$4 C_\star T_\varepsilon^{- (\mu_1 - \mu_0)} \le \varepsilon$. Since~$f^{(k_\ell)} \to f$ in~$C_\loc^{\alpha_0}([T, +\infty); \R^N)$, there exists~$L_\varepsilon \in \N$ such that~$\| f^{(k_\ell)} - f \|_{C^{\alpha_0}([T, T_\varepsilon + 1]; \R^N)} \le T_\varepsilon^{-\mu_0} \varepsilon / 2$ for all~$\ell \ge L_\varepsilon$. Accordingly,
\begin{align*}
\| f^{(k_\ell)} - f \|_{\FF_{T, \mu_0}^{\alpha_0}} & \le \sup_{t \in (T, T_\varepsilon)} \left( t^{\mu_0} \, |f^{(k_\ell)}(t) - f(t)| \right) + \sup_{t \in (T, T_\varepsilon)} \left( t^{\mu_0} \, [f^{(k_\ell)} - f]_{C^{\alpha_0}([t, t + 1]; \R^N)} \right) \\
& \quad + \sup_{t > T_\varepsilon} \left( t^{\mu_0} \, |f^{(k_\ell)}(t) - f(t)| \right) + \sup_{t > T_\varepsilon} \left( t^{\mu_0} \, [f^{(k_\ell)} - f]_{C^{\alpha_0}([t, t + 1]; \R^N)} \right) \\
& \le T_\varepsilon^{\mu_0} \, \| f^{(k_\ell)} - f \|_{C^{\alpha_0}([T, T_\varepsilon + 1]; \R^N)} + T_\varepsilon^{- (\mu_1 - \mu_0)} \left( \| f^{(k_\ell)} \|_{\FF_{T, \mu_1}^{\alpha_1}} + \| f \|_{\FF_{T, \mu_1}^{\alpha_1}} \right) \le \varepsilon.
\end{align*}
for all~$\ell \ge L_\varepsilon$. As~$\varepsilon$ is an arbitrary positive number, we conclude that~$f^{(k_\ell)} \rightarrow f$ in~$\FF_{T, \mu_0}^{\alpha_0}$ as~$\ell \rightarrow +\infty$ and therefore that the inclusion~$\FF_{T, \mu_1}^{\alpha_1} \hookrightarrow \FF_{T, \mu_0}^{\alpha_0}$ is compact.
\end{proof}

\begin{proof}[Proof of Lemma~\ref{Fcompactlem}]
Let~$\mu_1 \in (\mu_0, 4 s / (1 + 2 s))$ and~$\alpha_1 \in (0, \alpha)$ be fixed, with~$\alpha$ as in~\eqref{Fhderest}. We claim that~$\mathbf{F}: \bar{B}_1(\HH_{T, \mu}) \to \FF_{T, \mu_1}^{\alpha_1}$ is continuous. Notice that, once we have this, the conclusion of the lemma would then follow from Lemma~\ref{compemblem}.

Let~$\{ h^{(k)} \} \subset \bar{B}_1(\HH_{T, \mu})$ be a sequence converging to some~$h$ in~$\bar{B}_1(\HH_{T, \mu})$. In view of estimate~\eqref{Fhderest}, we know that~$\| \mathbf{F}[h^{(k)}] \|_{\FF_{T, 4 s / (1 + 2 s)}^\alpha}$ is bounded uniformly in~$k$. Thanks to Lemma~\ref{compemblem}, there exists then a diverging sequence~$\{ k_\ell \}$ along which~$\mathbf{F}[h^{(k_\ell)}]$ converges to some~$\widehat{\mathbf{F}}$ in~$\FF_{T, \mu_1}^{\alpha_1}$ as~$\ell \rightarrow +\infty$.

Writing now~$\mathbf{F} = \mathbf{F}^{(1)} + \mathbf{F}^{(2)}$ as in~\eqref{Fdef}-\eqref{F2def}, it is clear that~$\mathbf{F}^{(1)}[h^{(k)}](t) \rightarrow \mathbf{F}^{(1)}[h](t)$ as~$k \rightarrow +\infty$ for every~$t > T$. Using Lebesgue's dominated convergence theorem along with the fact that~$\Psi[h^{(j)}] \rightarrow \Psi[h]$ in~$\R \times (T, +\infty)$, by Lemma~\ref{Psicontlem}, it is also easy to see that~$\mathbf{F}^{(2)}[h^{(k)}](t)$ converges to~$\mathbf{F}^{(2)}[h](t)$ for every~$t > T$. Consequently,~$\widehat{\mathbf{F}} = \mathbf{F}[h]$ and the lemma is proved.
\end{proof}

We can now establish Theorem~\ref{Msystsolvthm}.

\begin{proof}[Proof of Theorem~\ref{Msystsolvthm}]
Write
$$
\K_{h^{(0)}}[h] := \mathbf{H}_{h^{(0)}}[\mathbf{F}[h]] \quad \mbox{for } h \in \bar{B}_1(\HH_{T, \mu}).
$$
Observe that~\eqref{hdot=f} admits a solution if and only if~$\K_{h^{(0)}}$ has a fixed point.

Thanks to Proposition~\ref{linearhprop} and Lemma~\ref{Fcompactlem}, we know that~$\K_{h^{(0)}}$ is a compact map from~$\bar{B}_1(\HH_{T, \mu})$ to~$\HH_{T, \mu_0}$, provided~$\mu > 3 s / (1 + 2 s)$ and~$\mu_0 \in (0, \min \{ 4 s / (1 + 2 s), 1 + \delta \})$. Assuming, in addition to these requirements, that~$\mu_0 > \mu$ (which can be done, as~$\mu$ satisfies~\eqref{mulimit1}, by hypothesis), we also have that
\begin{equation} \label{Kself}
\K_{h^{(0)}}: \bar{B}_1(\HH_{T, \mu}) \to \bar{B}_1(\HH_{T, \mu}),
\end{equation}
provided~$T \ge T_0$, for some generic constant~$T_0 \ge 1$. Indeed, this is immediate to verify, since it holds~$\| \widehat{h} \|_{\HH_{T, \mu}} \le T^{- (\mu_0 - \mu)} \| \widehat{h} \|_{\HH_{T, \mu_0}}$ for every~$\widehat{h} \in \HH_{T, \mu_0}$, and therefore, from estimates~\eqref{Hmapbounded} and~\eqref{Fhest} we get that
\begin{align*}
\| \K_{h^{(0)}}[h] \|_{\HH_{T, \mu}} & \le T^{- (\mu_0 - \mu)} \| \mathbf{H}_{h^{(0)}}[\mathbf{F}[h]] \|_{\HH_{T, \mu_0}} \le C \, T^{- (\mu_0 - \mu)} \left( \| \mathbf{F}[h] \|_{\FF_{T, \frac{4 s}{1 + 2 s}}} + T^{\mu_0 - 1} |h^{(0)}| \right) \\
& \le C \left( T_0^{- (\mu_0 - \mu)} + T^{\mu - 1} |h^{(0)}| \right),
\end{align*}
for some generic constant~$C > 0$. From this,~\eqref{Kself} readily follows by taking~$T_0$ sufficiently large and~$T^{\mu - 1} |h^{(0)}|$ sufficiently small.

In view of~\eqref{Kself} and the compactness of~$\K_{h^{(0)}}$, we may apply the Schauder fixed point theorem, deducing the existence of a fixed point for~$\K_{h^{(0)}}$ within~$\bar{B}_1(\HH_{T, \mu})$. The proof is thus complete.
\end{proof}

\section{Stability results} \label{stabsec}

\noindent
In this section, we address the validity of Theorem~\ref{asymptstabthm}.

Let~$\widetilde{u}: \R \times [T, +\infty) \to \R$ be the solution of
\begin{equation} \label{probforutilde}
\begin{cases}
\partial_t \widetilde{u} + (-\Delta)^s \widetilde{u} + W'(\widetilde{u}) = 0 & \quad \mbox{in } \R \times (T, +\infty) \\
\widetilde{u} = \widetilde{u}_0 & \quad \mbox{on } \R \times \{ T \},
\end{cases}
\end{equation}
with initial datum~$\widetilde{u}_0$ given by
\begin{equation} \label{u0given}
\widetilde{u}_0(x) = \sum_{j = 1}^N w(x - \xi_j^0(T)) + \eta(x) \quad \mbox{for } x \in \R,
\end{equation}
for some small odd function~$\eta$. The main step in the proof of Theorem~\ref{asymptstabthm} consists in proving that any such~$\widetilde{u}$ is actually one of the solutions constructed in Theorem~\ref{mainthm2}. This holds true for any~$s \in (0, 1)$, provided we take~$\widetilde{u}_0$ to be sufficiently regular. We will then cover the general case of a bounded initial datum through an approximation procedure.

The next result shows that~$\widetilde{u}_0$ can be rewritten in the form required by Theorem~\ref{mainthm2}.

\begin{lemma} \label{reparalem}
There exist two generic constants~$T_0 \ge 1$ and~$\delta_0 \in (0, 1)$ such that, if~$\eta \in \A_{\{ T \}}^1$ is an odd function satisfying
\begin{equation} \label{etasmall}
\| \eta \|_{\A_{\{ T \}}} \le \delta_0,
\end{equation}
for some~$T \ge T_0$, then the function~$\widetilde{u}_0$ defined by~\eqref{u0given} can be written as
\begin{equation} \label{u0desired}
\widetilde{u}_0(x) = \sum_{j = 1}^N w(x - \xi_j^0(T) - h^{(0)}_j) + \psi_0 \quad \mbox{for all } x \in \R.
\end{equation}
for some odd function~$\psi_0 \in \A^1_{\{ T \}}$ and some vector~$h^{(0)} \in \R$ satisfying~\eqref{h0odd},
\begin{equation} \label{psi0ortcond}
\int_\R \psi_0(x) w'(x - \xi_i^0(T) - h^{(0)}_i) \, dx = 0 \quad \mbox{for all } i = 1, \ldots, N,
\end{equation}
and~\eqref{h0psi0small}, with~$\varepsilon_0$ given by Theorem~\ref{mainthm2}.
\end{lemma}
\begin{proof}
Let~$h^{(0)} \in \R^N$ be any vector for which~\eqref{h0odd} holds true and write~$\eta$ as
\begin{equation} \label{etadecomp}
\eta(x) = \sum_{j = 1}^N \widehat{\eta}_j w'(x - \xi_j^0(T) - h^{(0)}_j) + \eta^\perp(x) \quad \mbox{for all } x \in \R,
\end{equation}
for some coefficients~$\{ \widehat{\eta}_j \}_{j = 1}^N \subset \R$ satisfying
\begin{equation} \label{etajodd}
\widehat{\eta}_j = - \widehat{\eta}_{N + 1 - j} \quad \mbox{for all } j = 1, \ldots, N
\end{equation}
and some odd function~$\eta^\perp: \R \to \R$ such that
\begin{equation} \label{etaperporth}
\int_{\R} \eta^\perp(x) w'(x - \xi_i^0(T) - h^{(0)}_i) \, dx = 0 \quad \mbox{for all } i = 1, \ldots, N.
\end{equation}
Note that~$\eta$ can be (uniquely) decomposed as in~\eqref{etadecomp}-\eqref{etaperporth}. Indeed, let
$$
\widetilde{Z}_i(x) = \widetilde{Z}_i[h^{(0)}](x) := w'(x - \xi_i^0(T) - h^{(0)}_i)
$$
and consider the symmetric matrix~$\widetilde{A} = \widetilde{A}[h^{(0)}]$ given by
$$
\widetilde{A}_{i j} := \int_\R \widetilde{Z}_i(x) \widetilde{Z}_j(x) \, dx \quad \mbox{for all } i, j = 1, \ldots, N.
$$
Arguing as in the proof of Lemma~\ref{clem}, we get that
\begin{equation} \label{Aepsalmostdiag}
\widetilde{A}_{i j} = \gamma^{-1} \delta_{i j} + (1 - \delta_{i j}) \, O \! \left( T^{-\frac{2 s}{1 + 2 s}} \right) \quad \mbox{for all } i, j = 1, \ldots, N.
\end{equation}
In particular,~$\widetilde{A}$ is invertible for~$T_0$ large enough. Therefore, defining
\begin{alignat*}{3}
\widehat{\eta}_i & := \sum_{j = 1}^N \widetilde{A}^{-1}_{i j} \int_\R \eta(x) \widetilde{Z}_j(x) \, dx && \qquad \mbox{for all } i = 1, \ldots, N, \\
\eta^\perp(x) & := \eta(x) - \sum_{j = 1}^N \widehat{\eta}_j \widetilde{Z}_j(x) && \qquad \mbox{for all } x \in \R,
\end{alignat*}
one easily sees that properties~\eqref{etadecomp}-\eqref{etaperporth} are satisfied. In addition, from the above definitions,~\eqref{Aepsalmostdiag}, and the fact that~$w' \in L^1(\R)$, it is immediate to verify that
\begin{equation} \label{etahatbounds}
|\widehat{\eta}_i| \le C \, T^{- \frac{2 s}{1 + 2 s}} \| \eta \|_{\A_{\{ T \}}} \quad \mbox{for all } i = 1, \ldots, N,
\end{equation}
for some generic~$C > 0$.

We now observe that, recalling~\eqref{u0given} and~\eqref{etadecomp}, the initial datum~$\widetilde{u}_0$ can be written as in~\eqref{u0desired} if and only if~$\psi_0$ is given by
\begin{equation} \label{psi0def}
\psi_0(x) = \widetilde{\NN}[h^{(0)}](x) + \sum_{j = 1}^N \left( \widehat{\eta}_j + h^{(0)}_j \right) \widetilde{Z}_j(x) + \eta^\perp(x),
\end{equation}
with
$$
\widetilde{\NN}[h^{(0)}](x) := \sum_{j = 1}^N \Big\{ w(x - \xi_j^0(T)) - w(x - \xi_j^0(T) - h^{(0)}_j) - h^{(0)}_j w'(x - \xi_j^0(T) - h^{(0)}_j) \Big\}.
$$
In addition, thanks to~\eqref{etaperporth}, the function~$\psi_0$ satisfies the orthogonality conditions~\eqref{psi0ortcond} if and only if
$$
\sum_{j = 1}^N \widetilde{A}_{i j} \left( \widehat{\eta}_j + h^{(0)}_j \right) = - \int_\R \widetilde{\NN}[h^{(0)}](x) \widetilde{Z}_i(x) \, dx \quad \mbox{for all } i = 1, \ldots, N.
$$
Taking advantage of~\eqref{Aepsalmostdiag} with~$i = j$, this is equivalent to the identity
$$
\mathbf{J}(h^{(0)}) = h^{(0)},
$$
where~$\mathbf{J}: \R^N \to \R^N$ is the map defined by
$$
\mathbf{J}_i(h^{(0)}) := - \widehat{\eta}_i - \gamma \left\{ \sum_{j \ne i} \widetilde{A}_{i j}[h^{(0)}] \! \left( \widehat{\eta}_j + h^{(0)}_j \right) + \int_\R \widetilde{\NN}[h^{(0)}](x) \widetilde{Z}_i[h^{(0)}](x) \, dx \right\},
$$
for all~$h^{(0)} \in \R^N$ and~$i = 1, \ldots, N$.

Let~$\rho_0 \in (0, 1)$ to be later chosen and define
$$
\dot{\D} := \left\{ h^{(0)} \in \R^N : h^{(0)} \mbox{ satisfies } \eqref{h0odd} \mbox{ and } |h^{(0)}| \le \rho_0 \, T^{- \frac{2 s}{1 + 2 s}} \right\}.
$$
We claim that~$\mathbf{J}$ has a fixed point within~$\dot{\D}$, provided~$T_0$ is large and~$\delta_0$ is small. As~$\mathbf{J}$ is continuous, in view of the Brouwer fixed point theorem it suffices to prove that
\begin{equation} \label{JfromDtoD}
\mathbf{J}: \dot{\D} \to \dot{\D}.
\end{equation}
On the one hand, the fact that
$$
\mathbf{J}_i(h^{(0)}) = - \mathbf{J}_{N - i + 1}(h^{(0)}) \quad \mbox{for all } i = 1, \ldots, N \mbox{ and } h^{(0)} \in \dot{\D}
$$
is a consequence of the symmetry relations~\eqref{etajodd},~\eqref{h0odd}, and
\begin{equation} \label{symmprop2}
\begin{alignedat}{3}
\widetilde{Z}_i(- x) & = \widetilde{Z}_{N - i + 1}(x) && \qquad \mbox{for all } x \in \R \mbox{ and } i = 1, \ldots, N, \\
\widetilde{\NN}[h^{(0)}](- x) & = - \widetilde{\NN}[h^{(0)}](x) && \qquad \mbox{for all } x \in \R, \\
\widetilde{A}_{i j} & = \widetilde{A}_{N - i + 1 \, N - j + 1} && \qquad \mbox{for all } i, j = 1, \ldots, N.
\end{alignedat}
\end{equation}
On the other hand, we have
\begin{equation} \label{Nh}
\widetilde{\NN}[h^{(0)}](x) = \sum_{j = 1}^N \int_0^{h_j^{(0)}} \left( \int_\tau^{h^{(0)}_j} w''(x - \xi_j^0(T) - \sigma) \, d\sigma \right) d\tau,
\end{equation}
so that, using that~$w'' \in L^\infty(\R)$, we estimate
$$
\| \widetilde{\NN}[h^{(0)}] \|_{L^\infty(\R)} \le C |h^{(0)}|^2.
$$
By this,~\eqref{Aepsalmostdiag},~\eqref{etahatbounds},~\eqref{etasmall}, and the fact that~$w' \in L^1(\R)$, we get, for every~$h^{(0)} \in \dot{\D}$,
\begin{align*}
|\mathbf{J}_i(h^{(0)})| & \le |\widehat{\eta}_i| + C \left\{ T^{- \frac{2 s}{1 + 2 s}} \sum_{j \ne i} \left( |\widehat{\eta}_j| + |h_j^{(0)}| \right) + |h^{(0)}|^2 \int_\R w'(x - \xi_i^0(T) - h_i^{(0)}) \, dx \right\} \\
& \le C \, T^{- \frac{2 s}{1 + 2 s}} \left\{ \delta_0 + (\rho_0 + \rho_0^2) T_0^{- \frac{2 s}{1 + 2 s}} \right\} \le \frac{\rho_0}{N} \, T^{- \frac{2 s}{1 + 2 s}},
\end{align*}
for any~$\rho_0 \in (0, 1)$, provided~$T_0$ is large and~$\delta_0 \le \delta_1 \rho_0$, with~$\delta_1 \in (0, 1)$ small and generic. Accordingly,~\eqref{JfromDtoD} is valid and~$\mathbf{J}$ has a fixed point inside~$\dot{\D}$.

Up to now, we proved that~$\widetilde{u}_0$ can be written as in~\eqref{u0desired} for some~$h^{(0)} \in \dot{\D}$ and~$\psi_0$ satisfying the orthogonal conditions~\eqref{psi0ortcond}. Notice that, as~$h^{(0)} \in \dot{\D}$, we also know that~$h^{(0)}$ satisfies~\eqref{h0odd} and that~$T^{\mu - 1} |h^{(0)}| \le \varepsilon_0 / 2$, provided~$T_0$ is large enough. In addition, recalling identity~\eqref{psi0def}, properties~\eqref{symmprop2}, the fact that the~$\widehat{\eta}_i$'s satisfy~\eqref{etajodd}, and the oddness of~$\eta^\perp$, it is immediate to verify that~$\psi_0$ is odd as well. To finish the proof we are thus left to check that~$\psi_0$ belongs to~$\A_{\{ T \}}^1$ and that it holds
\begin{equation} \label{psi0inA1est}
\| \psi_0 \|_{\A_{\{ T \}}} \le \frac{\varepsilon_0}{2}.
\end{equation}

First of all, notice that, by~\eqref{w'asympt} and taking~$\rho_0$ sufficiently small,
$$
|\widetilde{Z}_j(x)| \le C \left( \chi_{[- 2 \xi_N^0(T), 2 \xi_N^0(T)]}(x) + \frac{\chi_{\R \setminus [- 2 \xi_N^0(T), 2 \xi_N^0(T)]}(x)}{|x - \xi_j^0(T) - h_j^{(0)}|^{1 + 2 s}} \right) \le C \, T^{\frac{2 s}{1 + 2 s}} \Phi(x, T) \quad \mbox{for all } x \in \R.
$$
Similarly, we deduce from identity~\eqref{Nh} and estimate~\eqref{w''decay} of Proposition~\ref{w''decayprop} that
$$
|\widetilde{\NN}[h^{(0)}](x)| \le C \, T_0^{- \frac{2 s}{1 + 2 s}} \Phi(x, T) \quad \mbox{for all } x \in \R.
$$
In light of these two estimates,~\eqref{psi0def},~\eqref{etadecomp},~\eqref{etasmall}, and recalling that~$\delta_0 \le \rho_0$, we conclude that
$$
|\psi_0(x)| \le |\widetilde{\NN}[h^{(0)}](x)| + |\eta(x)| + \sum_{j = 1}^N |h_j^{(0)}| \widetilde{Z}_j(x) \le C \left( T_0^{- \frac{2 s}{1 + 2 s}} + \rho_0 \right) \Phi(x, T) \quad \mbox{for all } x \in \R.
$$
Claim~\eqref{psi0inA1est} readily follows by taking~$\rho_0$ suitably small and~$T_0$ suitably large. The fact that~$\psi_0$ is of class~$C^1$, with bounded derivative, can be easily verified by differentiating~\eqref{psi0def}. Hence,~$\psi_0 \in \Adot_{\{ T \}}^1$ and the proof of the lemma is complete.
\end{proof}

Thanks to Lemma~\ref{reparalem}, Theorem~\ref{asymptstabthm} is now an immediate consequence of Theorem~\ref{mainthm2} and of the unique solvability of problem~\eqref{probforutilde}---at least when~$\widetilde{u}_0$ is regular. The details---both in the regular and irregular case---are as follows.

\begin{proof}[Proof of Theorem~\ref{asymptstabthm}]
To begin with, assume that~$\widetilde{u}_0$ is given by~\eqref{u0given} with~$\eta \in \A_{\{ T \}}^1$. In this case, we may apply Lemma~\ref{reparalem} and rewrite~$\widetilde{u}_0$ as in~\eqref{u0desired}, with~$h^{(0)} \in \R$ and~$\psi_0 \in \A_{\{ T \}}^1$ satisfying~\eqref{h0odd},~\eqref{h0psi0small}, and~\eqref{psi0ortcond}. Hence, by the unique solvability of~\eqref{probforutilde}---see Proposition~\ref{existinLinftyprop}---,~$\widetilde{u}$ is the solution built in Theorem~\ref{mainthm2} in correspondence to the initial data~$h^{(0)}$ and~$\psi_0$. Accordingly, we have that
$$
\widetilde{u}(x, t) = \sum_{i = 1}^N w(x - \xi_i^0(t) - \widetilde{h}_i(t)) + \widetilde{\psi}(x, t) \quad \mbox{for } x \in \R, \, t \ge T,
$$
for some~$\widetilde{h} \in \bar{B}_1(\HH_{T, \mu})$, with~$\widetilde{h}(T) = h^{(0)}$, and some~$\widetilde{\psi} \in \A_T$, with~$\widetilde{\psi}(\cdot, T) = \psi_0$ and~$\| \widetilde{\psi} \|_{\A_T} \le C$, for some generic~$C \ge 1$. Using this and the properties of~$u$, we compute, for every~$x \in \R$ and~$t > T$,
\begin{align*}
\left| \widetilde{u}(x, t) - u(x, t) \right| & \le \sum_{i = 1}^N \left| w(x - \xi_i^0(t) - \widetilde{h}_i(t)) - w(x - \xi_i^0(t) - h_i(t)) \right| + |\widetilde{\psi}(x, t)| + |\psi(x, t)| \\
& \le \| w' \|_{L^\infty(\R)} \sum_{i = 1}^N \left( |h_i(t)| + |\widetilde{h}_i(t)| \right) + \left( \| \widetilde{\psi} \|_{\A_T} + \| \psi \|_{\A_T} \right)\Phi(x, t) \\
& \le C \left\{ \left( \| h \|_{\HH_{T, \mu}} + \| \widetilde{h} \|_{\HH_{T, \mu}} \right) t^{1 - \mu} + t^{- \frac{2 s}{1 + 2 s}} \right\} \le C \, t^{1 - \mu}.
\end{align*}
Since~$s > 1/2$ (and therefore~$\mu > 1$), this gives in particular~\eqref{u-utildeto0}. Hence, Theorem~\ref{asymptstabthm} is proved, under the assumption that~$\eta \in \A_{\{ T \}}^1$.

To get~\eqref{u-utildeto0} in the more general case of~$\eta \in \A_{\{ T \}}$, we argue as follows. First, we consider, for~$\varepsilon > 0$ small, the mollification~$\eta_\varepsilon$ of~$\eta$ via convolution against a smooth, compactly supported, and symmetric kernel. We have that~$\eta_\varepsilon$ is a smooth, odd function. As~$\eta$ is bounded,~$\eta_{\varepsilon}' \in L^\infty(\R)$. Moreover, it is easy to see that~$\eta_\varepsilon$ satisfies~\eqref{etasmall}, up to possibly replacing~$\delta_0$ with~$2 \delta_0$ and taking~$\varepsilon$ sufficiently small. Consequently,~$\eta_\varepsilon \in \A_{\{ T \}}^1$ and we may apply to it the result established earlier, deducing that the solution~$\widetilde{u}^{(\varepsilon)}$ of
\begin{equation} \label{eqforueps}
\begin{cases}
\partial_t \widetilde{u}^{(\varepsilon)} + (-\Delta)^s \widetilde{u}^{(\varepsilon)} + W'(\widetilde{u}^{(\varepsilon)}) = 0 & \quad \mbox{in } \R \times (T, +\infty) \\
\widetilde{u}^{(\varepsilon)}(\cdot, T) = u(\cdot, T) + \eta_\varepsilon & \quad \mbox{in } \R
\end{cases}
\end{equation}
satisfies
\begin{equation} \label{ueps-udecay}
|\widetilde{u}^{(\varepsilon)}(x, t) - u(x, t)| \le C \, t^{1 - \mu} \quad \mbox{for all } x \in \R, \, t > T,
\end{equation}
for a generic constant~$C \ge 1$. Now, as~$\eta_\varepsilon$ is bounded in~$L^\infty(\R)$ uniformly in~$\varepsilon$, equation~\eqref{eqforueps} and Proposition~\ref{regforstrongprop}\ref{intreg} in Section~\ref{solsec} give that, for every large~$k \in \N$, it holds~$\| \widetilde{u}^{(\varepsilon)} \|_{C^\alpha(T + 1/k, T + k)} \le C_k$, for some exponent~$\alpha \in (0, 1)$ and constant~$C_k > 0$ independent of~$\varepsilon$. Hence, by Ascoli-Arzel\`a theorem and a diagonal argument, there exists an infinitesimal sequence~$\{ \varepsilon_j \}_{j \in \N}$ such that~$\{ u^{(\varepsilon_j)} \}$ converges to some bounded function~$\widehat{u}$ a.e.~in~$\R \times (T, +\infty)$. Accordingly,~\eqref{ueps-udecay} passes to the limit and yields
\begin{equation} \label{uhat-udecay}
|\widehat{u}(x, t) - u(x, t)| \le C \, t^{1 - \mu} \quad \mbox{for all } x \in \R, \, t > T.
\end{equation}
Since~$\eta_\varepsilon \rightarrow \eta$ a.e.~in~$\R$, from the mild formulation of~\eqref{eqforueps} and Lebesgue's dominated convergence theorem, we deduce that~$\widehat{u}$ solves problem~\eqref{probforutilde}. By uniqueness (see Proposition~\ref{existinLinftyprop}), we then conclude that~$\widehat{u} = \widetilde{u}$ and therefore that~\eqref{u-utildeto0} follows from~\eqref{uhat-udecay}.
\end{proof}

\appendix

\section{Proofs of Propositions~\ref{uimprovasymptprop} and~\ref{w''decayprop}} \label{layerapp}

\noindent
In this first appendix, we address the validity of the results stated in Section~\ref{layersec}, about the asymptotic behavior of the layer solution~$w$ and of its derivatives.

The general strategy of the proof of Proposition~\ref{uimprovasymptprop} is based on techniques developed in~\cite[Sections~5 and~6]{GM12} and~\cite[Sections~4-7]{DPV15}. Our improvements of those arguments mainly consist in the establishment and application of a maximum principle for odd solutions of linear equations, used in conjunction with a particular odd barrier---the function~$\omega_3$ introduced below in~\eqref{omegasdef}. Our technique heavily relies on the parity of~$W$ (hypothesis~\eqref{Weven}) and, thus, cannot be applied in the more general frameworks of~\cite{GM12,DPV15,DFV14}. 

Estimate~\eqref{w'''decay} of Proposition~\ref{w''decayprop} also follows from an odd comparison principle, applied with the function~$\omega_2$ instead of~$\omega_3$---see again~\eqref{omegasdef} for its definition.

\subsection{Three auxiliary functions}

In this subsection, we obtain some information on the decay at infinity of the fractional Laplacian applied to the functions
\begin{equation} \label{omegasdef}
\omega_1(x) := \frac{1}{(1 + x^2)^{\frac{1 + 2 s}{2}}}, \quad \omega_2(x) := \frac{x}{(1 + x^2)^{\frac{3 + 2 s}{2}}}, \quad \mbox{and} \quad \omega_3(x) := \frac{x}{(1 + x^2)^{\frac{1 + 4 s}{2}}}.
\end{equation}
First, we compute the asymptotic expansion of~$\omega_1$ at infinity.

\begin{lemma} \label{philem}
It holds
\begin{equation} \label{expforomega1}
\frac{(-\Delta)^s \omega_1(x)}{\omega_1(x)} = - \int_\R \omega_1(y) \, dy + \lambda_{1, s} \, |x|^{- 2 s} + o \! \left( |x|^{- 2 s} \right) \quad \mbox{as } x \rightarrow \pm \infty,
\end{equation}
for some constant~$\lambda_{1, s} \in \R$. In particular,
\begin{equation} \label{Dsomega1leomega1}
\left| (-\Delta)^s \omega_1(x) \right| \le C_1 \, \omega_1(x) \quad \mbox{for all } x \in \R,
\end{equation}
for some constant~$C_1 > 0$.
\end{lemma}
\begin{proof}
We focus on the expansion~\eqref{expforomega1}, as estimate~\eqref{Dsomega1leomega1} is an immediate consequence of it, taking also into account the continuity of~$\omega_1$ and~$(-\Delta)^s \omega_1$ as well as the positivity of~$\omega_1$.

First of all, we notice that, by symmetry, it is enough to compute the expansion as~$x \rightarrow +\infty$. Then, we write down the first and second derivatives of~$\omega_1$:
\begin{equation} \label{phi'phi''}
\omega_1'(x) = - (1 + 2 s) \frac{x}{(1 + x^2)^{\frac{3 + 2 s}{2}}} \quad \mbox{and} \quad \omega_1''(x) = (1 + 2 s) \frac{2(1 + s) x^2 - 1}{(1 + x^2)^{\frac{5 + 2 s}{2}}}.
\end{equation}

We have
$$
\frac{(- \Delta)^s \omega_1(x)}{\omega_1(x)} = I_1(x) + I_2(x) + I_3(x) + I_4(x),
$$
with
\begin{align*}
I_1(x) & := \frac{1}{\omega_1(x)} \int_{\frac{x}{2}}^{\frac{3x}{2}} \frac{\omega_1(x) - \omega_1(y) - \omega_1'(x) (x - y)}{|x - y|^{1 + 2 s}} \, dy, \\
I_2(x) & := \frac{1}{\omega_1(x)} \int_{\left( - \infty, - \frac{x}{2} \right) \cup \left( \frac{3 x}{2}, +\infty \right)} \frac{\omega_1(x) - \omega_1(y)}{|x - y|^{1 + 2 s}} \, dy, \\
I_3(x) & := \int_{- \frac{x}{2}}^{\frac{x}{2}} \frac{dy}{|x - y|^{1 + 2 s}}, \quad \mbox{and} \quad I_4(x) := - \frac{1}{\omega_1(x)} \int_{- \frac{x}{2}}^{\frac{x}{2}} \frac{\omega_1(y)}{|x - y|^{1 + 2 s}} \, dy.
\end{align*}

We now proceed to evaluate the terms~$I_j(x)$ one by one. We will frequently use the change of variables~$t := y/x$. We begin with~$I_1(x)$. Applying this change of coordinates, we have
$$
I_1(x) = \frac{x^{- 2 s}}{\omega_1(x)} \int_{\frac{1}{2}}^{\frac{3}{2}} \frac{\omega_1(x) - \omega_1(t x) - x \omega_1'(x) (1 - t)}{|1 - t|^{1 + 2 s}} \, dt.
$$
A Taylor expansion at~$t = 1$ gives that
$$
\omega_1(x) - \omega_1(t x) - x \omega_1'(x) (1 - t) = - \frac{x^2 \omega_1''(z(t, x))}{2} (t - 1)^2 \quad \mbox{for all } t \in \left[ \frac{1}{2}, \frac{3}{2} \right],
$$
for some~$z(t, x) \in [x/2, 3x/2]$. Hence, taking advantage of~\eqref{phi'phi''},
$$
\frac{\left| \omega_1(x) - \omega_1(t x) - x \omega_1'(x) (1 - t) \right|}{\omega_1(x)} \le C (t - 1)^2 \quad \mbox{for all } t \in \left[ \frac{1}{2}, \frac{3}{2} \right],
$$
for every large~$x$ and for some constant~$C > 0$ depending only on~$s$.
%Hence,
%$$
%\frac{x^{2 s}}{x^{3 + 2 s} \omega_1(x)} \int_{\frac{1}{2}}^{\frac{3}{2}} |t - 1|^{1 - 2 s} \, dt \le \frac{C}{x^{2(1 - s)}} \int_0^{\frac{1}{2}} \tau^{1 - 2 s} \, d\tau \le C,
%$$
%with~$C$ now possibly depending also on~$s$, but still independent of~$x$.
Consequently, we may apply Lebesgue's dominated convergence theorem and deduce that
\begin{equation} \label{I1limit}
\begin{aligned}
\lim_{x \rightarrow +\infty} x^{2 s} I_1(x) & = \int_{\frac{1}{2}}^{\frac{3}{2}} \left\{ \lim_{x \rightarrow +\infty} \frac{\omega_1(x) - \omega_1(t x) - x \omega_1'(x) (1 - t) }{\omega_1(x)} \right\} \frac{dt}{|1 - t|^{1 + 2 s}} \\
& = \int_{\frac{1}{2}}^{\frac{3}{2}} \left\{ \lim_{x \rightarrow +\infty} \left( 1 - \left( \frac{1 + x^2}{1 + t^2 x^2} \right)^{\frac{1 + 2 s}{2}} + (1 + 2 s) \frac{(1 - t) x^2}{1 + x^2} \right) \right\} \frac{dt}{|1 - t|^{1 + 2 s}} \\
& = \int_{\frac{1}{2}}^{\frac{3}{2}} \frac{1 - |t|^{- 1 - 2 s} + (1 + 2 s) (1 - t)}{|1 - t|^{1 + 2 s}} \, dt = \PV \int_{\frac{1}{2}}^{\frac{3}{2}} \frac{1 - |t|^{- 1 - 2 s}}{|1 - t|^{1 + 2 s}} \, dt,
\end{aligned}
\end{equation}
where~$\PV$ indicates that the integral has to be intended in the Cauchy principal value sense. A similar (but simpler) argument yields that
\begin{equation} \label{I2limit}
\lim_{x \rightarrow +\infty} x^{2 s} I_2(x) = \int_{\left( - \infty, - \frac{1}{2} \right) \cup \left( \frac{3}{2}, +\infty \right)} \frac{1 - |t|^{- 1 - 2 s}}{|1 - t|^{1 + 2 s}} \, dt,
\end{equation}
whereas
\begin{equation} \label{I3limit}
x^{2 s} I_3(x) \equiv \int_{- \frac{1}{2}}^{\frac{1}{2}} \frac{dt}{|1 - t|^{1 + 2 s}}.
\end{equation}

The expansion of~$I_4(x)$ is more involved. Set~$A := \int_\R \omega_1(y) \, dy$. Changing variables as before, we see that
$$
I_4(x) = - \frac{x^{- 2 s}}{\omega_1(x)} \int_{- \frac{1}{2}}^{\frac{1}{2}} \frac{\omega_1(t x)}{|1 - t|^{1 + 2 s}} \, dt = - x^{- 2 s} \int_{- \frac{1}{2}}^{\frac{1}{2}} \left( \frac{1 + x^2}{1 + t^2 x^2} \right)^{\frac{1 + 2 s}{2}} \frac{dt}{(1 - t)^{1 + 2 s}}.
$$
Notice now that the constant~$A$ can be written as
$$
A = \int_\R \omega_1(y) \, dy = \frac{x}{(1 + x^2)^{\frac{1 + 2 s}{2}}} \left\{ \int_{- \frac{1}{2}}^{\frac{1}{2}} \left( \frac{1 + x^2}{1 + t^2 x^2} \right)^{\frac{1 + 2 s}{2}} \, dt + \int_{\R \setminus \left( - \frac{1}{2}, \frac{1}{2} \right)} \left( \frac{1 + x^2}{1 + t^2 x^2} \right)^{\frac{1 + 2 s}{2}} \, dt \right\},
$$
for every~$x$. Using this, we find
\begin{align*}
\lim_{x \rightarrow +\infty} x^{2 s} \left\{ I_4(x) + A \right\} & = \lim_{x \rightarrow +\infty} \int_{- \frac{1}{2}}^{\frac{1}{2}} \left( \frac{1 + x^2}{1 + t^2 x^2} \right)^{\frac{1 + 2 s}{2}} \left\{ \left( \frac{x^2}{1 + x^2} \right)^{\frac{1 + 2 s}{2}} - (1 - t)^{- 1 - 2 s} \right\} dt \\
& \quad + \int_{\R \setminus \left( - \frac{1}{2}, \frac{1}{2} \right)} \frac{dt}{|t|^{1 + 2 s}}.
\end{align*}
Note that, by symmetry,
$$
(1 + 2 s) \int_{- \frac{1}{2}}^{\frac{1}{2}} \left( \frac{1 + x^2}{1 + t^2 x^2} \right)^{\frac{1 + 2 s}{2}} t \, dt = 0 \quad \mbox{for every } x.
$$
Hence, we have that
\begin{equation} \label{I4prelim}
\lim_{x \rightarrow +\infty} x^{2 s} \left\{ I_4(x) + A \right\} = \lim_{x \rightarrow +\infty} \int_{- \frac{1}{2}}^{\frac{1}{2}} \left( \frac{1 + x^2}{1 + t^2 x^2} \right)^{\frac{1 + 2 s}{2}} \left\{ \mu(t) + \nu(x) \right\}  dt + \int_{\R \setminus \left( - \frac{1}{2}, \frac{1}{2} \right)} \frac{dt}{|t|^{1 + 2 s}},
\end{equation}
with
$$
\mu(t) := 1 - (1 - t)^{- 1 - 2 s} + (1 + 2 s) t \quad \mbox{and} \quad \nu(x) := \left( \frac{x^2}{1 + x^2} \right)^{\frac{1 + 2 s}{2}} - 1.
$$
By Taylor expanding~$\mu$ at~$t = 0$, it is easy to see that
$$
|\mu(t)| \le C \, t^2 \quad \mbox{for all } t \in \left( -\frac{1}{2}, \frac{1}{2} \right).
$$
Therefore,
$$
\left( \frac{1 + x^2}{1 + t^2 x^2} \right)^{\frac{1 + 2 s}{2}} |\mu(t)| \le C \, |t|^{1 - 2 s} \in L^1 \! \left( - \frac{1}{2}, \frac{1}{2} \right),
$$
for every~$x \ge 1$, and the dominated convergence theorem gives that
\begin{equation} \label{mulimit}
\lim_{x \rightarrow +\infty} \int_{- \frac{1}{2}}^{\frac{1}{2}} \left( \frac{1 + x^2}{1 + t^2 x^2} \right)^{\frac{1 + 2 s}{2}} \mu(t) \, dt = \int_{- \frac{1}{2}}^{\frac{1}{2}} \frac{\mu(t)}{|t|^{1 + 2 s}} \, dt = \PV \int_{- \frac{1}{2}}^{\frac{1}{2}} \frac{1 - (1 - t)^{- 1 - 2 s}}{|t|^{1 + 2 s}} \, dt.
\end{equation}
On the other hand, since~$|\nu(x)| \le C x^{- 2}$ for every large~$x$, we obtain
$$
\left| \nu(x) \int_{- \frac{1}{2}}^{\frac{1}{2}} \left( \frac{1 + x^2}{1 + t^2 x^2} \right)^{\frac{1 + 2 s}{2}} dt \right| \le \frac{C}{x^{1 - 2 s}} \int_{- \frac{1}{2}}^{\frac{1}{2}} \frac{dt}{(1 + t^2 x^2)^{\frac{1 + 2 s}{2}}} \le \frac{C}{x^{2(1 - s)}} \int_{\R} \frac{dy}{(1 + y^2)^{\frac{1 + 2 s}{2}}} \longrightarrow 0,
$$
as~$x \rightarrow +\infty$. By combining this with~\eqref{mulimit} and~\eqref{I4prelim}, we conclude that
\begin{equation} \label{I4limit}
\lim_{x \rightarrow +\infty} x^{2 s} \left\{ I_4(x) + A \right\} = \PV \int_{- \frac{1}{2}}^{\frac{1}{2}} \frac{1 - (1 - t)^{- 1 - 2 s}}{|t|^{1 + 2 s}} \, dt + \int_{\R \setminus \left( - \frac{1}{2}, \frac{1}{2} \right)} \frac{dt}{|t|^{1 + 2 s}}.
\end{equation}
Identity~\eqref{expforomega1} follows from this,~\eqref{I1limit},~\eqref{I2limit}, and~\eqref{I3limit}.
\end{proof}

Similarly, we have the following expansion for~$\omega_2$.

\begin{lemma} \label{psilem}
It holds
\begin{equation} \label{Dspsiasympt}
\frac{(-\Delta)^s \omega_2(x)}{\omega_2(x)} = - \int_\R \omega_1(y) \, dy + \lambda_{2, s} |x|^{- 2 s} + o \! \left( |x|^{- 2 s} \right) \quad \mbox{as } x \rightarrow \pm \infty,
\end{equation}
for some constant~$\lambda_{2, s} \in \R$. In particular,
\begin{equation} \label{Dspsilepsi}
\left| (-\Delta)^s \omega_2(x) \right| \le C_2 |\omega_2(x)| \quad \mbox{for all } x \in \R,
\end{equation}
for some constant~$C_2 > 0$.
\end{lemma}
\begin{proof}
As in Lemma~\ref{philem}, it is enough to verify~\eqref{Dspsiasympt} for~$x \rightarrow +\infty$.
%we begin by computing the first and second derivatives of~$\omega_2$:
%\begin{equation} \label{phi'phi''}
%\omega_2'(x) = \frac{1 - 2 (1 + s) x^2}{(1 + x^2)^{\frac{5 + 2 s}{2}}} \quad \mbox{and} \quad \omega_2''(x) = (3 + 2 s) \frac{\left\{ 2 (1 + s) x^2 - 3 \right\} x}{(1 + x^2)^{\frac{7 + 2 s}{2}}}.
%\end{equation}
We decompose~$(-\Delta)^s \omega_2(x) / \omega_2(x)$ as
$$
\frac{(- \Delta)^s \omega_2(x)}{\omega_2(x)} = J_1(x) + J_2(x) + J_3(x) + J_4(x),
$$
with
\begin{align*}
J_1(x) & := \frac{1}{\omega_2(x)} \int_{\frac{x}{2}}^{\frac{3x}{2}} \frac{\omega_2(x) - \omega_2(y) - \omega_2'(x) (x - y)}{|x - y|^{1 + 2 s}} \, dy, \\
J_2(x) & := \frac{1}{\omega_2(x)} \int_{\left( - \infty, - \frac{x}{2} \right) \cup \left( \frac{3 x}{2}, +\infty \right)} \frac{\omega_2(x) - \omega_2(y)}{|x - y|^{1 + 2 s}} \, dy, \\
J_3(x) & := \int_{- \frac{x}{2}}^{\frac{x}{2}} \frac{dy}{|x - y|^{1 + 2 s}}, \quad \mbox{and} \quad J_4(x) := - \frac{1}{\omega_2(x)} \int_{- \frac{x}{2}}^{\frac{x}{2}} \frac{\omega_2(y)}{|x - y|^{1 + 2 s}} \, dy.
\end{align*}
Arguing as in Lemma~\ref{philem}, it is easy to see that
\begin{equation} \label{J123limit}
\lim_{x \rightarrow +\infty} x^{2 s} \left\{ J_1(x) + J_2(x) + J_3(x) \right\} = \PV \int_{\R \setminus \left( - \frac{1}{2}, \frac{1}{2} \right)} \frac{1 - |t|^{- 3 - 2 s} t}{|1 - t|^{1 + 2 s}} \, dt + \int_{- \frac{1}{2}}^{\frac{1}{2}} \frac{dt}{|1 - t|^{1 + 2 s}}.
\end{equation}

To deal with~$J_4(x)$, we first observe that~$(1 + 2 s) \omega_2 = - \omega_1'$ and integrate by parts, obtaining
$$
J_4(x) = \frac{2^{1 + 2 s} \left( 1 - 3^{- 1 - 2 s} \right)}{1 + 2 s} \frac{\omega_1(x/2)}{x^{1 + 2 s} \omega_2(x)} - \frac{1}{\omega_2(x)} \int_{- \frac{x}{2}}^{\frac{x}{2}} \frac{\omega_1(y)}{(x - y)^{2 + 2 s}} \, dy.
$$
Writing~$A := \int_\R \omega_1(y) \, dy$ and changing variables appropriately, we then have
\begin{align*}
x^{2 s} \left\{ J_4(x) + A \right\} & = \frac{2^{1 + 2 s} \left( 1 - 3^{- 1 - 2 s} \right)}{1 + 2 s} \frac{1 + x^2}{x^2} \left( \frac{1 + x^2}{1 + x^2/4} \right)^{\frac{1 + 2 s}{2}} + \int_{\R \setminus \left( - \frac{1}{2}, \frac{1}{2} \right)} \left( \frac{x^2}{1 + t^2 x^2} \right)^{\frac{1 + 2 s}{2}} \, dt \\
& \quad + \frac{1 + x^2}{x^2} \int_{- \frac{1}{2}}^{\frac{1}{2}} \left( \frac{1 + x^2}{1 + t^2 x^2} \right)^{\frac{1 + 2 s}{2}} \left\{ \left( \frac{x^2}{1 + x^2} \right)^{\frac{3 + 2 s}{2}} - (1 - t)^{- 2 - 2 s} \right\} dt,
\end{align*}
and therefore
\begin{equation} \label{J4prelim}
\begin{aligned}
\lim_{x \rightarrow +\infty} x^{2 s} \left\{ J_4(x) + A \right\} & = \frac{4^{1 + 2 s} \left( 1 - 3^{- 1 - 2 s} \right)}{1 + 2 s} + \int_{\R \setminus \left( - \frac{1}{2}, \frac{1}{2} \right)} \frac{dt}{|t|^{1 + 2 s}} \\
& \quad + \lim_{x \rightarrow +\infty} \int_{- \frac{1}{2}}^{\frac{1}{2}} \left( \frac{1 + x^2}{1 + t^2 x^2} \right)^{\frac{1 + 2 s}{2}} \left\{ \left( \frac{x^2}{1 + x^2} \right)^{\frac{3 + 2 s}{2}} - (1 - t)^{- 2 - 2 s} \right\} dt.
\end{aligned}
\end{equation}
As
$$
2 (1 + s) \int_{- \frac{1}{2}}^{\frac{1}{2}} \left( \frac{1 + x^2}{1 + t^2 x^2} \right)^{\frac{1 + 2 s}{2}} t \, dt = 0 \quad \mbox{for every } x,
$$
we may add this term to the argument of the limit on the second row of~\eqref{J4prelim} and, arguing as we did to get~\eqref{I4limit}, we deduce that
\begin{align*}
\lim_{x \rightarrow +\infty} x^{2 s} \left\{ J_4(x) + A \right\} & = \frac{4^{1 + 2 s} \left( 1 - 3^{- 1 - 2 s} \right)}{1 + 2 s} + \int_{\R \setminus \left( - \frac{1}{2}, \frac{1}{2} \right)} \frac{dt}{|t|^{1 + 2 s}} + \PV \int_{- \frac{1}{2}}^{\frac{1}{2}} \frac{1 - (1 - t)^{- 2 - 2 s}}{|t|^{1 + 2 s}} \, dt.
\end{align*}
This and~\eqref{J123limit} lead to the asymptotic expansion~\eqref{Dspsiasympt}.

In order to obtain~\eqref{Dspsilepsi}, we first observe that, by symmetry,~$\omega_2(0) = (-\Delta)^s \omega_2(0) = 0$. Also,
$$
\omega_2'(x) = \frac{1 - 2(1 + s) x^2}{(1 + x^2)^{\frac{5 + 2 s}{2}}}.
$$
Hence,~$\omega_2'(0) = 1$ and applying L'H\^opital's rule we have
$$
\lim_{x \rightarrow 0} \frac{(-\Delta)^s \omega_2(x)}{\omega_2(x)} = \lim_{x \rightarrow 0} \frac{(-\Delta)^s \omega_2'(x)}{\omega_2'(x)} = (-\Delta)^s \omega_2'(0) \in \R.
$$
From this,~\eqref{Dspsiasympt}, the continuity of~$\omega_2$ and~$(-\Delta)^s \omega_2$, as well as the positivity of~$|\omega_2|$ away from~$0$, we conclude the validity of~\eqref{Dspsilepsi}.
\end{proof}

Finally, we have the following estimate on the decay of~$(-\Delta)^s \omega_3$. This result is the most important of the subsection for the applications that will follow.

\begin{lemma} \label{omegabarrierlem}
It holds
\begin{equation} \label{slaplundercontrol}
\left| (-\Delta)^s \omega_3(x) \right| \le C_3 |\omega_3(x)| \quad \mbox{for all } x \in \R,
\end{equation}
for some constant~$C_3 > 0$. 
\end{lemma}
\begin{proof}
As~$w$ is odd, so is~$(-\Delta)^s \omega_3$. Consequently, we may reduce ourselves to verify~\eqref{slaplundercontrol} for~$x > 0$.

First, we compute the first and second derivatives of~$\omega_3$. We have
\begin{equation} \label{w'w''}
\omega_3'(x) =  \frac{1 - 4 s x^2}{(1 + x^2)^{\frac{4 s + 3}{2}}} \quad \mbox{and} \quad
\omega_3''(x) = \frac{(4 s + 1) (4 s x^2 - 3) x }{(1 + x^2)^{\frac{4 s + 5}{2}}}.
\end{equation}
Since~$\omega_3'(0) = 1$, using L'H\^opital's rule we infer that
$$
\lim_{x \rightarrow 0} \frac{(-\Delta)^s \omega_3(x)}{\omega_3(x)} = \lim_{x \rightarrow 0} \frac{(-\Delta)^s \omega_3'(x)}{\omega_3'(x)} = (- \Delta)^s \omega_3'(0) \in \R.
$$
In view of this, of the continuity of~$\omega_3$ and~$(-\Delta)^s \omega_3$, and of the positivity of~$\omega_3$ in~$(0, +\infty)$, it is enough to check~\eqref{slaplundercontrol} only for large values of~$x$.

We let~$x > 10$ and write
$$
(-\Delta)^s \omega_3(x) = \ell_1(x) + \ell_2(x) + \ell_3(x) + \ell_4(x),
$$
with
\begin{align*}
\ell_1(x) & := \int_{\frac{x}{2}}^{\frac{3x}{2}} \frac{\omega_3(x) - \omega_3(y) + \omega_3'(x) (y - x)}{|x - y|^{1 + 2 s}} \, dy, \quad \ell_2(x) := \int_{\left( -\infty, - \frac{x}{2} \right) \cup \left( \frac{3 x}{2}, +\infty \right)} \frac{\omega_3(x) - \omega_3(y)}{|x - y|^{1 + 2 s}} \, dy, \\
\ell_3(x) & := \omega_3(x) \int_{- \frac{x}{2}}^{\frac{x}{2}} \frac{dy}{|x - y|^{1 + 2 s}}, \quad \mbox{and} \quad \ell_4(x) := -\int_{- \frac{x}{2}}^{\frac{x}{2}} \frac{\omega_3(y)}{|x - y|^{1 + 2 s}} \, dy.
\end{align*}

We now proceed to estimate the terms~$\ell_i(x)$ one by one. We begin with~$\ell_1(x)$. By Taylor's theorem,
$$
\omega_3(x) - \omega_3(y) + \omega_3'(x) (y - x) = - \frac{\omega_3''(z(x, y))}{2} (y - x)^2,
$$
with~$z(x, y) \in [x/2, 3 x/2]$. Recalling~\eqref{w'w''}, we then have
$$
|\omega_3(x) - \omega_3(y) + \omega_3'(x) (y - x)| \le \frac{C}{x^{4 s + 2}} (y - x)^2,
$$
for some constant~$C > 0$ depending only on~$s$. Hence,
\begin{equation} \label{I1est}
|\ell_1(x)| \le \frac{C}{x^{4 s + 2}}\int_{\frac{x}{2}}^{\frac{3x}{2}} |y - x|^{1 - 2 s} \, dy \le \frac{C}{x^{6 s}}. 
\end{equation}
To bound~$\ell_2(x)$, we simply recall definition~\eqref{omegasdef} and change variables appropriately:
\begin{equation} \label{I2est}
|\ell_2(x)| \le \int_{\left( -\infty, - \frac{x}{2} \right) \cup \left( \frac{3 x}{2}, +\infty \right)} \frac{|\omega_3(x)| + |\omega_3(y)|}{|x - y|^{1 + 2 s}} \, dy \le \frac{C}{x^{4 s}} \int_{\frac{x}{2}}^{+\infty} \frac{dz}{z^{1 + 2 s}} \le \frac{C}{x^{6 s}}.
\end{equation}
The computation for~$\ell_3(x)$ is immediate and also gives
\begin{equation} \label{I3est}
|\ell_3(x)| \le \frac{C}{x^{6 s}}.
\end{equation}
We are thus left to deal with~$\ell_4(x)$. Notice that
$$
\omega_3(x) = \Omega_3'(x), \quad \mbox{with} \quad \Omega_3(x) = \begin{dcases} \frac{(1 + x^2)^{\frac{1 - 4 s}{2}}}{1 - 4 s} & \mbox{ if } s \ne 1/4, \\
\frac{\log (1 + x^2)}{2} & \mbox{ if } s = 1/4.
\end{dcases}
$$
Integrating by parts, we then have
$$
\ell_4(x) = - 2^{1 + 2 s} \left( 1 - 3^{- 1 -2 s} \right) x^{- 1 - 2 s} \, \Omega_3 \left( \frac{x}{2} \right) + (1 + 2 s) \int_{- \frac{x}{2}}^{\frac{x}{2}} \frac{\Omega_3(y)}{(x - y)^{2 + 2 s}} \, dy.
$$
Using that~$|\Omega_3(y)| \le C$ for every~$|y| \le 1$ and that~$|\Omega_3(y)| \le C_\varepsilon |y|^{1 - 4 s + \varepsilon}$ for every~$|y| \ge 1$ and~$\varepsilon > 0$ (obviously, one can even take~$\varepsilon = 0$ when~$s \ne 1/4$), we estimate
\begin{align*}
|\ell_4(x)| & \le C_\varepsilon \left( x^{- 6 s + \varepsilon} + \int_{0}^1 \frac{dy}{x^{2 + 2 s}} + \int_1^{\frac{x}{2}} \frac{y^{1 - 4 s + \varepsilon}}{(x - y)^{2 + 2 s}} \, dy \right) \\
& \le C_\varepsilon \left( x^{- 6 s + \varepsilon} + x^{- 2 - 2 s} + x^{- 4 s} \int_1^{\frac{x}{2}} \frac{dy}{y^{1 + 2 s - \varepsilon}} \right) \le \frac{C}{x^{4 s}},
\end{align*}
where the last inequality follows by taking, say,~$\varepsilon = s$. By combining this with~\eqref{I1est},~\eqref{I2est}, and~\eqref{I3est}, we infer that~$|(-\Delta)^s \omega_3(x)| \le C x^{- 4 s}$. Since~$\omega_3(x) \ge C^{-1} x^{- 4 s}$ for~$x > 10$, we conclude that~\eqref{slaplundercontrol} holds true.
\end{proof}

\subsection{Maximum principles}

We include here a series of maximum principles that will be used in the next subsection to obtain decay estimates for linear equations driven by the fractional Laplacian. As we will be mostly interested in~\emph{odd} solutions, the following remark is rather relevant---see also~\cite{JW16,FS19} for similar observations.

\begin{remark} \label{fraclapforoddrmk}
Let~$v: \R \to \R$ be an odd function and write
$$
\R_+ := (0, +\infty).
$$
At every point~$x \in \R_+$ for which~$(-\Delta)^s v(x)$ is well-defined we have
\begin{equation} \label{fraclapLKrel}
\begin{aligned}
(-\Delta)^s v(x) & = \PV \int_0^{+\infty} \frac{v(x) - v(y)}{|x - y|^{1 + 2 s}} \, dy + \int_{-\infty}^0 \frac{v(x) - v(z)}{|x - z|^{1 + 2 s}} \, dz \\
& = \PV \int_0^{+\infty} \frac{v(x) - v(y)}{|x - y|^{1 + 2 s}} \, dy + \int_0^{+\infty} \frac{v(x) + v(y)}{(x + y)^{1 + 2 s}} \, dy \\
%& = \L_{K_s} v(x) + 2 v(x) \int_0^{+\infty} \frac{dy}{(x + y)^{1 + 2 s}} \\
& = \L_{K_s} v(x) + \frac{x^{- 2 s}}{s} \, v(x),
\end{aligned}
\end{equation}
where the second identity follows from the change of variables~$y := - z$ and where we set
$$
\L_{K_s} v(x) := \PV \int_{\R_+} \left\{ v(x) - v(y) \right\} K_s(x, y) \, dy,
$$
with
\begin{equation} \label{Ksdef}
K_s(x, y) := \frac{1}{|x - y|^{1 + 2 s}} - \frac{1}{(x + y)^{1 + 2 s}} \quad \mbox{for all } x, y \in \R_+.
\end{equation}
Notice that
\begin{equation} \label{Kspositive}
K_s(x, y) > 0 \quad \mbox{for every } x, y \in \R_+.
\end{equation}
\end{remark}

Our first maximum principle holds for subsolutions of linear integro-differential equations in~$\R_+$ that vanish at zero and at infinity. Here, we can allow for very general operators of the form
\begin{equation} \label{LKdef}
\L_K \phi(x) := \PV \int_{\R_+} \left\{ \phi(x) - \phi(y) \right\} K(x, y) \, dy,
\end{equation}
for a non-negative kernel~$K: \R_+ \times \R_+ \to \R$. In view of Remark~\ref{fraclapforoddrmk}, it can be applied in particular to odd solutions of equations driven by the fractional Laplacian.

\begin{proposition} \label{maxprincprop}
Let~$K: \R_+ \times \R_+ \to \R$ and~$d: \R_+ \to \R$ be two measurable functions, with~$K$ non-negative and~$d$ satisfying
\begin{equation} \label{d>delta}
d(x) \ge \delta \quad \mbox{for all } x \in \R_+,
\end{equation}
for some constant~$\delta > 0$. Let~$\phi \in C^0(\R_+)$ be such that
\begin{equation} \label{philimits}
\lim_{x \rightarrow 0^+} \phi(x) = \lim_{x \rightarrow + \infty} \phi(x) = 0.
\end{equation}
Assume that, for every~$x \in \R_+$, the quantity~$\L_K \phi(x)$ as in~\eqref{LKdef} is well-defined and satisfies
\begin{equation} \label{LKv+dvle0}
\L_K \phi + d \phi \le 0 \quad \mbox{in } \R_+.
\end{equation}
Then,~$\phi \le 0$ in~$\R_+$.
\end{proposition}
\begin{proof}
Suppose, by contradiction that~$\phi(x_\star) > 0$ at some~$x_\star \in \R_+$. In view of assumption~\eqref{philimits} and the continuity of~$\phi$ in~$\R_+$, there exists~$x_M \in \R_+$ at which~$\phi(x_M) = \max_{\R_+} \phi > 0$. Thanks to the non-negativity of~$K$, hypothesis~\eqref{d>delta} on~$d$, and equation~\eqref{LKv+dvle0}, at this point it holds
$$
0 \ge \L_K \phi(x_M) + d(x_M) \phi(x_M) \ge \delta \phi(x_M),
$$
which is clearly a contradiction. Consequently,~$\phi \le 0$ in the whole~$\R_+$.
\end{proof}

Next is a modification of the previous result, which holds for bounded subsolutions in~$(L, +\infty)$, with~$L > 0$, regardless of their behavior at infinity. Notice that here we take~$K$ to be the kernel~$K_s$ defined by~\eqref{Ksdef}.

\begin{proposition} \label{weakMPprop2}
Let~$\phi \in L^\infty(\R_+) \cap C_\loc^\alpha(\R_+)$, for some~$\alpha > 2 s$, be such that
\begin{equation} \label{LKphi+dphi}
\L_{K_s} \phi + d \phi \le 0 \quad \mbox{in } (L, +\infty),
\end{equation}
for some measurable function~$d: (L, +\infty) \to \R$ satisfying
\begin{equation} \label{d>deltax>L}
d(x) \ge \delta \quad \mbox{for all } x > L,
\end{equation}
for some constants~$\delta, L > 0$. Assume that~$\phi \le 0$ in~$(0, L]$. Then,~$\phi \le 0$ in~$\R_+$.
\end{proposition}

This result is an immediate consequence of the following lemma, which can be seen as a one-dimensional, nonlocal version of the weak Omori-Yau maximum principle studied, e.g., in~\cite{PRS03}.

\begin{lemma} \label{weakOYMPlem}
Let~$\phi \in L^\infty(\R_+) \cap C_\loc^{\alpha}(\R_+)$, for some~$\alpha > 2 s$, and assume that
\begin{equation} \label{u*notattin0L}
\phi^* := \sup_{\R_+} \phi > \sup_{(0, L)} \phi,
\end{equation}
for some~$L > 0$. Then, there exists a sequence of points~$\{ x_k \} \subset (L, +\infty)$ such that
\begin{equation} \label{weakMPforLKs}
\phi^* - \phi(x_k) < \frac{1}{k} \quad \mbox{and} \quad \L_{K_s} \phi(x_k) > - \frac{1}{k},
\end{equation}
for every~$k \in \N$.
\end{lemma}
\begin{proof}
Let~$\sigma \in (0, 2 s)$ be fixed and
$$
\eta(x) := |x|^{\sigma - 1} x \quad \mbox{for } x \in \R.
$$
We claim that there exists a constant~$C > 0$, depending only on~$s$,~$\sigma$, and~$L$, for which
\begin{equation} \label{Dsetabounded}
|(-\Delta)^s \eta(x)| \le C \quad \mbox{for all } x \ge L.
\end{equation}
This follows from a simple computation. Indeed, since
$$
\eta''(x) = \sigma (\sigma - 1) x^{\sigma - 2} \quad \mbox{for all } x > 0
$$
and
$$
|\eta(y)| = |y|^\sigma \le C_\sigma \left( x^\sigma + |x - y|^\sigma \right) \quad \mbox{for all } x > 0, \, y \in \R,
$$
it is immediate to see that
\begin{align*}
| (-\Delta)^s \eta(x) | & \le \int_{\frac{x}{2}}^{\frac{3 x}{2}} \frac{|\eta(x) - \eta(y) - \eta'(x) (x - y)|}{|x - y|^{1 + 2 s}} \, dy + \int_{\left( -\infty, \frac{x}{2} \right) \cup \left( \frac{3 x}{2}, +\infty \right)} \frac{|\eta(x)| + |\eta(y)|}{|x - y|^{1 + 2 s}} \, dy \\
& \le C \left( x^{\sigma - 2} \int_{0}^{\frac{x}{2}} z^{1 - 2 s} \, dz + x^\sigma \int_{\frac{x}{2}}^{+\infty} \frac{dz}{z^{1 + 2 s}} + \int_{\frac{x}{2}}^{+\infty} z^{\sigma - 1 - 2s} \, dz \right) \le C x^{\sigma - 2 s} \le C.
\end{align*}
As~$\eta$ is odd, by relation~\eqref{fraclapLKrel} estimate~\eqref{Dsetabounded} yields that
\begin{equation} \label{LKetabounded}
|\L_{K_s} \eta(x)| \le C_\star \quad \mbox{for all } x \ge L,
\end{equation}
for some constant~$C_\star > 0$ depending only on~$s$,~$\sigma$, and~$L$.

Let now~$\{ \varepsilon_j \}$ be an infinitesimal sequence of positive numbers and define
$$
\phi_j(x) := \phi(x) - \varepsilon_j \eta(x) \quad \mbox{for } x \in \R_+.
$$
Since~$\phi_j$ is continuous in~$\R_+$ and we have
$$
\sup_{(0, L)} \phi_j \le \sup_{(0, L)} \phi \quad \mbox{and} \quad \lim_{x \rightarrow +\infty} \phi_j(x) = -\infty,
$$
in view of hypothesis~\eqref{u*notattin0L} it is clear that there exists a point~$y_j > L$ such that
\begin{equation} \label{yjmaxofuj}
\phi_j(y_j) \ge \phi_j(x) \quad \mbox{for all } x \in \R_+,
\end{equation}
at least for~$j$ large enough. Recalling~\eqref{Kspositive}, we deduce that~$\L_{K_s} \phi_j(y_j) \ge 0$ and, therefore, by the linearity of~$\L_{K_s}$ and the bound~\eqref{LKetabounded}, that
\begin{equation} \label{LKsugeCepsj}
\L_{K_s} \phi(y_j) \ge \varepsilon_j \L_{K_s} \eta(y_j) \ge - C_\star \varepsilon_j.
\end{equation}
On the other hand, let~$\{ z_k \} \subset (L, +\infty)$ be any sequence of points satisfying~$\phi^* - \phi(z_k) < 1/(2k)$ for every~$k \in \N$. Then, applying~\eqref{yjmaxofuj} with~$x = z_k$, we get
\begin{equation} \label{uyjge}
\phi(y_j) \ge \phi(z_k) - \varepsilon_j \eta(z_k) + \varepsilon_j \eta(y_j) \ge \phi^* - \frac{1}{2 k} - \varepsilon_j |z_k|^\sigma \quad \mbox{ for every large } j \mbox{ and every } k.
\end{equation}
For any fixed~$k$, let now~$j = j_k$ be sufficiently large to have~$C_\star \varepsilon_j \le 1/k$ and~$\varepsilon_j |z_k|^\sigma < 1/(2 k)$. From~\eqref{LKsugeCepsj} and~\eqref{uyjge} it then follows that the conclusion~\eqref{weakMPforLKs} holds true with~$x_k := y_{j_k}$.
\end{proof}

With this in hand, we may now proceed to establish Proposition~\ref{weakMPprop2}.

\begin{proof}[Proof of Proposition~\ref{weakMPprop2}]
Suppose, by contradiction, that~$\phi^* := \sup_{\R_+} \phi > 0$. As~$\sup_{(0, L)} \phi \le 0 < \phi^*$, we can apply Lemma~\ref{weakOYMPlem} and deduce the existence of a sequence of points~$\{ x_k \} \subset (L, +\infty)$ for which~\eqref{weakMPforLKs} holds true. Evaluating~\eqref{LKphi+dphi} along this sequence, we get
$$
0 \ge \L_{K_s} \phi(x_k) + d(x) \phi(x_k) \ge - \frac{1}{k} + \delta \left( \phi^* - \frac{1}{k} \right),
$$
for every~$k$ sufficiently large. Letting~$k \rightarrow +\infty$, we obtain~$0 \ge \delta \phi^*$, which contradicts our assumptions. We conclude that~$\phi \le 0$ in the whole~$\R_+$.
\end{proof}

Through similar techniques, we also have the following result, which holds for general subsolutions of equations set in the whole real line and driven by the fractional Laplacian, under no symmetry assumptions.

\begin{proposition} \label{weakMPprop3}
Let~$\phi \in L^\infty(\R) \cap C_\loc^{\alpha}(\R)$, for some~$\alpha > 2 s$, be such that
\begin{equation} \label{phisubsolinRminusL}
(-\Delta)^s \phi + d \phi \le 0 \quad \mbox{in } \R \setminus [- L, L],
\end{equation}
for some measurable function~$d: \R \setminus [- L, L] \to \R$ satisfying
\begin{equation} \label{d>delta|x|>L}
d(x) \ge \delta \quad \mbox{for all } x \in \R \setminus [- L, L],
\end{equation}
for some constants~$\delta, L > 0$. Assume that~$\phi \le 0$ in~$[-L, L]$. Then,~$\phi \le 0$ in~$\R$.
\end{proposition}
\begin{proof}
The argument is almost identical to those used to prove Lemma~\ref{weakOYMPlem} and Proposition~\ref{weakMPprop2}.

Assume by contradiction that~$\phi^* := \sup_{\R} \phi > 0$. First, one shows that there exists a sequence of points~$\{ x_k \} \subset \R \setminus [-L, L]$ such that~$\phi^* - \phi(x_k) < 1/k$ and~$(-\Delta)^s \phi(x_k) > - 1/k$ for every~$k \in \N$. This can be done as in Lemma~\ref{weakOYMPlem}, using that, by hypothesis,~$\phi \le 0$ in~$[- L, L]$ and considering perturbations~$\phi_j := \phi - \varepsilon_j \eta$ determined by, e.g., the function~$\eta(x) := (1 + x^2)^{\sigma / 2}$, with~$\sigma \in (0, 2 s)$. Observe that~$|(-\Delta)^s \eta(x)| \le C$ for all~$x \in \R$ and for some constant~$C$ depending only on~$s$ and~$\sigma$.

Then, evaluating~\eqref{phisubsolinRminusL} at~$x_k$, we obtain~$0 \ge (-\Delta)^s \phi(x_k) + d(x_k) \phi(x_k) \ge - 1/k + \delta (\phi^* - 1/k)$ for~$k$ large enough. Letting~$k \rightarrow +\infty$, we get a contradiction.
\end{proof}

\subsection{Decay estimates}

We now deal with the proofs of the main results of the sections. Before moving forward, we spend a few words on the regularity properties of the layer solution~$w$.

\begin{remark} \label{regrmk}
In view of, say,~\cite[Lemma~4.4]{CS14}, since~$W$ is of class~$C^3$ we know that~$w \in C^{2, \alpha}(\R)$, for some~$\alpha \in (0, 1)$. By differentiating equation~\eqref{ellPNeq} twice, we see that~$w''$ satisfies
\begin{equation} \label{eqforw''}
(-\Delta)^s w'' = - W''(w) w'' - W'''(w) (w')^2 \quad \mbox{in } \R.
\end{equation}
If~$W \in C^{3, 1}(\R)$, then the right-hand side of this equation lies in~$C^\alpha(\R)$. Using~\cite[Proposition~2.8]{S07}, we then get that~$w'' \in C^{2 s + \alpha}(\R)$. Accordingly, the regularity of the right-hand side of~\eqref{eqforw''} improves to~$C^{2 s + \alpha}$ and thus the solution~$w''$ belongs to~$C^{\min \{ 4 s + \alpha, 1 + 2 s \}}$. By applying this argument a finite number of times (as in the proofs of~\cite[Lemma~4.4]{CS14} or~\cite[Proposition~3.13]{CP16}), we eventually obtain that~$w'' \in C^{1 + 2 s}(\R)$. Differentiating~\eqref{eqforw''}, we have that~$w'''$ solves
$$
(-\Delta)^s w''' = - W''(w) w''' - 3 W'''(w) w' w'' - W''''(w) (w')^3 \quad \mbox{in } \R.
$$
Assuming the potential~$W$ to be of class~$C^{4 ,1}$ and proceeding as before, one gets that~$w''' \in C^{1 + 2 s}(\R)$.
\end{remark}

Knowing this, we now proceed with the proofs of Propositions~\ref{uimprovasymptprop} and~\ref{w''decayprop}.

Following the approach of~\cite[Section~6]{DPV15}, we take~$\omega_1$ as in~\eqref{omegasdef} and define
\begin{equation} \label{Adef}
A := \int_0^{+\infty} \omega_1(x) \, dx
\end{equation}
and
\begin{equation} \label{Thetadef}
\Omega(x) := \frac{1}{2} \left( 1 + \frac{1}{A} \int_0^x \omega_1(y) \, dy \right).
\end{equation}
The function~$\Omega$ is strictly increasing and it satisfies~$\Omega(0) = 1/2$,
\begin{equation} \label{Omega1odd}
\Omega(x) = 1 - \Omega(- x) \quad \mbox{for all } x \in \R,
\end{equation}
and~$\lim_{x \rightarrow +\infty} \Omega(x) = 1$. Moreover,
$$
\Omega'(x) = \frac{\omega_1(x)}{2 A} \quad \mbox{and} \quad \Omega''(x) = \frac{\omega_1'(x)}{2 A} = - \frac{1 + 2 s}{2 A} \, \omega_2(x) \quad \mbox{for all } x \in \R,
$$
where~$\omega_2$ is defined in~\eqref{omegasdef}. Finally, a straightforward computation shows that
\begin{equation} \label{Thetaasympt}
\left| \Omega(x) - \chi_{(0, +\infty)}(x) + \frac{1}{4 s A} \frac{x}{|x|^{1 + 2 s}} \right| \le \frac{1}{2 A} \frac{1}{|x|^{2 + 2 s}} \quad \mbox{for all } x \in \R.
\end{equation}

Set now
$$
L(x) := - (-\Delta)^s \Omega(x) \quad \mbox{for } x \in \R.
$$
Since~$\Omega$ is injective, we can consider its inverse~$\Omega^{-1}: (0, 1) \to \R$ and define
\begin{equation} \label{Vdef}
V(r) := \int_{-1}^r L(\Omega^{-1}(\rho)) \, d\rho \quad \mbox{for } r \in (0, 1).
\end{equation}
It is clear that, with this choice,~$\Omega$ solves the equation
\begin{equation} \label{eqforTheta}
(-\Delta)^s \Omega + V'(\Omega) = 0 \quad \mbox{in } \R.
\end{equation}
In addition, from~\eqref{Omega1odd} we infer that~$\Omega^{-1}$ and~$L$ are both odd (w.r.t.~$1/2$ and~$0$, respectively). Consequently,~$V' = L \circ \Omega^{-1}$ is also odd w.r.t.~$1/2$ and~$V$ is even w.r.t.~$1/2$. Finally, we have
\begin{align*}
V'(r) & = L(\Omega^{-1}(r)) = - (-\Delta)^s \Omega (\Omega^{-1}(r)), \\
V''(r) & = \frac{L'(\Omega^{-1}(r))}{\Omega'(\Omega^{-1}(r))}  = - \frac{(-\Delta)^s \omega_1(\Omega^{-1}(r))}{\omega_1(\Omega^{-1}(r))}, \\
V'''(r) & = \frac{\Omega'(\Omega^{-1}(r)) L''(\Omega^{-1}(r)) - \Omega''(\Omega^{-1}(r)) L'(\Omega^{-1}(r))}{\Omega'(\Omega^{-1}(r))^3} \\
& = 2 (1 + 2 s) A \frac{\omega_1(\Omega^{-1}(r)) (-\Delta)^s \omega_2(\Omega^{-1}(r)) - \omega_2(\Omega^{-1}(r)) (-\Delta)^s \omega_1(\Omega^{-1}(r))}{\omega_1(\Omega^{-1}(r))^3},
\end{align*}
for all~$r \in (0, 1)$. From this, recalling Lemmas~\ref{philem} and~\ref{psilem}, it follows that
\begin{equation} \label{V'V''V'''}
\begin{aligned}
\lim_{r \rightarrow 0} V'(r) & = - \lim_{x \rightarrow -\infty} (-\Delta)^s \Omega(x) = 0, \\
\lim_{r \rightarrow 0} V''(r) & = - \lim_{x \rightarrow -\infty} \frac{(-\Delta)^s \omega_1(x)}{\omega_1(x)} = 2 A, \\
\lim_{r \rightarrow 0} V'''(r) & = 2 (1 + 2 s) A \lim_{x \rightarrow -\infty} \frac{\omega_1(x) (-\Delta)^s \omega_2(x) - \omega_2(x) (-\Delta)^s \omega_1(x)}{\omega_1(x)^3} \\
& = - 2 (1 + 2 s) A \lim_{x \rightarrow -\infty} |x|^{2 s} \left\{ \frac{(-\Delta)^s \omega_2(x)}{\omega_2(x)} - \frac{(-\Delta)^s \omega_1(x)}{\omega_1(x)} \right\} \\
& = 2 (1 + 2 s) A \left( \lambda_{1, s} - \lambda_{2, s} \right).
\end{aligned}
\end{equation}
Consequently,~$V$ can be extended to a function of class~$C^3([0, 1])$, which is even w.r.t.~$1/2$ and satisfies
\begin{equation} \label{Vprops}
V(0) = V(1) = V'(0) = V'(1) = 0 \quad \mbox{and} \quad V''(0) = V''(1) = 2 A.
\end{equation}

To prove Proposition~\ref{uimprovasymptprop}, we will make use of the following abstract decay estimate.

\begin{proposition} \label{vdecayprop}
Let~$d, f: \R_+ \to \R$ be two measurable functions satisfying~\eqref{d>delta} and
\begin{equation} \label{fleMomega}
|f(x)| \le M \frac{x}{(1 + x^2)^{\frac{4 s + 1}{2}}} \quad \mbox{for all } x \in \R_+,
\end{equation}
for two constants~$\delta, M > 0$. Let~$v \in C_\loc^\alpha(\R)$, for some~$\alpha > 2 s$, be an odd function satisfying
\begin{equation} \label{vlimitatinfty}
\lim_{x \rightarrow \pm \infty} v(x) = 0
\end{equation}
and
\begin{equation} \label{vequation}
(-\Delta)^s v + d v = f \quad \mbox{in } \R_+.
\end{equation}
Then, there exists a constant~$C > 0$, depending only on~$s$,~$\delta$, and~$M$, such that
\begin{equation} \label{vbounds}
|v(x)| \le C \frac{|x|}{(1 + x^2)^{\frac{4 s + 1}{2}}} \quad \mbox{for all } x \in \R.
\end{equation}
\end{proposition}
\begin{proof}
Let~$\omega_3$ be the function defined in~\eqref{omegasdef}. For~$\varepsilon \in (0, 1]$ to be decided later, consider its rescaling
$$
\omega_3^{(\varepsilon)}(x) := \omega_3(\varepsilon x).
$$
Using that~$\varepsilon \le 1$, it is easy to see that
\begin{equation} \label{omegaepsomega}
\varepsilon \omega_3(x) \le \omega_3^{(\varepsilon)}(x) \le \varepsilon^{- 4 s} \omega_3(x) \quad \mbox{for all } x \in \R_+.
\end{equation}
Moreover, in light of Lemma~\ref{omegabarrierlem}, the function~$\omega_3^{(\varepsilon)}$ satisfies
$$
(-\Delta)^s \omega_3^{(\varepsilon)}(x) = \varepsilon^{2 s} (-\Delta)^s \omega_3(\varepsilon x) \ge - C_3 \varepsilon^{2 s} \omega_3(\varepsilon x) = - C_3 \varepsilon^{2 s} \omega_3^{(\varepsilon)}(x) \quad \mbox{for all } x \in \R_+.
$$
Therefore, taking~$\varepsilon := \min \{ (2 C_3 / \delta)^{-1/2s}, 1 \}$ and using hypothesis~\eqref{d>delta}, we get that
$$
(-\Delta)^s \omega_3^{(\varepsilon)} + d \omega_3^{(\varepsilon)} \ge \frac{\delta}{2} \, \omega_3^{(\varepsilon)} \quad \mbox{in } \R_+.
$$
Thanks to this, assumption~\eqref{fleMomega} on~$f$, the equation~\eqref{vequation} for~$v$, the definition~\eqref{omegasdef} of~$\omega_3$, and the left-hand inequality in~\eqref{omegaepsomega}, we have that~$\phi := v - 2 M (\delta \varepsilon)^{-1} \omega_3^{(\varepsilon)}$ satisfies
\begin{equation} \label{phisubharmonic}
(- \Delta)^s \phi + d \phi \le 0 \quad \mbox{in } \R_+.
\end{equation}

Observe that~$\phi$ is a bounded odd function of class~$C^\alpha$, with~$\alpha > 2 s$. In view of Remark~\ref{fraclapforoddrmk}, we may then rewrite~\eqref{phisubharmonic} as
$$
\L_{K_s} \phi + \widetilde{d} \phi \le 0 \quad \mbox{in } \R_+,
$$
with~$\widetilde{d}(x) := d(x) + x^{- 2 s} / s$. Notice that~$\widetilde{d} \ge d \ge \delta$ in~$\R_+$ and that~$K_s$ is positive. Furthermore,~$\phi$ satisfies~\eqref{philimits}. Hence, we can apply Proposition~\ref{maxprincprop} and deduce that~$\phi \le 0$ in~$\R_+$. Taking into account the right-hand inequality in~\eqref{omegaepsomega}, this means that~$v \le C \omega_3$ in~$\R_+$, for some constant~$C > 0$ depending only on~$s$,~$\delta$, and~$M$. As the same argument can be applied to~$- v$ instead of~$v$, we conclude that~\eqref{vbounds} holds true.
\end{proof}

We can now address Proposition~\ref{uimprovasymptprop}.

\begin{proof}[Proof of Proposition~\ref{uimprovasymptprop}]
For~$a > 0$, consider the rescaling~$\omega_1^{(a)}(x) := \omega_1(a x)$ of the function~$\omega_1$ defined by~\eqref{omegasdef}. Associated to this function, we also have the constant~$A_a$ and the function~$\Omega_a$, defined analogously to~\eqref{Adef} and~\eqref{Thetadef}. It is immediate to verify that~$A_a = A/a$ and~$\Omega_a = \Omega(a \, \cdot \,)$. Recalling~\eqref{eqforTheta}, it is also easy to see that~$\Omega_a$ satisfies the equation
$$
(-\Delta)^s \Omega_a + V_a'(\Omega_a) = 0 \quad \mbox{in } \R,
$$
where~$V_a(r) := a^{2 s} V(r)$ for every~$r \in [-1, 1]$ and~$V$ is defined by~\eqref{Vdef}. Notice that~$V_a \in C^3([0, 1])$ and that it is even w.r.t.~$1/2$. Furthermore, as a consequence of~\eqref{Vprops}, it satisfies
$$
V_a(0) = V_a(1) = V_a'(0) = V_a'(1) = 0 \quad \mbox{and} \quad V_a''(0) = V_a''(1) = 2 a^{2 s} A.
$$
The choice~$a := (W''(0) / (2 A))^{1/2s}$ gives that~$V_a''(0) = V_a''(1) = W''(0) = W''(1)$. Moreover, the asymptotic expansion~\eqref{Thetaasympt} for~$\Omega$ translates into
$$
\left| \Omega_a(x) - \chi_{(0, +\infty)}(x) + \frac{1}{2 s W''(0)} \frac{x}{|x|^{1 + 2 s}} \right| \le \frac{C}{|x|^{2 + 2 s}} \quad \mbox{for all } x \in \R,
$$
for some constant~$C> 0$ depending only on~$s$ and~$W$. Thus,~\eqref{uimprovasympt} will be proved if we show that
\begin{equation} \label{u-Thetaclaim}
\left| w(x) - \Omega_a(x) \right| \le \frac{C}{|x|^{4 s}} \quad \mbox{for all } x \in \R.
\end{equation}

To do this, we consider the difference~$v := w - \Omega_a$. In light of~\eqref{wodd} and~\eqref{Omega1odd}, we know that~$v$ is odd. Moreover, it satisfies
\begin{equation} \label{veq1}
(- \Delta)^s v = - W'(w) + V_a'(\Omega_a) \quad \mbox{in } \R.
\end{equation}
Using that~$W'(1) = V_a'(1) = 0$ and~$W''(1) = V_a''(1)$, we have the Taylor expansions
\begin{equation} \label{WVataylor}
\begin{aligned}
W'(r) & = W''(1) (r - 1) + \frac{W'''(1 - \theta_1(r) (1 - r))}{2} (r - 1)^2, \\
V_a'(r) & = W''(1) (r - 1) + \frac{V_a'''(1 - \theta_2(r) (1 - r))}{2} (r - 1)^2,
\end{aligned}
\end{equation}
for all~$r \in (0, 1)$ and for some~$\theta_1(r), \theta_2(r) \in [0, 1]$. Consequently, in view of~\eqref{wasympt}, of the the fact that~$\Omega_a - \chi_{(0, +\infty)}$ has the same decay, and of the~$C^3([0, 1])$ regularity of~$W$ and~$V_a$, we have 
$$
- W'(w(x)) + V'_a(\Omega_a(x)) = - W''(1) \left( w(x) - \Omega_a(x) \right) + O \! \left( |x|^{- 4 s} \right) \quad \mbox{as } x \rightarrow +\infty.
$$
Hence, equation~\eqref{veq1} can be rewritten as
\begin{equation} \label{deltav+W''v=f}
(- \Delta)^s v + W''(1) v = f \quad \mbox{in } \R,
\end{equation}
for some odd~$C^1$ function~$f$ satisfying the decay estimate~\eqref{fleMomega}, with~$M > 0$ depending only on~$s$ and~$W$. Since~$v$ also satisfies the limit condition~\eqref{vlimitatinfty}, we can apply Proposition~\ref{vdecayprop}, from which~\eqref{u-Thetaclaim} plainly follows.
\end{proof}

\begin{remark} \label{impros=12rmk}
When~$s = 1/2$, the auxiliary potential~$V$ satisfies~$V'''(0) = 0$. This can be checked, through the last identity in~\eqref{V'V''V'''}, by verifying that the two constants~$\lambda_{1, 1/2}$ and~$\lambda_{2, 1/2}$ found in Lemmas~\ref{philem} and~\ref{psilem} coincide. Alternatively---and more easily---one can deduce it from definition~\eqref{Vdef}, by realizing that~$\Omega$ is the layer solution corresponding to the explicit potential~$V(r) = \frac{1}{4\pi} (1 - \cos (2 \pi r))$---see~\cite{T97,CS05}.

As a consequence of the vanishing of~$V'''(0)$, when~$W \in C^4(\R)$ one can get a sharp version of~\eqref{uimprovasympt}, in which its right-hand side decays as~$|x|^{- 3}$, instead of~$|x|^{- 4s} = |x|^{- 2}$. To see this, it suffices to notice that, in this case, the Taylor expansions~\eqref{WVataylor} give that~$W'(r)$ and~$V_a'(r)$ are approximated by their linear parts with an error of the order of~$(1 - r)^3$. From this, it follows that the function~$f$ appearing in~\eqref{deltav+W''v=f} decays as~$x^{- 6s} = x^{- 3}$ as~$x \rightarrow +\infty$. One then concludes by observing that this decay is transferred to~$v$, a fact that can be established via a conveniently modified version of Proposition~\ref{vdecayprop}, proved using~$\omega_2$ instead of~$\omega_3$ as a barrier.

For~$s \ne 1/2$, it seems that~$\lambda_{1, s}$ and~$\lambda_{2, s}$ are generally not equal, and therefore that~$V'''(0) \ne 0$. As a result, the improvement just described cannot be performed as is. We believe it would be very interesting to understand whether a different auxiliary function can be used in place of~$\Omega$, to change~$V$ into a new even potential that still satisfies~\eqref{Vprops}, but with vanishing third derivative at~$0$. This would lead to the improvement of~\eqref{uimprovasympt}.
\end{remark}

We now move on to the proof of Proposition~\ref{w''decayprop}. To carry it through, we first need the following two auxiliary estimates.

\begin{proposition} \label{vdecayprop2}
Let~$d, f: \R_+ \to \R$ be two measurable functions satisfying~\eqref{d>deltax>L} and
\begin{equation} \label{gleNpsi}
|f(x)| \le \frac{M}{(1 + |x|)^{2 + 2 s}} \quad \mbox{for all } x > L,
\end{equation}
for three constants~$\delta, L, M > 0$. Let~$v \in L^\infty(\R) \cap C_\loc^\alpha(\R)$, for some~$\alpha > \max \{ 2 s, 1 \}$, be an odd function satisfying
\begin{equation} \label{wequation}
(-\Delta)^s v + d v = f \quad \mbox{in } (L, +\infty).
\end{equation}
Then, there exists a constant~$C > 0$, depending only on~$s$,~$\delta$,~$L$,~$M$, and~$\| v' \|_{L^\infty(0, L)}$, such that
$$
|v(x)| \le \frac{C}{(1 + |x|)^{2 + 2 s}} \quad \mbox{for all } x \in \R_+.
$$
\end{proposition}
\begin{proof}
Let~$\omega_2$ be as in~\eqref{omegasdef} and, like in the proof of Proposition~\ref{vdecayprop}, let
$$
\omega_2^{(\varepsilon)}(x) := \omega_2(\varepsilon x),
$$
for~$\varepsilon \in (0, 1]$. We have
\begin{equation} \label{psiepsbounds}
\varepsilon \omega_2(x) \le \omega_2^{(\varepsilon)}(x) \le \varepsilon^{-2 - 2s } \omega_2(x) \quad \mbox{for all } x \in \R_+
\end{equation}
and, using inequality~\eqref{Dspsilepsi} of Lemma~\ref{Dspsiasympt},
$$
(-\Delta)^s \omega_2^{(\varepsilon)}(x) = \varepsilon^{2s} (-\Delta)^s \omega_2(\varepsilon x) \ge - C_2 \varepsilon^{2 s} \omega_2(\varepsilon x) = - C_2 \varepsilon^{2 s} \omega_2^{(\varepsilon)}(x) \quad \mbox{for all } x \in \R_+.
$$
Hence, taking~$\varepsilon := \min \{ (2 C_2 / \delta)^{- 1/2s}, 1 \}$ and using that~$d$ satisfies~\eqref{d>deltax>L}, we conclude that
\begin{equation} \label{psiepseq}
(-\Delta)^s \omega_2^{(\varepsilon)} + d \omega_2^{(\varepsilon)} \ge \frac{\delta}{2} \, \omega_2^{(\varepsilon)} \quad \mbox{for all } x > L.
\end{equation}

Let now
\begin{equation} \label{Ccircdef}
C_\circ := \frac{1}{\varepsilon} \max \left\{ (1 + L^2)^{\frac{3 + 2 s}{2}} \| v' \|_{L^\infty(0, L)}, \frac{1 + L}{L} \frac{2 M}{\delta} \right\}
\end{equation}
and define~$\phi := - v - C_\circ \omega_2^{(\varepsilon)}$. By the right-hand bound in~\eqref{psiepsbounds} and the fact that a similar estimate can be proved with~$v$ in place of~$-v$, the proof will be over if we show that
\begin{equation} \label{phile0claim}
\phi \le 0 \quad \mbox{in } \R_+.
\end{equation}
First, by the left-hand inequality in~\eqref{psiepsbounds}, the definition~\eqref{omegasdef} of~$\omega_2$, and~\eqref{Ccircdef}, it holds
$$
- v(x) = - \int_0^x v'(y) \, dy \le \| v' \|_{L^\infty(0, L)} x \le C_\circ \frac{\varepsilon x}{(1 + L^2)^{\frac{3 + 2 s}{2}}} \le C_\circ \omega_2^{(\varepsilon)}(x) \quad \mbox{for all } x \in [0, L].
$$
That is,~$\phi \le 0$ in~$[0, L]$. Using~\eqref{gleNpsi},~\eqref{wequation},~\eqref{psiepseq}, the left-hand bound in~\eqref{psiepsbounds}, and~\eqref{Ccircdef}, it is easy to see that
$$
(-\Delta)^s \phi + d \phi \le 0 \quad \mbox{for all } x > L.
$$
As~$\phi$ is odd, recalling identity~\eqref{fraclapLKrel} the above inequality can be rephrased as
$$
\L_{K_s} \phi + \widetilde{d} \phi \le 0 \quad \mbox{for all } x > L,
$$
where~$\widetilde{d}(x) := d(x) + x^{-2 s} / s$. As, by~\eqref{d>deltax>L},~$\widetilde{d} \ge \delta$ in~$(L, +\infty)$, and~$\phi \le 0$ in~$(0, L]$, we are in position to apply Proposition~\ref{weakMPprop2} and conclude that~\eqref{phile0claim} holds true.
\end{proof}

Similarly, we can prove the next estimate for solutions that are not necessarily odd.

\begin{proposition} \label{vdecayprop3}
Let~$d, f: \R \to \R$ be two measurable functions satisfying~\eqref{d>delta|x|>L} and
\begin{equation} \label{fleMomega1}
|f(x)| \le \frac{M}{(1 + |x|)^{1 + 2 s}} \quad \mbox{for all } x \in \R \setminus [-L, L],
\end{equation}
for three constants~$\delta, L, M > 0$. Let~$v \in L^\infty(\R) \cap C_\loc^\alpha(\R)$, for some~$\alpha > 2 s$, be a solution of
\begin{equation} \label{vequation3}
(-\Delta)^s v + d v = f \quad \mbox{in } \R \setminus [- L, L].
\end{equation}
Then, there exists a constant~$C > 0$, depending only on~$s$,~$\delta$,~$L$,~$M$, and~$\| v \|_{L^\infty(\R)}$, such that
\begin{equation} \label{vbounds3}
|v(x)| \le \frac{C}{(1 + |x|)^{1 + 2 s}} \quad \mbox{for all } x \in \R.
\end{equation}
\end{proposition}
\begin{proof}
The argument is similar to those used to prove Propositions~\ref{vdecayprop} and~\ref{vdecayprop2}. Given~$\varepsilon \in (0, 1]$ and~$\omega_1$ as in~\eqref{omegasdef}, we define~$\omega_1^{(\varepsilon)} := \omega_1(\varepsilon \, \cdot \,)$. It is easy to check that, if~$\varepsilon \le (2 C_1 / \delta)^{- 1/2s}$ with~$C_1$ as in~\eqref{Dsomega1leomega1}, then~$\omega_1$ satisfies
$$
(-\Delta)^s \omega_1^{(\varepsilon)} + d \omega_1^{(\varepsilon)} \ge \frac{\delta}{2} \, \omega_1^{(\varepsilon)} \quad \mbox{in } \R \setminus [-L, L].
$$
Hence, setting~$\phi := v - C_{\mbox{\tiny $\bullet$}} \omega_1^{(\varepsilon)}$, with~$C_{\mbox{\tiny $\bullet$}} := \max \{ (2M) / (\delta \varepsilon), (1 + L)^{1 + 2 s} \| v \|_{L^\infty(-L, L)} \}$, from~\eqref{vequation3} and~\eqref{fleMomega1} it follows that
$$
(-\Delta)^s \phi + d \phi \le 0 \quad \mbox{in } \R \setminus [-L, L].
$$
In addition, taking~$\varepsilon \le 1$, we get
$$
\phi(x) = v(x) - C_{\mbox{\tiny $\bullet$}} \omega_1^{(\varepsilon)}(x) \le \| v \|_{L^\infty(-L, L)} - C_{\mbox{\tiny $\bullet$}} (1 + \varepsilon^2 x^2)^{- \frac{1 + 2 s}{2}} \le 0 \quad \mbox{for all } x \in [-L, L].
$$
Thus, we can apply Proposition~\ref{weakMPprop3} to~$\phi$ and deduce that~$\phi \le 0$, that is,~$v \le C_{\mbox{\tiny $\bullet$}} \omega_1^{(\varepsilon)}$ in the whole of~$\R$. Since the analogous bound from below for~$v$ can be established in a similar way, we conclude that~\eqref{vbounds3} holds true.
\end{proof}

Thanks to the last two results, we are in position to establish Proposition~\ref{w''decayprop}.

\begin{proof}[Proof of Proposition~\ref{w''decayprop}]
The second derivative~$w''$ of the layer solution is odd, bounded, of class~$C^{2 s + 1}$ (recall Remark~\ref{regrmk}), and it satisfies equation~\eqref{eqforw''}, which we rewrite as
$$
(-\Delta)^s w'' + d w'' = f_1 \quad \mbox{in } \R,
$$
with~$d := W''(w)$ and~$f_1 := - W'''(w) (w')^2$. Notice that, thanks to the boundedness of~$W'''$ and the decay estimate~\eqref{w'asympt} for~$w'$, we have that
$$
\left| f_1(x) \right| \le \frac{M_1}{(1 + |x|)^{2 + 4 s}} \quad \mbox{for all } x \in \R,
$$
for some constant~$M_1 > 0$ depending only on~$s$ and~$W$. Furthermore, as~$w(x) \rightarrow 1$ as~$x \rightarrow +\infty$ and~$W''$ is continuous, the zeroth order coefficient~$d$ satisfies~\eqref{d>deltax>L} for some~$L > 0$ and with~$\delta := W''(1) /2 > 0$. Consequently, we can apply Proposition~\ref{vdecayprop2} with~$v := w''$ and conclude that~\eqref{w''decay} holds true.

We now deal with~\eqref{w'''decay}. Recalling again Remark~\ref{regrmk}, under the assumption that~$W \in C^{4, 1}(\R)$ we have that~$w''' \in C^{2 s + 1}(\R)$ and that it satisfies
$$
(-\Delta)^s w''' + d w''' = f_2 \quad \mbox{in } \R,
$$
with~$d := W''(w)$ and~$f_2 := - 3 W'''(w) w' w'' - W''''(w) (w')^3$. Arguing as before, we know that~$d$ satisfies~\eqref{d>delta|x|>L} for some~$L > 0$ and with~$\delta := W''(0) / 2 = W''(1) / 2 > 0$. Moreover, using~\eqref{w'asympt},~\eqref{w''decay}, and the boundedness of~$W'''$ and~$W''''$, we get that
$$
\left| f_2(x) \right| \le \frac{M_2}{(1 + |x|)^{3 + 4 s}} \quad \mbox{for all } x \in \R,
$$
for some~$M_2 > 0$. Estimate~\eqref{w'''decay} then follows after an application of Proposition~\ref{vdecayprop3}.
\end{proof}

\section{Proof of Proposition~\ref{regforstrongprop}} \label{estforstrongapp}

\noindent
We address here the regularity estimates of Proposition~\ref{regforstrongprop}. Recalling Definition~\ref{strongsoldef}, we write the mild solution~$u$ of problem~\eqref{Gprob} (with~$f = f(x, t)$ in place of~$G[u]$) as
$$
u = U_0 + U_1,
$$
with
\begin{alignat}{3}
\label{U0def}
U_0(x, t) & := \int_\R p(x - z, t - t_0) u_0(z) \, dz && \qquad \mbox{for } x \in \R, \, t > t_0,\\
\label{U1def}
U_1(x, t) & := \int_{t_0}^t \int_\R p(x - z, t - \sigma) f(z, \sigma) \, dz d\sigma && \qquad \mbox{for } x \in \R, \, t \in (t_0, t_1).
\end{alignat}

To establish Proposition~\ref{regforstrongprop}, we will use the following bounds on the heat kernel~$p$, in addition to the properties~\ref{preg}-\ref{pbounds} listed in Section~\ref{solsec}:
\begin{align}
\label{pxbound}
|\partial_x^k p(x, t)| & \le C \, t \left( |x| + t^{1/2s} \right)^{- 1 - 2 s - k} \quad \mbox{for } k \in \N,\\
\label{ptbound}
|\partial_t p(x, t)| & \le C \, t^{-1} p(x, t),
\end{align}
for all~$x \in \R$,~$t > 0$, and for some constant~$C>0$ depending only on~$s$ and possibly~$k$. A proof of~\eqref{pxbound} can be found in~\cite[Lemma~2.2]{CZ16}, while~\eqref{ptbound} follows from~\cite[Proposition~2.1]{VDQR17}.

We begin by dealing with~$U_0$.

\begin{lemma} \label{U0lem}
Let~$u_0 \in L^\infty(\R)$ and~$U_0$ be given by~\eqref{U0def}. Then, for every~$t_\star > t_0$ and~$k \in \N$, it holds
\begin{equation} \label{Ulocest}
\sup_{t > t_\star} \| U_0(\cdot, t) \|_{C^{k}(\R)} + \sup_{x \in \R} \| U_0(x, \cdot) \|_{C^1(t_\star, +\infty)} \le C_{t_\star} \| u_0 \|_{L^\infty(\R)},
\end{equation}
for some constant~$C_{t_\star}$ depending only on~$s$,~$k$, and~$t_\star$. Moreover, if~$u_0 \in C^\alpha(\R)$ for some~$\alpha \in (0, 2 s \wedge 1)$, then
\begin{equation} \label{Uglobest}
\sup_{t > t_0} \| U_0(\cdot, t) \|_{C^\alpha(\R)} + \sup_{x \in \R} \| U_0(x, \cdot) \|_{C^{\frac{2s \wedge 1}{2s} \alpha}(t_0, t_0 + 1)} \le C \| u_0 \|_{C^\alpha(\R)},
\end{equation}
for some constant~$C$ depending only on~$s$ and~$\alpha$.
\end{lemma}
\begin{proof}
Up to a translation in the variable~$t$, we may assume that~$t_0 = 0$. We first address~\eqref{Ulocest}. On the one hand, using~\eqref{pxbound},~\ref{pbounds},~\ref{pmass=1}, and a change of coordinates, we have
\begin{equation} \label{Ujxbound}
\begin{aligned}
|\partial^k_x U_0(x, t)| & \le \| u_0 \|_{L^\infty(\R)} \int_\R | \partial^k_x p(x - z, t)| \, dz \le C \, t^{- k/2s} \| u_0 \|_{L^\infty(\R)} \int_\R p(y, t) \, dy \\
& \le C \, t_\star^{-k/2s} \| u_0 \|_{L^\infty(\R)},
\end{aligned}
\end{equation}
for all~$x \in \R$,~$t > t_\star > 0$, and~$k \in \N \cup \{ 0\}$. On the other hand, by~\eqref{ptbound} and~\ref{pmass=1},
$$
|\partial_t U_0(x, t)| \le \| u_0 \|_{L^\infty(\R)} \int_\R |\partial_t p(x - z, t)| \, dz \le C \, t^{- 1} \| u_0 \|_{L^\infty(\R)} \int_\R p(y, t) \, dy \le C \, t_\star^{-1} \| u_0 \|_{L^\infty(\R)},
$$
for all~$x \in \R$ and~$t > t_\star > 0$. These two inequalities give~\eqref{Ulocest}.

To prove~\eqref{Uglobest}, we use once again~\ref{pmass=1} to find that
\begin{equation} \label{UglobCalphaspace}
\left| U_0(x, t) - U_0(y, t) \right| \le \int_\R p(z, t) |u_0(x - z) - u_0(y - z)| \, dz \le [u_0]_{C^\alpha(\R)} |x - y|^\alpha,
\end{equation}
for all~$x, y \in \R$ and~$t > 0$. On the other hand, taking into account~\ref{pscales} and changing variables appropriately, we write
$$
U_0(x, t) = t^{-1/2s} \int_{\R} p(t^{-1/2s} (x - z), 1) u_0(z) \, dz = \int_{\R} p(w, 1) u_0(x - t^{1/2s} w) \, dw.
$$
But then, by~\ref{pbounds} we get
\begin{align*}
|U_0(x, t) - U_0(x, \tau)| & \le \int_{\R} p(w, 1) |u_0(x - t^{1/2s} w) - u_0(x - \tau^{1/2s} w)| \, dw \\
& \le C [u_0]_{C^\alpha(\R)} \left| t^{1/2s} - \tau^{1/2s} \right|^\alpha \int_{\R} \frac{|w|^\alpha}{(1 + |w|)^{1 + 2 s}} \, dw \le C [u_0]_{C^\alpha(\R)} |t - \tau|^{\frac{2s \wedge 1}{2s} \alpha},
\end{align*}
for all~$x \in \R$ and~$t, \tau \in (0, 1)$. Claim~\eqref{Uglobest} follows from this,~\eqref{UglobCalphaspace}, and~\eqref{Ujxbound} with~$k = 0$.
\end{proof}

We now address the regularity of~$U_1$.

\begin{lemma} \label{U1lem}
Let~$f \in L^\infty(\R \times (t_0, t_1))$ and~$U_1$ be given by~\eqref{U1def}. Then, for every~$\sigma \in (0, s]$ ($\sigma < s$, if~$s = 1/2$),~$\theta \in (0, 1)$, and~$T \ge t_1 - t_0$, it holds
\begin{equation} \label{U1est}
\sup_{t \in (t_0, t_1)} \| U_1(\cdot, t) \|_{C^{2 \sigma}(\R)} + \sup_{x \in \R} \| U_1(x, \cdot) \|_{C^\theta(t_0, t_1)} \le C_T \| f \|_{L^\infty(\R \times (t_0, t_1)},
\end{equation}
for some constant~$C_T$ depending only on~$s$,~$\sigma$,~$\theta$, and~$T$.
\end{lemma}
\begin{proof}
As in Lemma~\ref{U0lem}, we assume without loss of generality that~$t_0 = 0$. We start dealing with the first term on the left-hand side of~\eqref{U1est}. By property~\ref{pmass=1}, we have
\begin{equation} \label{Vbounded}
|U_1(x, t)| \le T \| f \|_{L^\infty(\R \times (0, t_1))} \quad \mbox{for all } x \in \R, \, t \in (0, t_1).
\end{equation}

In order to bound the spatial~$C^{2 \sigma}$ seminorm of~$U_1$, we first observe that~\eqref{pxbound} and~\ref{pbounds} give
\begin{equation} \label{pkxLip}
\left| \partial_x^j p(x, t) - \partial_x^j p(y, t) \right| \le C \, t^{-j/2s} \left( t^{-1/2s} |x - y| \wedge 1 \right) \! \Big( p(x, t) + p(y, t) \Big),
\end{equation}
for all~$x, y \in \R$,~$t > 0$, and~$j \in \N \cup \{ 0 \}$. We then consider separately the two cases~$s \in (0, 1/2]$ and~$s \in (1/2, 1)$.

When~$s \in (0, 1/2]$, we employ~\eqref{pkxLip} with~$j = 0$ in combination with~\ref{pmass=1} and compute
\begin{align*}
\left| U_1(x, t) - U_1(y, t) \right| & \le \| f \|_{L^\infty(\R \times (0, t_1))} \int_0^t \int_{\R} |p(x - z, \tau) - p(y - z, \tau)| \, dz d\tau \\
& \le C \| f \|_{L^\infty(\R \times (0, t_1))} \int_0^t \int_{\R} \left( \tau^{-1/2s} |x - y| \wedge 1 \right) \! \Big( p(x - z, \tau) + p(y - z, \tau) \Big) \, dz d\tau \\
& \le C \| f \|_{L^\infty(\R \times (0, t_1))} \left( \int_0^{|x - y|^{2 s}} d\tau + |x - y| \int_{|x - y|^{2 s}}^t \frac{d\tau}{\tau^{1/2s}} \right),
\end{align*}
for all~$t \in (0, t_1)$ and~$x, y \in \R$ such that~$|x - y| < t^{1/2s}$. It follows that
\begin{equation} \label{U1spaces<1/2}
\left| U_1(x, t) - U_2(y, t) \right| \le C \| f \|_{L^\infty(\R \times (0, t_1))}
\begin{dcases}
|x - y|^{2 s} & \quad \mbox{if } s \in (0, 1/2), \\
|x - y| \left( 1 + \log \frac{T}{|x - y|} \right) & \quad \mbox{if } s = 1/2,
\end{dcases}
\end{equation}
for all such~$x$,~$y$, and~$t$.

Suppose now that~$s \in (1/2, 1)$. In this case, using~\eqref{pxbound} with~$k = 1$,~\ref{pbounds}, and~\ref{pmass=1} we get
\begin{equation} \label{VLipest}
\begin{aligned}
|\partial_x U_1(x, t)| & \le \| f \|_{L^\infty(\R \times (0, t_1))} \int_0^t \int_{\R} |\partial_x p(x - z, \tau)| \, dz d\tau \\
& \le C \| f \|_{L^\infty(\R \times (0, t_1))} \int_0^t \tau^{-1/2s} \left( \int_{\R} p(x - z, \tau) \, dz \right) d\tau \le C \, T^{\frac{2 s - 1}{2s}} \| f \|_{L^\infty(\R \times (0, t_1))},
\end{aligned}
\end{equation}
for all~$x \in \R$ and~$t \in (0, T)$. In addition, applying~\eqref{pkxLip} with~$j = 1$ and arguing similarly to before, we obtain
$$
|\partial_x U_1(x, t) - \partial_x U_1(y, t)| \le C \| f \|_{L^\infty(\R \times (0, t_1))} |x - y|^{2 s - 1},
$$
for all~$t \in (0, t_1)$ and~$x, y \in \R$ such that~$|x - y| < t^{1/2s}$. From this,~\eqref{Vbounded},~\eqref{U1spaces<1/2}, and~\eqref{VLipest}, we conclude that the bound for~$\sup_{t \in (0, t_1)} \| U_1(\cdot, t) \|_{C^{2 \sigma}(\R)}$ in~\eqref{U1est} holds true.

We now move to the time regularity of~$U_1$. From~\eqref{ptbound} it follows that
$$
|p(x, t) - p(x, \tau)|
%\le C \int_\tau^t \frac{p(x, \sigma)}{\sigma} \, d\sigma
\le C |\log t - \log \tau| \left( p(x, t) + p(x, \tau) \right),
%\le C |t - \tau| \left( \frac{p(x, t)}{t} + \frac{p(x, \tau)}{\tau} \right),
$$
for all~$x \in \R$ and~$t, \tau \in (0, +\infty)$. Using this and~\ref{pmass=1}, we compute, for~$x \in \R$ and~$0 < \tau \le t < t_1$,
\begin{align*}
|U_1(x, t) - U_1(x, \tau)| & \le \int_0^\tau \int_\R |p(x - z, t - \sigma) - p(x - z, \tau - \sigma)| |f(z, \sigma)| \, dz d\sigma \\
& \quad + \int_\tau^t \int_\R p(x - z, t - \sigma) |f(z, \sigma)| \, dz d\sigma \\
& \le C \| f \|_{L^\infty(\R \times (0, t_1))} \left\{ \int_0^\tau \left( \log (t - \sigma) - \log (\tau - \sigma) \right) d\sigma + (t - \tau) \right\} \\
%& \le C \| f \|_{L^\infty(\R \times (0, t_1))} \left\{ t \log t - (t - \tau) \log (t - \tau) - \tau \log \tau + t - \tau \right\} \\
%& \le C \| f \|_{L^\infty(\R \times (0, t_1))} \left\{ t \log \frac{t}{t - \tau} - \tau \log \frac{\tau}{t - \tau} + t - \tau \right\} \\
& \le C \| f \|_{L^\infty(\R \times (0, t_1))} (t - \tau) \left( 1 + \log \frac{T}{t - \tau} \right).
\end{align*}
Estimate~\eqref{U1est} for~$\sup_{x \in \R} \| U_1(x, \cdot) \|_{C^\theta(t_0, t_1)}$ follows from this and~\eqref{Vbounded}.
\end{proof}

Proposition~\ref{regforstrongprop} is an immediate consequence of Lemmas~\ref{U0lem} and~\ref{U1lem}.

\section{Non-existence of stationary multiple dislocations} \label{nomultiapp}

\noindent
In Theorem~\ref{mainthm}, we constructed a solution~$u$ of the parabolic Peierls-Nabarro equation~\eqref{maineqfromT} that models the evolution of~$N \ge 2$ equally oriented dislocations. In particular, at each time~$t > T$, the function~$u(\cdot, t)$ connects the two zeroes~$0$ and~$N$ of the potential~$W$ at~$\pm \infty$.

One may wonder whether there exist \emph{stationary} solutions of~\eqref{maineqfromT} that have the same property, that is, a solution~$v$ of
\begin{equation} \label{ellPNeqapp}
(-\Delta)^s v + W'(v) = 0 \quad \mbox{in } \R,
\end{equation}
such that~$\lim_{x \rightarrow - \infty} v(x) = 0$ and~$\lim_{x \rightarrow +\infty} v(x) = N$, for some integer~$N \ge 2$. The answer to this question is known to be negative if we additionally require~$v$ to be monotone---this is a simple consequence of the Modica type estimates of~\cite{CS05,CS14}---see, for instance,~\cite[Theorem~2.2]{CS14}. The same holds true regardless of the monotonicity assumption when~$s = 1/2$ and~$W(r) = 1 - \cos(2 \pi r)$, as in this case bounded solutions of~\eqref{ellPNeqapp} have been completely classified by Toland~\cite{T97}.

The aim of this short appendix is to show that no such solution exists (monotone or not) for a general~$s \in (0, 1)$ and periodic multi-well potential~$W$. Differently from the rest of the paper, the parity of~$W$ plays no role here. In fact, we only assume~$W$ to satisfy
\begin{equation} \label{Wpropapp}
\begin{alignedat}{3}
& W(r + k) = W(r) && \quad \mbox{for all } r \in \R, \, k \in \Z, \\
& W(r) > 0 && \quad \mbox{for all } r \in (0, 1), \\
& W(0) = W'(0) = 0, && \quad \mbox{and} \quad W''(0) > 0.
\end{alignedat}
\end{equation}

%Call~$w_+$ and~$w_-$ the monotone increasing and decreasing layer solutions of~\eqref{ellPNeqapp}. Namely,~$w_\pm$ is the unique solution of~\eqref{ellPNeqapp} satisfying~$\lim_{x \rightarrow \mp \infty} w_\pm(x) = 0$,~$\lim_{x \rightarrow \pm \infty} w_\pm(x) = 1$,~$\pm w_\pm' > 0$ in~$\R$, and~$w_\pm(0) = 1/2$---see, e.g.,~\cite[Theorem~2.4]{CS15}. Of course,~$w_+$ is the solution that was previously indicated with~$w$ and it holds~$w_-(x) = w_+(- x)$ for all~$x \in \R$ if~$W$ is even w.r.t.~$1/2$.

We then have the following result. Recall that~$w$ is the layer solution of~\eqref{ellPNeqapp}, i.e., the unique increasing solution of~\eqref{ellPNeqapp} satisfying~\eqref{wcond}.

\begin{proposition} \label{classprop}
Let~$s \in (0, 1)$ and~$W \in C^{2, 1}(\R)$ be a potential satisfying~\eqref{Wpropapp}. If~$v$ is a bounded solution of~\eqref{ellPNeqapp} for which
\begin{equation} \label{vliminZ}
\lim_{x \rightarrow \pm \infty} v(x) \in \Z,
\end{equation}
then either~$v$ is constant or it holds
\begin{equation} \label{v=w+k}
v(x) = w(a x + b) + k \quad \mbox{for all } x \in \R, 
\end{equation}
for some~$a \in \{ -1, 1 \}$,~$b \in \R$, and~$k \in \Z$.
\end{proposition}

Proposition~\ref{classprop} highlights the rigidity of equation~\eqref{ellPNeqapp}. This is in sharp contrast with other closely related models---such as, for instance, arbitrarily small heterogeneous perturbations of~\eqref{ellPNeqapp}---, which possess a plethora of qualitatively different solutions connecting the zeroes of the potential~$W$---see, e.g.,~\cite{DPV17} and~\cite[Section~5]{CDV16}.

We do not know whether other solutions of~\eqref{ellPNeqapp} exist, besides those mentioned in Proposition~\ref{classprop} and the periodic ones that will be constructed in~\cite{CMS19}---i.e., whether a classification result as that established in~\cite{T97} for the case~$s = 1/2$,~$W(r) = 1 - \cos(2 \pi r)$ could be established in this more general setting too.

The proof of Proposition~\ref{classprop} is based on an application of the sliding method, similar to the one used to verify the uniqueness statement of~\cite[Theorem~2]{PSV13}. Following are the details.

\begin{proof}[Proof of Proposition~\ref{classprop}]
First of all, by the results of~\cite[Section~2]{S07}, we know that~$v \in C^2(\R)$---see also~\cite[Lemma~4.4]{CS14}. Up to a vertical translation, we can suppose that~$\lim_{x \rightarrow -\infty} v(x) = 0$. Then, we either have that~$v \equiv 0$, in which case we are done, or at least one between~$\sup_{\R} v$ and~$- \inf_{\R} v$ is strictly positive. We assume the former to hold, as the latter case can be dealt with analogously.

Set~$M := \sup_{\R} v$ and let~$k$ be the largest integer strictly smaller than~$M$, that is,~$k := \lceil M \rceil - 1$. As~$M > 0$, we have that~$k \ge 0$. For~$b \in \R$ and~$j \in \N$, consider the function~$w_{b, j}(x) := w(x + b) + k + 1/j$.
We claim that
\begin{equation} \label{v<wbjclaim}
v < w_{b, j} \mbox{ in } \R, \mbox{ if } b \mbox{ is sufficiently large.}
\end{equation}
Indeed, as~$\lim_{x \rightarrow -\infty} v(x) = 0$, there exists~$L_j > 0$ such that~$v(x) < 1/j$ for every~$x < -L_j$. Thus, in particular,~$v(x) < w_{b, j}(x)$ for every~$x < - L_j$ and~$b \in \R$. But then, as~$\lim_{b \rightarrow +\infty} w_{b, j}(-L_j) \ge M + 1/j$ and~$w_{b, j}$ is increasing, we have that~$v(x) \le M < w_{b, j}(x)$ for all~$x \ge - L_j$, provided~$b$ is sufficiently large. This proves~\eqref{v<wbjclaim}.

Assuming~$j > (M - \lceil M \rceil + 1)^{-1}$, we now take~$b$ smaller and smaller, until~$w_{b, j}$ first touches~$v$ from above. That is, we consider~$b_j \in \R$ and~$\bar{x}_j \in \R$ such that
\begin{equation} \label{w=v}
w_{b_j, j}(\bar{x}_j) = v(\bar{x}_j)
\end{equation}
and
\begin{equation} \label{w>v}
w_{b_j, j}(y) \ge v(y) \quad \mbox{for all } y \in \R.
\end{equation}
The existence of~$b_j$ and~$\bar{x}_j$ easily follows from the fact that~$\lim_{b \rightarrow - \infty} w_{b, j} \equiv k + 1/j < M$. Since both~$v$ and~$w$ solve~\eqref{ellPNeqapp} and~\eqref{w=v}-\eqref{w>v} hold true, we get that
\begin{equation} \label{eqforvandw}
W' \! \left( \! v(\bar{x}_j) - \frac{1}{j} \right) - W'(v(\bar{x}_j)) = (-\Delta)^s (v - w_{b_j, j})(\bar{x}_j) =\int_\R \frac{w_{b_j, j}(y) - v(y)}{|\bar{x}_j - y|^{1 + 2 s}} \, dy \ge 0.
\end{equation}

We now claim that~$\{ \bar{x}_j \}$ is a bounded sequence. If this is not the case, then we can extract a subsequence~$\{ \bar{x}_{j_i} \}$ such that
\begin{equation} \label{xjidiverges}
\lim_{i \rightarrow +\infty} |\bar{x}_{j_i}| = +\infty.
\end{equation}
As~$W''(0) > 0$ and~$W'$ is~$1$-periodic---by assumption~\eqref{Wpropapp}---, we have that~$W'$ is strictly increasing in all intervals~$[k - \varepsilon, k + \varepsilon]$ with~$k \in \Z$, for some small~$\varepsilon > 0$. In view of~\eqref{vliminZ} and~\eqref{xjidiverges}, both~$v(\bar{x}_{j_i}) - 1/j_i$ and~$v(\bar{x}_{j_i})$ lie in one of those intervals for~$i$ large enough. Hence,~$W'(v(\bar{x}_{j_i}) - 1/j_i) < W'(v(\bar{x}_{j_i}))$ for all such~$i$'s, which clearly contradicts~\eqref{eqforvandw}.

From its boundedness, we infer that, up to a subsequence,~$\{ \bar{x}_j \}$ converges to some~$\bar{x} \in \R$. By virtue of this, we easily deduce that~$\{ b_j \}$ must be bounded. Indeed, if~$b_{j_i} \rightarrow -\infty$ as~$i \rightarrow +\infty$ along some diverging sequence~$\{ j_i \}$, then from~\eqref{w>v} we get that~$v \le \lceil M \rceil - 1$ in~$\R$, in contradiction with the fact that~$M = \sup_{\R} v$. If, otherwise,~$b_{j_i} \rightarrow +\infty$, then~\eqref{w=v} gives that~$v(\bar{x}) = \lceil M \rceil$. As~$v(\bar{x}) \le \sup_\R v = M$ and~$\lceil M \rceil \ge M$, we infer that~$M = \lceil M \rceil$. Consequently,~$M \in \N$ and, using~\eqref{ellPNeqapp} and~\eqref{Wpropapp},
$$
\int_{\R} \frac{M - v(y)}{|\bar{x} - y|^{1 + 2 s}} \, dy = - W'(M) = 0.
$$
Since~$v \le M$ in~$\R$, this yields in turn that~$v \equiv M \ge 1$, contradicting the fact that~$\lim_{x \rightarrow -\infty} v(x) = 0$.

The sequence~$\{ b_j \}$ is thus bounded. Up to a subsequence, it then converges to some~$\bar{b} \in \R$. As both~\eqref{w=v} and~\eqref{w>v} pass to the limit, we infer that~$w(\bar{x} + \bar{b}) + k = v(\bar{x})$ and~$w(y + \bar{b}) + k \ge v(y)$ for all~$y \in \R$. Accordingly, using once again that~$w$ and~$v$ satisfy~\eqref{ellPNeqapp}---or letting~$j \rightarrow +\infty$ in~\eqref{eqforvandw}---we obtain
$$
\int_{\R} \frac{w(y + \bar{b}) + k - v(y)}{|\bar{x} - y|^{1 + 2 s}} \, dy = 0.
$$
Therefore,~$v(y) = w(y + \bar{b}) + k$ for all~$y \in \R$. As this means that~$v$ takes the form~\eqref{v=w+k}, the proof of the proposition is complete.
\end{proof}

\end{document}